\documentclass[11pt,a4paper,reqno]{amsart}
\numberwithin{equation}{section}

\usepackage[T1]{fontenc}
\usepackage{lmodern}
\usepackage[american]{babel}
\usepackage[final]{microtype}
\UseMicrotypeSet[protrusion]{basicmath}
\microtypecontext{spacing=nonfrench}

\usepackage{amssymb}
\usepackage{mathtools}
\usepackage{mathrsfs}
\usepackage{stmaryrd}
\usepackage{bm}

\mathtoolsset{centercolon}

\usepackage{geometry}
\usepackage{enumitem}
\setlist{itemsep=1pt, topsep=1pt, parsep=1pt}

\usepackage[dvipsnames]{xcolor}
\usepackage[symbol,perpage]{footmisc}
\usepackage{caption}

\usepackage{hyperref}
\hypersetup{
  colorlinks,
  linkcolor={red!50!black},
  citecolor={blue!50!black},
  urlcolor={blue!80!black}
}

\usepackage{aliascnt}
\usepackage{cleveref}

\newcommand{\newaliastheorem}[3]{%
  \newaliascnt{#1}{#2}%
  \newtheorem{#1}[#1]{#3}%
  \aliascntresetthe{#1}%
}

\newcommand{\setcrefnames}[5]{%
  \crefname{#1}{#2}{#3}%
  \Crefname{#1}{#4}{#5}%
}

\theoremstyle{plain}
\newtheorem{theorem}{Theorem}[section]
\newaliastheorem{lemma}{theorem}{Lemma}
\newaliastheorem{claim}{theorem}{Claim}
\newaliastheorem{proposition}{theorem}{Proposition}
\newaliastheorem{corollary}{theorem}{Corollary}
\newaliastheorem{conjecture}{theorem}{Conjecture}

\theoremstyle{definition}
\newaliastheorem{definition}{theorem}{Definition}
\newaliastheorem{example}{theorem}{Example}

\theoremstyle{remark}
\newaliastheorem{remark}{theorem}{Remark}

\setcrefnames{theorem}{theorem}{theorems}{Theorem}{Theorems}
\setcrefnames{lemma}{lemma}{lemmas}{Lemma}{Lemmas}
\setcrefnames{claim}{claim}{claims}{Claim}{Claims}
\setcrefnames{proposition}{proposition}{propositions}{Proposition}{Propositions}
\setcrefnames{corollary}{corollary}{corollaries}{Corollary}{Corollaries}
\setcrefnames{conjecture}{conjecture}{conjectures}{Conjecture}{Conjectures}
\setcrefnames{definition}{definition}{definitions}{Definition}{Definitions}
\setcrefnames{example}{example}{examples}{Example}{Examples}
\setcrefnames{remark}{remark}{remarks}{Remark}{Remarks}

\DeclarePairedDelimiter\abs{\lvert}{\rvert}
\DeclarePairedDelimiter\norm{\lVert}{\rVert}

\DeclarePairedDelimiter\set{\{}{\}}

\newcommand*{\indic}{\bm{1}}
\newcommand*{\ind}[1]{\bm{1}_{\set{#1}}}

\newcommand*{\Var}{\mathbb{V}\mathrm{ar}}
\newcommand*{\Cov}{\mathbb{C}\mathrm{ov}}
\newcommand*{\de}{\mathop{}\!\mathrm{d}}
\newcommand*{\dd}{\de}
\newcommand*{\disc}{\mathrm{disc}}

\newcommand*{\tf}{\mathtt{F}}
\newcommand*{\oneover}[1]{\frac{1}{#1}}

\newcommand*{\EE}{\mathbb{E}}
\newcommand*{\E}{\mathbf{E}}

\let\P\relax
\newcommand*{\P}{\mathbf{P}}
\newcommand*{\PP}{\mathbb{P}}

\newcommand*{\tEE}{\widetilde{\EE}}
\newcommand*{\tPP}{\widetilde{\PP}}
\newcommand*{\tVV}{\widetilde{\mathbb{V}}\mathrm{ar}}
\newcommand*{\tC}{\widetilde{\mathbb{C}}\mathrm{ov}\,}

\newcommand*{\N}{\mathbb{N}}
\newcommand*{\Z}{\mathbb{Z}}
\newcommand*{\R}{\mathbb{R}}

\newcommand{\e}{\mathrm{e}}
\newcommand{\w}{\omega}

\renewcommand*{\th}{\vartheta}

\newcommand*{\eps}{\varepsilon}
\renewcommand{\phi}{\varphi}

\newcommand*{\llb}{\llbracket}
\newcommand*{\rrb}{\rrbracket}

\newcommand{\cF}{\mathcal{F}}
\newcommand{\cH}{\mathcal{H}}
\newcommand{\cM}{\mathcal{M}}

\newcommand{\cV}{\mathcal{V}}
\newcommand{\cY}{\mathcal{Y}}

\newcommand{\tN}{\tilde{N}}

\newcommand*{\Tm}{T^{(m)}}
\newcommand*{\tauk}{\tau^{(m)}_k}

\newcommand*{\sumtwo}[2]{\sum_{\substack{#1\\#2}}}

\usepackage{scalerel,stackengine}
\newcommand\pig[1]{\scalerel*[5pt]{\big#1}{%
\ensurestackMath{\addstackgap[1.5pt]{\big#1}}}}
\newcommand\pigl[1]{\mathopen{\pig{#1}}}
\newcommand\pigr[1]{\mathclose{\pig{#1}}}

\title[Strong Disorder for Stochastic Heat Flow and 2D Directed Polymers]{Strong Disorder for Stochastic Heat Flow and\\2D Directed Polymers}
\author{Quentin Berger}
\address{Université Sorbonne Paris Nord, Laboratoire d'Analyse, Géométrie et Applications, CNRS UMR 7539, 99 Av. J-B Clément, 93436 Villetaneuse, France and Institut Universitaire de France}
\email{quentin.berger@math.univ-paris13.fr}
\author{Francesco Caravenna}
\address{Dipartimento di Matematica e Applicazioni, Università degli Studi di Milano-Bicocca, via Cozzi 55, 20125 Milano, Italy}
\email{francesco.caravenna@unimib.it}
\author{Nicola Turchi}
\address{Dipartimento di Matematica e Applicazioni, Università degli Studi di Milano-Bicocca, via Cozzi 55, 20125 Milano, Italy}
\email{nicola.turchi@unimib.it}
\date{\today}

\begin{document}

\begin{abstract}
  The critical 2D Stochastic Heat Flow (SHF) is a universal measure-valued process that provides a notion of solution to the ill-defined 2D stochastic heat equation. 
  We investigate the SHF in the large-time and strong-disorder regimes, proving a sharp form of \emph{local extinction}: we identify the rate at which the distribution collapses to zero.
  We also identify the spatial scale governing the transition from vanishing mass to diverging mass, and from extinction to an averaged behavior.
  Corresponding results are established for the partition functions of 2D directed polymers, yielding precise free-energy estimates. 
  Our proof provides a unified framework of change of measure and coarse-graining arguments.

  These results offer new insights into the 2D stochastic heat equation regularized via space-time discretization: for any regime of supercritical disorder strength~$\beta$, including the case where $\beta > 0$ is kept fixed, the solution exhibits fluctuations on a superdiffusive scale.
\end{abstract}

\maketitle

\section{Introduction and main results on the Stochastic Heat Flow}

The critical 2D \emph{Stochastic Heat Flow (SHF)} with disorder strength $\th \in \R$ is a stochastic process $\mathscr{Z}^\th = (\mathscr{Z}_{s,t}^{\th}(\dd x, \dd y))_{0 \le s \le t < \infty}$ of random measures on $\R^2 \times \R^2$.
It was introduced in~\cite{CSZ23} as the universal scaling limit of 2D directed polymer partition functions, recalled below, under a critical rescaling of the disorder strength.
It also arises as the limit of solutions to the 2D stochastic heat equation with mollified noise; see~\cite{Tsai24}, which provides an axiomatic definition. 

The fact that dimension two is the \emph{critical} spatial dimension for the stochastic heat equation and directed polymers makes the SHF especially interesting:
it is one of the few examples of a \emph{non-Gaussian scaling limit at the critical dimension and at the critical point}.
A brief overview of the literature on the SHF is presented in \Cref{sec:literature}. 
We refer to the lecture notes~\cite{CSZ24} for an extended discussion, as well as additional background and connections to singular SPDEs.

\subsection{Overview of our contribution}

We focus on the one-time marginal of the SHF:
\begin{equation*}
  \mathscr{Z}_{t}^{\th}(\dd x)
  \coloneqq\mathscr{Z}_{0,t}^{\th}(\R^2, \dd x)
\end{equation*}
which is a random measure on $\R^2$.
We investigate both the \emph{strong-disorder} regime $\th \to \infty$ and the \emph{large-time} regime $t \to \infty$, where a phenomenon of \emph{local extinction} occurs, in the sense that $\mathscr{Z}_{t}^{\th}(K) \to 0$ for any compact set $K \subset \R^2$.

Our main results identify the \emph{decay rate of the SHF distribution} (see \Cref{th:mainSHF}) as well as the \emph{growth rate of the spatial scale} at which a transition in the mass of the SHF occurs, from a regime of vanishing mass to a regime of diverging mass (see \Cref{cor:mainSHF}).
Both rates are shown to be exponential in time~$t$ and doubly exponential in the disorder strength~$\th$.

In \Cref{sec:DP-SHE} we present corresponding results for 2D directed polymers, which are of independent interest because they are valid across \emph{all regimes of disorder strength} (see \Cref{thm:quantitative}). 
This allows us to derive \emph{refined free-energy estimates} (see \Cref{thm:freeenergy}) which improve on the best available bounds in the literature \cite{Lac10a,BL17}.

We also discuss in \Cref{sec:SHE} the implications of our results for the 2D stochastic heat equation regularized through space-time discretization. 
We allow the disorder strength~$\beta$ to vary arbitrarily in the supercritical regime, including the case where \emph{$\beta > 0$ is kept fixed} as the regularization is removed. 
We show that the solution exhibits a transition from local extinction to an averaged behavior at an explicit \emph{superdiffusive scale} (see \Cref{th:mainSHE}). 
This identifies the regime where non-trivial fluctuations of the solution can be observed.

\smallskip

The strategy of our proof builds on the by-now classical approach based on \emph{change of measure} and \emph{coarse-graining}, introduced in the seminal works \cite{DGLT09,GLT10,GLT10b} and subsequently applied in various contexts. 
A key difficulty in successfully implementing this strategy lies in the choice of a suitable \emph{proxy} for the random variable of interest, in our case the partition function of 2D directed polymers, which must be tractable enough while remaining sufficiently close to the original partition function.

The main novelty of our approach is a ``canonical'' recipe for constructing such a proxy from a \emph{coarse-grained chaos expansion}.
The main ideas are illustrated in \Cref{sec:strategy,sec:keyprop}, where we also develop \emph{change-of-scale arguments} of independent interest. 
We believe that our strategy is sufficiently robust and transparent to be broadly applicable, and we expect it to be useful in other contexts.

\subsection{A quick overview of the SHF literature}
\label{sec:literature}

Many features of the SHF have been investigated, in particular its moments. 
The second moment was first studied in~\cite{BC98} in the context of solutions to the 2D stochastic heat equation, exploiting a connection with the delta-Bose gas from~\cite{ABD95}; refined results, also in the setting of directed polymers, were later obtained in \cite{CSZ19-Dickman}. 
The third moment was obtained in \cite{CSZ19-3rd}, and all integer moments were later derived in \cite{GQT21}; see also \cite{Che24} for further connections with the delta-Bose gas.

The asymptotic analysis of moments is challenging, due to their intricate structure. 
Important progress has recently been made in \cite{GN25}, where a sharp lower bound on their growth rate was established through a novel connection between moments of the SHF and the Gaussian Free Field;
note that the moments grow at a doubly exponential rate, a feature that also appears in our results.
Let us also mention \cite{LZ24}, where small-scale asymptotics were derived, extending the approach developed by \cite{CZ23} in the sub-critical regime.

Concerning the properties of the SHF as a random measure, estimates on its singularity and regularity were obtained in \cite{CSZ25}. 
It was also proved in \cite{CSZ23-GMC} that the SHF is \emph{not a Gaussian Multiplicative Chaos (GMC)} on~$\R^2$ via comparison of moments.
Very recently, the SHF was shown in \cite{CT25} to enjoy a \emph{conditional GMC structure} on path space, which yields as corollaries the full support property (strict positivity), also obtained independently in \cite{N25pos}, and the local extinction of mass for strong disorder, discussed below.

Other features of the SHF include a Chapman--Kolmogorov property and the construction of associated polymer measures \cite{CM24}, continuity in time and the characterization already mentioned above \cite{Tsai24}. 
Let us also mention the black noise property \cite{GT25} and an enhanced noise sensitivity property for directed polymer partition functions \cite{CD25}, which both yield independence between SHF and white noise.
Recent progress on a martingale description of the SHF was obtained in \cite{N25mart,Chen25mart}. 

Most of these results concern the SHF at finite time horizon and at fixed disorder strength $\th \in \R$.
Some results are also available in the weak-disorder regime $\th \to -\infty$, such as Edwards-Wilkinson (Gaussian) fluctuations \cite[Theorem~1.2]{CCR25} and an asymptotic log-normality for small scales \cite[Theorem~1.2]{CSZ25}.
Corresponding results, and many others, hold for directed polymers and the stochastic heat equation in the sub-critical regime, for which we refer again to~\cite{CSZ24}.
Here, by contrast, we investigate the SHF in the large-time regime \(t\to\infty\) and in the strong-disorder regime \(\th\to+\infty\).

\subsection{Main results for the SHF}

The first moment $\EE[\mathscr{Z}_{t}^{\th}(\dd x)] = \dd x$ of the SHF is simply the Lebesgue measure on $\R^2$. 
In particular, using the functional notation
\begin{equation*}
  \mathscr{Z}_{t}^{\th}(\varphi) 
  \coloneqq \int_{\R^2} \varphi(x) \, \mathscr{Z}_{t}^{\th}(\dd x) \,,
\end{equation*}
we have $\EE[\mathscr{Z}_t^\th(\varphi)] = 1$ for any \(t>0\) and each probability density $\varphi$ on~$\R^2$ (we call $\varphi$ the \emph{initial condition}). 

It turns out that the second moment diverges for strong disorder: 
for any \(t>0\) and each probability density~$\varphi$
\begin{equation*}
  \lim_{\th\to\infty} \EE[\mathscr{Z}_t^\th(\varphi)^2] 
  = \infty \,.
\end{equation*}
In fact, higher moments diverge even faster. 
More precisely, size-biasing and Jensen's inequality yield (see \cite[Remark~1.14]{CSZ25}) that
\begin{equation*}
  \forall h > 2 \colon \qquad \frac{\EE[\mathscr{Z}_t^\th(\varphi)^h]}
  {\EE[\mathscr{Z}_t^\th(\varphi)^2]^{\frac{h}{2}}}\,
  \xrightarrow[\ \th\to\infty \ ]{}\, \infty \,,
\end{equation*}
which expresses a form of \emph{intermittency}. 
For each fixed $\th$, similar asymptotics hold as $t\to\infty$.

\subsubsection{Strong disorder and local extinction}

In view of the intermittent behavior described above, it is natural to expect that $\mathscr{Z}_t^\th(\varphi)$ vanishes for strong disorder or large time: 
\begin{equation} 
\label{eq:mainSHF0}
  \mathscr{Z}_t^\th(\varphi) 
  \xrightarrow[\quad]{} 0 \quad \text{in distribution as $t\to\infty$  or  $\th\to\infty$}\,.
\end{equation}
The large-time convergence was obtained in \cite{CSZ25}, while the strong disorder convergence was very recently proved in \cite{CT25}, as a consequence of a \emph{conditional GMC structure}.

We establish a quantitative version of this convergence, uniform throughout the strong-disorder and large-time regimes. 
We allow for varying initial conditions~$\varphi$ with possibly \emph{diverging support}, and we establish \emph{optimal bounds}, displaying an exponential decay rate in time~\(t\) and a doubly exponential decay rate in the disorder strength~$\th$.
Let us denote by $\cM_1(r)$ the set of probability densities with support in the ball of radius~$r$:
\begin{equation}
\label{eq:probab-density}
  \cM_1(r) 
  = \Bigl\{ \varphi \colon \R^2 \to [0,\infty)\Bigm| \varphi \text{ integrable}, \int_{\R^2} \varphi(x) \de x= 1\text{ and } \varphi(x) = 0 \text{ for } \abs{x} > r \Bigr\}.
\end{equation}
We can now state our first main result, which we prove in \Cref{sec:others}.

\begin{theorem}[Strong disorder and large time for the SHF]
\label{th:mainSHF}
  There exist universal constants $c_0, c_1, c_2 \in (0,\infty)$ such that, for any $t > 0$ and $\th\in\R$,
  \begin{equation}
  \label{eq:mainSHF}
    \frac{1}{c_1} \, \e^{ -c_1 \, t \, \e^{\th} } 
    \leq \sup_{\varphi \in \cM_1\bigl(\e^{c_0 \, t\, \e^{\th}} \sqrt{t  \, }\,\bigr)} \EE\bigl[ \mathscr{Z}_t^\th(\varphi) \wedge 1 \bigr]
    \leq \frac{1}{c_2}\, \e^{-c_2 \, t \, \e^{\th} } \,.
  \end{equation}
  The same bounds hold upon replacing $\EE\bigl[ \mathscr{Z}_t^\th(\varphi) \wedge 1 \bigr]$ by a fractional moment $\EE\bigl[ \mathscr{Z}_t^\th(\varphi)^{\gamma} \bigr]$ with $\gamma \in (0,1)$, for constants \(c_0,c_1,c_2\) depending on \(\gamma\).

  Correspondingly, we can bound the right tail probability of $\mathscr{Z}_t^\th(\varphi)$: 
  for any $\eps \in (0,1)$ there are constants $C_{1,\eps}, C_{2,\eps} \in (0,\infty)$ such that
  \begin{equation} 
  \label{eq:probab-shf}
    C_{1,\eps} \: \e^{ -c_1 \, t \, \e^{\th} } 
    \leq \sup_{\varphi \in \cM_1\bigl(\e^{c_0 \, t\, \e^{\th}} \sqrt{t  \, }\,\bigr)} \PP\bigl( \mathscr{Z}_t^\th(\varphi) \ge \eps \bigr) 
    \le C_{2,\eps} \: \e^{-c_2 \, t \, \e^{\th}} \,.
  \end{equation}
\end{theorem}

The core of \Cref{th:mainSHF} is the upper bound in \eqref{eq:mainSHF}, which we derive from a corresponding result for 2D directed polymers; see \Cref{thm:quantitative} below.
The upper bound in \eqref{eq:probab-shf} follows by Markov's inequality, since $\PP(Z \ge \eps) \le (\eps \wedge 1)^{-1} \, \EE[Z \wedge 1]$ for any random variable $Z \ge 0$, while the lower bounds in \eqref{eq:mainSHF} and \eqref{eq:probab-shf} are obtained via the second moment method.

\begin{remark}[Lower bounds and second moment]
  We prove the lower bounds in \eqref{eq:mainSHF} and \eqref{eq:probab-shf} by the Paley--Zygmund inequality coupled with a variance upper bound; see \Cref{prop:second-moment-DP} below. 
  In both \eqref{eq:mainSHF} and \eqref{eq:probab-shf}, the $\sup$ can be removed if we choose $\varphi = \mathcal{U}_{\sqrt{t}}$ to be uniform on the ball of radius~$\sqrt{t}$, where we set
  \begin{equation} 
  \label{eq:unif-dist}
    \mathcal{U}_r(x) 
    \coloneqq \frac{1}{\pi r^2} \, \indic_{B(0,r)}(x) \qquad \text{with} \qquad B(0,r) \coloneqq \{x\in\R^2 \colon \abs{x} \le r\} \,.
  \end{equation}
\end{remark}

\begin{remark}[Truncated mean vs. fractional moments] 
  The \emph{truncated mean} $\EE[\mathscr{Z}_t^\th(\varphi) \wedge 1]$ appearing in \eqref{eq:mainSHF} may also be written as $\PP(\mathscr{Z}_t^\th(\varphi) > U)$ with $U$ an independent uniform random variable on $(0,1)$. 
  This quantity can be compared with \emph{fractional moments} $\EE[\mathscr{Z}_t^\th(\varphi)^\gamma]$ with $\gamma \in (0,1)$ (see \Cref{lem:frac-moment} below) and it is also directly related to the \emph{total variation distance} between the probability measure $\PP$ and its \emph{size-biased version} with respect to $\mathscr{Z}_t^\th(\varphi)$ (see \Cref{rem:TVdistance}).
  In our proofs, we will use both the truncated mean and fractional moments, since each quantity has its own advantages and limitations; see \Cref{sec:keyprop} for a discussion.
\end{remark}

\begin{remark}[Scaling covariance, strong disorder and large time]
  The dependence of our bounds \eqref{eq:mainSHF}, \eqref{eq:probab-shf} on the parameters $t$ and~$\th$ agrees with the \emph{scaling covariance property} of the SHF \cite[Theorem~1.2]{CSZ23}, which states that for any $t, \th, \varphi$ we have the equality in distribution
  \[
    \forall a > 0 \colon \qquad \mathscr{Z}_{a t}^\th\bigl(\varphi_{\sqrt{a}}\bigr)
    \,\overset{\mathrm{d}}{=}\, \mathscr{Z}_t^{\th + \log a}(\varphi)\qquad \text{where we set } \varphi_{\sqrt{a}}(x) 
    \coloneqq \tfrac{1}{a} \, \varphi\pigl(\tfrac{x}{\sqrt{a}}\pigr) \,.
  \]
  This property connects strong-disorder and large-time regimes: 
  replacing $t$ by $t/a$ and setting $a = \e^{-\th}$ (resp.\ $a=t$) allows us to set $\th=0$ (resp.\  $t=1$), which yields
  \begin{equation}
  \label{eq:scaling-th-t}
    \mathscr{Z}_{t}^{\th}\bigl(\varphi_{\sqrt{\e^{-\th}}}\bigr)
    \,\overset{\mathrm{d}}{=}\,\mathscr{Z}_{t \, \e^\th}^{0}(\varphi) \,,\qquad\mathscr{Z}_{t}^{\th}(\varphi_{\sqrt{t}})
    \,\overset{\mathrm{d}}{=}\,
    \mathscr{Z}_{1}^{\th + \log t}(\varphi) \,.
  \end{equation}
  Since it was proved in \cite[Theorem~1.4]{CSZ25} that $\mathscr{Z}_{T}^{\th_0}(\varphi) \to 0$ in probability as $T \to \infty$ for fixed~$\th_0$, we could deduce from the first relation in \eqref{eq:scaling-th-t} that $\mathscr{Z}_{t}^{\th}\bigl(\varphi_{\sqrt{\e^{-\th}}}\bigr) \to 0$ as $\th \to \infty$. 
  We stress, however, that this consequence is \emph{much weaker} than \eqref{eq:mainSHF0}, and a fortiori much weaker than \eqref{eq:mainSHF}, \eqref{eq:probab-shf} and \eqref{eq:SHF-largeball}, because shrinking the support of the initial condition $\varphi$ helps convergence to zero\footnote{For instance, by  \cite[Theorem~1.1]{CSZ25}, we have $\mathscr{Z}_{t}^{\th}(\varphi_{\sqrt{a}}) \to 0$ in probability as $a\downarrow 0$ even \emph{for fixed~$t,\th$}.}.

  Similarly, from property \eqref{eq:mainSHF0} and the second relation in \eqref{eq:scaling-th-t}, we could deduce that $\mathscr{Z}_{t}^{\th}(\varphi_{\sqrt{t}}) \to 0$ in probability as $t \to \infty$ for fixed~$\th$. 
  However, our bounds \eqref{eq:mainSHF}, \eqref{eq:probab-shf} are much stronger, since the space scale is increased by a factor $\e^{c \, t \, \e^\th}$.
\end{remark}

\subsubsection{Transition for the mass of the SHF on large spatial scales}

From \Cref{th:mainSHF} we deduce the behavior of the mass of the SHF in balls $B(0,r)$ with large radius $r \to \infty$. 
Even though $\EE[\mathscr{Z}_t^\th(B(0,r))] = \pi \, r^2 \to \infty$,  a \emph{transition occurs on the spatial scale $r = \e^{c \, t \, \e^\th} \sqrt{t}$} as either $\th\to\infty$ or $t\to\infty$: 
with high probability, the mass \emph{vanishes} for small $c > 0$, while it \emph{diverges} for large~$c$. 
Our second main result, proved in \Cref{sec:others}, is the following:

\begin{theorem}[Transition for the SHF mass in large balls]
\label{cor:mainSHF}
  There are constants $\delta > 0$ and $0 < c' < c'' < \infty$ such that the following holds for any $t > 0$ and $\th\in\R$:
  \begin{equation} 
  \label{eq:SHF-largeball}
    \text{with probability at least $\,1 - \tfrac{1}{\delta} \, \e^{-\delta \, t \, \e^{\th}}$:} \qquad
    \begin{cases}
      \mathscr{Z}_t^\th\pigl(B\pigl(0,\e^{c' \, t  \, \e^{\th}} \sqrt{t \,} \, \pigr)\pigr)
      \le t \, \e^{-\delta \, t \, \e^{\th}} \,, \\
      \rule{0pt}{1.4em}
      \mathscr{Z}_t^\th\pigl(B\pigl(0,\e^{c'' \, t \, \e^{\th}} \sqrt{t \,} \, \pigr)\pigr)
      \ge t \, \e^{+\delta \, t \, \e^{\th}}  .
    \end{cases}
  \end{equation}
\end{theorem}
We prove the first line of \eqref{eq:SHF-largeball} by exploiting the upper bound in \eqref{eq:mainSHF}, while for the second line we use the second moment method with a variance bound from \Cref{prop:second-moment-DP}.
We expect this transition to be sharp, in the following sense.
\begin{conjecture}
\label{conj:SHF-mass}
  There exists some \(\tilde c >0\) such that we have the following convergence in distribution: as \(t\to\infty\) or \(\th\to\infty\),
  \begin{equation*}
    \mathscr{Z}_t^\th\pigl(B\pigl(0,\e^{c \, t \, \e^{\th}} \sqrt{t \,} \, \pigr)\pigr)
    \xrightarrow{\ \ \mathrm{d}\ \ }
    \begin{cases}
      0  & \text{ if } c <\tilde c \,,\\
      +\infty  & \text{ if } c >\tilde c \,.
    \end{cases}
  \end{equation*}
\end{conjecture}

Since \(\EE[\mathscr{Z}_t^\th(\dd x)] = \dd x\), it is natural to compare the mass of the SHF on a large spatial scale with the mass of the Lebesgue measure, thereby capturing the ``escape of mass to infinity'' of the SHF $\mathscr{Z}_t^\th(\dd x)$.
\Cref{cor:mainSHF} already shows that on scales \(\e^{c \, t \, \e^{\th}} \sqrt{t}\) with $c < c'$ the SHF mass vanishes.
On the other hand, by spatial ergodicity, we expect some averaging to occur on larger scales.
Let us thus define a spatially rescaled version of the SHF \(\mathscr{Z}_t^\th(\dd x)\):
\begin{equation}
\label{eq:SHF-rescaled}
  \hat{\mathscr{Z}}_t^{\th,c}(\dd x)
  \coloneqq\frac{\mathscr{Z}_t^\th\pigl( \dd \bigl( \e^{c \, t \, \e^\th} \sqrt{t} \, x \bigr) \pigr)}{\bigl(\e^{c \, t \, \e^\th} \sqrt{t}\,\bigr)^{2}} \qquad\text{for } c \in (0,\infty) \,,
\end{equation}
where the normalization ensures $\EE[\hat{\mathscr{Z}}_t^{\th,c}(\dd x)] = \dd x$. 
We then prove the following transition from extinction to averaged behavior;
the proof is given in \Cref{sec:others}.

\begin{theorem}[Supercritical rescaling of SHF]
\label{cor2:SHF}
  There are constants \(0<c'<c''<\infty\) (possibly different from those in \Cref{cor:mainSHF}) such that, for any fixed density $\varphi$ on \(\R^2\), we have as \(t\to+\infty\) or \(\th\to+\infty\), 
  \begin{equation}
  \label{eq:SHF-fixed}
    \hat{\mathscr{Z}}_t^{\th,c}(\varphi) 
    = \int_{\R^2} \varphi(x) \, \hat{\mathscr{Z}}_t^{\th,c}(\dd x) \,
    \xrightarrow{\ \ \mathrm{d}\ \ } \,
    \begin{cases}
        0 & \text{if } c < c' \,, \\[2pt]
        1 & \text{if } c > c'' \,.
    \end{cases}
  \end{equation}
\end{theorem}
As in \Cref{conj:SHF-mass}, we expect that this transition is sharp.

\begin{conjecture} 
\label{conj:SHF-fixed-point}
  There exists a critical constant \(\hat c \in (0,\infty)\)  (possibly different from \(\tilde{c}\) in \Cref{conj:SHF-mass}) such that the convergence in distribution \eqref{eq:SHF-fixed} still holds with $c',c''$ both replaced by $\hat c$.
\end{conjecture}

\begin{remark}
  Assuming that the conjecture is true, it would be interesting to investigate what happens at the critical value \(c=\hat c\), namely whether a non-deterministic limit can be extracted from the sequence \(\hat{\mathscr{Z}}_t^{\th,\hat c}(\varphi)\) as \(t\to\infty\) or \(\th\to\infty\).
\end{remark}

We may interpret \Cref{cor:mainSHF,cor2:SHF} as manifestations of \emph{intermittent} behavior of the SHF.
The mass of the SHF vanishes on scales \(\e^{c \, t \, \e^{\th}} \sqrt{t}\) with \(c<c'\), ensuring that there are typically no ``high peaks'' in the distribution at this scale; however, since \(\EE[\mathscr{Z}_t^\th(\dd x)] =\dd x\), high peaks may occur, but only with very small probability.
On the other hand, at larger spatial scales \(\e^{c \, t \, \e^{\th}} \sqrt{t}\) with \(c>c''\) an averaging behavior occurs, meaning that we have encountered sufficiently many of the (rare but high) peaks.
\Cref{cor:mainSHF,cor2:SHF} thus give information on the spatial scale at which the high peaks appear.

\subsection{Structure of the paper}
In \Cref{sec:DP-SHE} we present our main results for 2D directed polymers and the stochastic heat equation, see in particular \Cref{thm:quantitative,th:mainSHE}.

In \Cref{sec:strategy} we describe the proof of \Cref{thm:quantitative}: by coarse-graining and change-of-scale arguments, we reduce it to the key \Cref{prop:key}, which is the core of the paper.

\Cref{sec:keyprop} contains the main ideas in the proof of the key \Cref{prop:key}, which involve size-biasing and change-of-measure arguments. 
This leads to some explicit moment estimates, which are proved in \Cref{sec:technicalI,sec:technicalII}.

\Cref{sec:others} collects the proofs of the other main results: 
\Cref{th:mainSHF,cor:mainSHF,cor2:SHF} on the SHF, \Cref{th:mainSHE} on the stochastic heat equation and \Cref{thm:freeenergy} on directed polymers.

Further technical results (which follow well-established paths) are postponed to the appendices.

\subsection{Notation}

\(\N\) denotes the set of non-negative integers. 
For a point \(x=(x_1,x_2)\) in \(\R^2\) we let \(\abs{x}=\sqrt{x_1^2+x_2^2}\) and $\abs{x}_\infty = \max\{\abs{x_1},\abs{x_2}\}$ denote its Euclidean and $\ell^\infty$ norms. 
For technical convenience, in the proofs we will slightly modify the families $\mathcal{M}_1(r)$ and $\mathcal{M}_1^{\mathrm{disc}}(r)$, see \eqref{eq:probab-density} and \eqref{eq:disc-density}, replacing $\abs{\,\cdot\,}$ with $\abs{\,\cdot\,}_\infty$ in their definition; this only affects constants.

Given two positive sequences \((a_N)_{N\in\N}\) and \((b_N)_{N\in\N}\), we write \(a_N \sim b_N\) if \(\lim_{N\to\infty} a_N/b_N = 1\) and \(a_N\ll b_N\) if \(\lim_{N\to\infty} a_N/b_N = 0\), \(a_N\gg b_N\) if \(\lim_{N\to\infty} a_N/b_N = +\infty\). 
For \(M,N\in \N\) with \(M\le N\) we write \(\llb M, N \rrb\) for the set \(\{M, M+1,\ldots,N\}\).
When \(A\) is a set, we denote by \(\abs{A}\) its cardinality and by \(\indic_A\) its indicator function, meaning that \(\indic_A(x) = 1\) if \(x\in A\) and \(\indic_A(x) = 0\) otherwise.
We write \(u\wedge v\) for \(\min(u,v)\) and \(u\vee v\) for \(\max(u,v)\).

\subsection*{Acknowledgements}
We are grateful to Rongfeng Sun and Nikos Zygouras for their comments on a preliminary version of the manuscript, which helped us improve the presentation. 
We would also like to thank Giuseppe Cannizzaro and Martin Hairer for insightful discussions.

F.C. and N.T. acknowledge the support of INdAM/GNAMPA.
Q.B. acknowledges the support of Institut Universitaire de France and ANR Local (ANR-22-CE40-0012-02).

\section{Main results for 2D directed polymers and stochastic heat equation}
\label{sec:DP-SHE}
The SHF was obtained in~\cite{CSZ23} as the limit of diffusively rescaled 2D directed polymer partition-function measures in an appropriate \textit{critical window} of disorder strength. 
We derive \Cref{th:mainSHF} from a corresponding result for the 2D directed polymer model, which in fact holds throughout the \emph{supercritical regime}.
We then discuss some consequences for the 2D stochastic heat equation regularized by space-time discretization, again in the supercritical regime.

\subsection{Strong disorder for 2D directed polymers}
We begin by recalling the definition of the directed polymer model.
Let \(S=(S_n)_{n\in\N}\) be the simple (symmetric, nearest-neighbor) random walk on \(\Z^2\), and denote~\(\P_x\) its law when $S_0 =x \in \Z^2$; let \(\E_x\) denote the corresponding expectation.
We simply write \(\P,\E\) for \(\P_0,\E_0\).
Additionally, consider a collection \(\w=\bigl(\w(n,x)\bigr)_{n\in\N,\,x\in\Z^2}\) of i.i.d.\ random variables, independent of \(S\), with law denoted by \(\mathbb{P}\).
By a slight abuse of notation, we also write \(\w\) for a generic copy of the disorder variables \(\w_{n,x}\).
We assume that
\begin{equation}
\label{eq:disorder}
  \EE[\w]=0\,,\qquad\EE[\w^2]=1\,,\qquad
  \lambda(\beta)
  \coloneqq\log\EE[\e^{\beta \w}] <+\infty\text{ for all }-2<\beta<2\,.
    \footnote{The endpoint \(2\) is chosen so that \(\sigma^2(\beta)\) is well-defined for all \(\beta\in(0,1)\), since it involves \(\lambda(2\beta)\). 
    For our purposes, one could equally well work on any smaller interval \((0,\beta_0)\) with \(\beta_0<1\), since the main interest is in the small-\(\beta\) regime.}
\end{equation}

For \(N\in\N\) and \(\beta>0\), the point-to-plane (\(1+2\text{-dimensional}\)) directed polymer model is defined as the Gibbs measure with Hamiltonian (up to a sign) 
\[
  H_N^{\beta,\w}(S) 
  \coloneqq \sum_{n=1}^N (\beta\w(n,S_n)-\lambda(\beta)).
\]
We are interested in the \emph{point-to-plane partition function} started from $x\in \Z^2$, defined by
\begin{equation}
\label{eq:polymeas}
  Z_{N}^{\beta,\w}(x)
  \coloneqq\E_x\Bigl[\e^{H_N^{\beta,\w}(S)}\Bigr]\, .
\end{equation}
We view \((Z_{N}^{\beta,\w}(x))_{x\in \Z^2}\) as a random field: for a function $f \in \ell^1(\Z^2)$ we define
\begin{equation}
\label{eq:Zf}
  Z_{N}^{\beta,\w}(f) 
  \coloneqq \sum_{x\in \mathbb{Z}^2} f(x)\, Z_{N}^{\beta,\w}(x) \,,
\end{equation} 
which is the integral of $f$ with respect to the random measure \( \sum_{x\in \Z^2} Z_{N}^{\beta,\w}(x)\, \delta_x\). 
The main result of~\cite{CSZ23} shows that this measure, diffusively rescaled, converges to a unique limit, which they named Critical 2D Stochastic Heat Flow (SHF), provided the disorder strength $\beta$ is rescaled in the so-called \textit{critical window}, which we now define.

Recalling \eqref{eq:disorder}, we define for $\beta \ge 0$ and $N\in\N$ the key quantities
\begin{equation}
\label{def:sigmaRN}
  \sigma^2(\beta) 
  \coloneqq \Var\bigl[\e^{\beta \w -\lambda(\beta)} \bigr] 
  = \e^{\lambda(2\beta) -2\lambda(\beta)} -1 \,,\qquad R_N 
  \coloneqq \sum_{n=1}^{N} \P(S_{2n}=0) \,.
\end{equation}
Note that \(\sigma^2(\beta) \sim \beta^2\) as \(\beta\downarrow 0\) and we can write, see \cite[Proposition~3.2]{CSZ19-Dickman},
\begin{equation}
\label{eq:RN}
  R_N 
  = \frac{1}{\pi} \bigl( \log N + \alpha_N \bigr) \qquad \text{with} \qquad
  \lim_{N\to\infty} \alpha_N \,
  =\,\alpha \,
  \coloneqq\, 4 \log 2 + \gamma - \pi\,\simeq\, 0.208 \,,
\end{equation}
where \(\gamma \coloneqq -\int_0^\infty \e^{-x} \, \log x \, \dd x \simeq 0.577\) is the Euler--Mascheroni constant.
Then the critical window corresponds to taking
\begin{equation}
\label{eq:critical-regime}
  \beta 
  = \beta_N(\th) \downarrow 0 \qquad\text{such that} \qquad \sigma^2(\beta) 
  = \frac{1}{R_N} \Bigl(1 + \frac{\th + o(1)}{\log N}\Bigr) \,,
\end{equation}
where \(\th \in \R\) is a fixed parameter, called the \textit{disorder strength} in the critical regime.

\smallskip

We now recall the main result of \cite{CSZ23}.
To match the random walk periodicity, we set
\begin{equation}
\label{eq:Zeven}
  \mathbb{Z}^d_{\mathrm{even}} 
  \coloneqq \{ x = (x^1,\ldots, x^d) \in \Z^d \colon \ x^1 + \ldots + x^d \text{ is even} \} \,.
\end{equation}
Given an integrable function $\varphi\colon\R^2 \to \R$, we define its rescaled version $\varphi^{(N)}\colon\mathbb{Z}^2_{\mathrm{even}} \to \R$
by
\begin{equation} 
\label{eq:varphiresc}
  \varphi^{(N)}(x) 
  \coloneqq \frac{N}{2} \int_{\abs{y - \frac{x}{\sqrt{N}}}_1 \le \frac{1}{\sqrt{N}}} \varphi(y) \, \dd y
\end{equation}
where $\abs{\,\cdot\,}_1$ denotes the $\ell^1$ norm. 
Then, for every \(t>0\), every \(\th \in \R\) and every integrable function $\varphi\colon\R^2 \to \R$, one has
\begin{equation}
\label{eq:convergence-polymer}
    Z_{\lfloor N t \rfloor}^{\beta_N,\w}(\varphi^{(N)}) 
    \xrightarrow[N\to\infty\,]{\mathrm{d}} \mathscr{Z}_{t}^{\th}(\varphi) \,,
\end{equation}
where \(\beta_N = \beta_N(\th)\) is chosen as in \eqref{eq:critical-regime}.

In this article, we go \emph{beyond the critical window}: we derive estimates which hold for \emph{arbitrary} $\beta \in (0,1)$ and $N \in \N$.
This means that $\th$ in \eqref{eq:critical-regime} need not be fixed, but may vary with~$N$ and~$\beta$. 
For this purpose, we refine the correspondence \eqref{eq:critical-regime} as follows:
\begin{equation} 
\label{eq:thetasharp}
  \sigma^2(\beta) 
  = \frac{1}{R_N} \, \Bigl( 1-\frac{\th + o(1)}{\pi R_N} \Bigr)^{-1}
  =  \frac{1}{R_N-\frac{\th + o(1)}{\pi}} \,.
\end{equation} 
Note that \eqref{eq:thetasharp} and \eqref{eq:critical-regime} are equivalent for any fixed \(\th \in \R\), in view of \eqref{eq:RN}.
More generally, they remain equivalent whenever \(\abs{\th} =o( \sqrt{\log N})\), but \emph{not} when \(\abs{\th}\gtrsim\sqrt{\log N}\).

\begin{remark}
  For $\beta \to 0$ and $N \to \infty$, we can rewrite \eqref{eq:thetasharp} more suggestively as
  \[
    \sigma^2(\beta) 
    = \frac{1}{R_{\lfloor N / \e^\th \rfloor} + o(1)}\,.
  \]
\end{remark}

Given any $\beta > 0$ and $N\in\N$, we therefore quantify the disorder strength by a parameter \(\th = \th(N,\beta) \) that we extract from \eqref{eq:thetasharp}: 
ignoring the $o(1)$ term, we define explicitly
\begin{equation}
\label{eq:Ntheta0}
  \th(N,\beta)
  \coloneqq \pi R_N - \frac{\pi}{\sigma^2(\beta)} \,.
\end{equation}
For $\beta = \beta_N(\th)$ in the critical window \eqref{eq:critical-regime} we have $\th(N,\beta) \to \th$ as $N\to\infty$, while any asymptotic regime \((N,\beta_N)\) above the critical window corresponds to $\th(N,\beta) \to \infty$.

We introduce a discrete analogue of~\eqref{eq:probab-density} for mass functions: 
\begin{equation} 
\label{eq:disc-density}
  \cM^{\disc}_1(r) 
  = \Bigl\{ f \colon \Z^2 \to [0,1] \Bigm| \sum_{z \in \Z^2} f(z) = 1 \text{ and } f(z) = 0 \ \text{ for } \abs{z} > r\Bigr\} \,.
\end{equation}
We can now state our main result for 2D directed polymers. 
The strategy of the proof is presented in \Cref{sec:strategy}, while the details are developed in the subsequent sections.

\begin{theorem}[Strong disorder for 2D directed polymers]
\label{thm:quantitative}
  There are constants \(c_0, c_1, c_2 \in (0,\infty)\) such that, uniformly over \(N\in\N\) and \(\beta \in (0,1)\), one has
  \begin{equation}
  \label{eq:upper-main}
    \frac{1}{c_1}\, \exp\Bigl(-c_1 \, {\e^{\th(N,\beta)}} \Bigr) 
    \leq \sup_{f \in \cM^{\disc}_1\bigl( \e^{c_0 \, \e^{\th(N,\beta)}} \sqrt{ N \, } \, \bigr)} \EE\bigl[ Z_{N}^{\beta,\w}(f) \wedge 1 \bigr] 
    \leq \frac{1}{c_2}\, \exp\Bigl(-c_2 \, {\e^{\th(N,\beta)}} \Bigr)  \,,
  \end{equation}
  where we define \( \th(N,\beta) \) as in~\eqref{eq:Ntheta0} (see also~\eqref{eq:thetasharp}). 
  Note that, in view of \eqref{eq:RN}, we have
  \begin{equation}
  \label{eq:etheta}
    \e^{\th(N,\beta)} \,
    =\,\e^{\alpha_N} N \, \e^{-\frac{\pi}{\sigma^2(\beta)}} \,
    =\, (1+ o(1))\, \e^{\alpha} N \, \e^{-\frac{\pi}{\sigma^2(\beta)}}\qquad \text{as } N\to\infty \,.
  \end{equation}
  In particular, if \((\beta_N)_{N\in\N}\subset(0,1)\) satisfies \(\th(N,\beta_N)\to \infty\), then for every choice of mass functions $f_N \in \cM^{\disc}_1\bigl( \e^{c_0 \, \e^{\th(N,\beta_N)}} \sqrt{ N \, } \, \bigr)$ we have convergence in probability $Z_N^{\beta_N,\w}(f_N) \to 0$.
\end{theorem}

We will deduce \Cref{th:mainSHF} for the SHF from \Cref{thm:quantitative} by taking $\beta = \beta_N(\th)$ in the critical regime \eqref{eq:critical-regime} so that $\th(N,\beta) \to \th$ (see Section~\ref{sec:others}). 
We stress, however, that \Cref{thm:quantitative} is \emph{stronger} since it allows for $\th(N,\beta) \to \infty$.
This also yields refined information on the \emph{free energy}, which we discuss in \Cref{sec:free-energy} below.

\subsection{On the supercritical 2D stochastic heat equation}
\label{sec:SHE}

The 2D Stochastic Heat Equation (SHE) is the singular stochastic PDE formally given for $t > 0$ and $x\in\R^2$ by
\begin{equation}
\label{eq:SHE}
  \begin{cases}
    \partial_t u(t,x) 
    = \tfrac{1}{2} \Delta u(t,x) + \beta \, \xi(t,x) \, u(t,x) \\[2pt]
    u(0,x) 
    = 1 \quad \text{(for simplicity)}
  \end{cases}
\end{equation}
where $\xi$ is space-time white noise and $\beta > 0$ tunes the disorder strength.
In mild formulation, this reads as
\begin{equation}
\label{eq:SHE-mild}
  u(t,x) 
  = 1 + \int_0^t \int_{\R^2} g_{t-s}(x-y) \, u(s,y) \,
  \beta \,  \xi(s,y)  \dd s \dd y 
\end{equation}
where \(g_t(x) = \frac{1}{2\pi t} \e^{- \frac{\abs{x}^2}{2t}}\) is the heat kernel.
Let us stress that this equation is ill-defined in dimension two and higher, since the solution $u$ is expected to be a genuine distribution, hence the singular product $u\cdot \xi$ has no clear meaning.

A natural way to remove the singularity is to suitably regularize the equation, so that a well-defined solution $u_1^\beta(t,x)$ exists (the subscript $1$ indicates the regularization scale). 
Many regularizations are possible including mollification in space, discretization in space or space-time and Fourier truncation.
We focus on space-time discretization, as in \cite{CSZ23}, turning the integrals in \eqref{eq:SHE-mild} into Riemann sums on the even lattice (see \eqref{eq:Zeven})
\begin{equation*}
  \mathbb{T}_1 
  \coloneqq (\N_0 \times \Z^2)_{\mathrm{even}} \,.
\end{equation*}
Replacing white noise by i.i.d. random variables and the heat kernel by the random walk transition kernel\footnote{
  More precisely, for $(t,x) \in \mathbb{T}_1$ we replace $\beta \, \xi(t,x)$ by $\e^{\beta\w(t,x) - \lambda(\beta)}-1$ and $g_t(x)$ by $q_t(x) \coloneqq \P(S_t = x)$.
}
, the solution $u_1^\beta(t,x)$ coincides with the \emph{partition function $Z_{[t]}^{\beta,\w_{[t]}}([x])$} from \eqref{eq:polymeas} with time-reversed environment~$\w_t$: 
upon piecewise constant extension, we have for $t \ge 0$, $x\in\R^2$,
\[
  u_1^\beta(t,x) 
  = Z_{[t]}^{\beta,\w_{{[t]}}}\bigl( [x] \bigr) \qquad \text{with} \quad \w_{m}
  = (\w_{m}(n,z) 
  \coloneqq \w(m-n, z))_{n\in\N, z\in\Z^2} \,,
\]
where we denote by $([t],[x])$ the closest point in $\mathbb{T}_1$ to $(t,x) \in [0,\infty) \times \R^2$.

\smallskip

We are interested in the \emph{large-scale scaling properties of $u_1^\beta(t,x)$}. 
In view of the parabolic nature of the SHE, it is natural to consider for $N\in\N$ the diffusively rescaled solution
\begin{equation}
\label{eq:SHE-rescaled}
  u_N^\beta(t,x) 
  \coloneqq u_1^\beta(N t, \sqrt{N} x) 
  = Z_{[Nt]}^{\beta,\w_{[Nt]}}\bigl( [\sqrt{N} x] \bigr) \,,
\end{equation}
so that the regime $N\to\infty$ corresponds to \emph{zooming out in space and time}. 
Also note that $u_N^\beta(t,x)$ is the solution of the SHE \eqref{eq:SHE-mild} discretized on the finer lattice 
\begin{equation*}
  \mathbb{T}_N 
  \coloneqq \bigl( \tfrac{1}{N}\N_0 \times \tfrac{1}{\sqrt{N}} \Z^2\bigr)_{\mathrm{even}}
\end{equation*}
which approximates space-time $[0,\infty) \times \R^2$ as $N\to\infty$. 
For this reason, finding a \emph{non-deterministic limit of $u_N^\beta(t,x)$ as $N \to \infty$} provides a notion of solution to the ill-defined 2D SHE. 
This is the viewpoint taken in \cite{CSZ23}, where it was shown that \emph{for $\beta = \beta_N(\th)$ in the critical window \eqref{eq:critical-regime}} the solution $u_N^\beta(t,x)$ converges to the SHF~$\mathscr{Z}_t^\th(\dd x)$, see \eqref{eq:convergence-polymer}.

\smallskip

A natural question is the behavior of $u_N^\beta(t,x)$ \emph{beyond the critical window}, that is for disorder strength $\beta = \beta_N$ such that $\th(N,\beta) \to \infty$, see \eqref{eq:Ntheta0}. 
We call this range of $(\beta,N)$ the \emph{supercritical regime}\footnote{The subcritical regime \emph{below the critical window} \eqref{eq:critical-regime}, corresponding to $\beta \ll \beta_N(\th)$, has been studied in depth in the literature, see e.g.~\cite{CSZ17b,CSZ20,CZ23,CZ24,CCR25,CD25,CNZ25}.}, which includes in particular the case when \emph{$\beta > 0$ is kept fixed as $N\to\infty$}. 
A direct consequence of \Cref{thm:quantitative} is a spatially averaged form of \emph{local extinction} of $u_N^\beta(t,x)$: 
for any $t > 0$ and any density $\varphi$, we have the convergence in distribution
\begin{equation}
\label{eq:SHE-extinction}
  \int_{\R^2} \varphi(x) \, u_N^\beta(t,x) \de x \,
  \xrightarrow[N\to\infty]{\mathrm{d}} \, 0
\end{equation}
in the supercritical regime. 
We now present quantitative refinements of this result.

Since rescaling space-time diffusively in \eqref{eq:SHE-rescaled} leads to the degenerate limit \eqref{eq:SHE-extinction}, we expect the solution $u_1^\beta(t,x)$ to display non-trivial fluctuations on a \emph{superdiffusive scale}. 
To capture this phenomenon, it is natural to modify the space rescaling in \eqref{eq:SHE-rescaled}, replacing $\sqrt{N}$ therein by $\sqrt{D_N  N}$ for a suitable diverging factor $D_N \to \infty$, with the aim of finding a non-trivial limit. 
This viewpoint was taken, for instance, in \cite{CMT25}, where the regularized 2D stochastic Burgers equation was shown to admit non-trivial fluctuations for the choice $D_N \propto (\log N)^{2/3}$. 

In our setting of the 2D SHE, it turns out that the space rescaling in \eqref{eq:SHE-rescaled} needs to be modified by an \emph{exponential factor}:
\begin{equation}
\label{eq:SHE-exp-rescaled}
  \begin{split}
    \hat{u}_{N}^{\beta,c}(t,x) 
    &\coloneqq u_1^\beta\pigl( N t \,, \, \rho_{Nt}^{\beta,c} \,\sqrt{ Nt\,}  \, x \pigr)\qquad \text{with} \quad \rho_{Nt}^{\beta,c} 
    \coloneqq \e^{c \, \e^{\th(\lfloor Nt\rfloor,\beta)}} , \quad c \in (0,\infty) \,.
  \end{split}
\end{equation}
In the supercritical regime $\th(N,\beta) \to \infty$ we have $\rho_{Nt}^{\beta,c} \to \infty$ for any $t>0$. 
Moreover, by \eqref{eq:thetasharp} we have 
\begin{equation}
\label{eq:thNt}
  \th(\lfloor Nt\rfloor,\beta)
  = \th(N,\beta) + \log t + o(1).
\end{equation}
Also note that, by \eqref{eq:etheta}, we can write
\begin{equation} 
\label{eq:exponential-scale}
  \rho_{Nt}^{\beta,c}
  = \e^{c \, \e^\alpha\, N t \, f_\beta (1+o(1))} \qquad \text{with} \quad f_\beta 
  \coloneqq \e^{-\frac{\pi}{\sigma^2(\beta)}} \,,
\end{equation}
which shows that, for fixed $\beta > 0$, the factor $\rho_{Nt}^{\beta,c}$ grows \emph{exponentially in $N$}  (see Remark~\ref{rem:stretched-exponential} for $\beta \downarrow 0$).

The following result, proved in \Cref{sec:others}, shows that the rescaled solution \(\hat{u}_{N}^{\beta,c}\) undergoes a transition from extinction to an averaged regime as \(c\) varies.

\begin{theorem}[Supercritical 2D stochastic heat equation]
\label{th:mainSHE}
  There are constants \(0 < c' < c'' < \infty\) such that the following holds: 
  if ($\beta_N)_{N\in\N} \subset (0,1)$ is in the supercritical regime \(\th(N,\beta_N) \to \infty\), then for any $t > 0$ and any fixed density $\varphi$ we have
  \begin{equation}
  \label{eq:mainSHE}
    \int_{\R^2} \varphi(x) \, \hat{u}_{N}^{\beta_N,c}(t,x) \, \dd x \,
    \xrightarrow[\ N\to\infty\ ]{\mathrm{d}} \,
    \begin{cases}
      0 & \text{if } c < c' \,, \\[2pt]
      1 & \text{if } c > c'' \,.
    \end{cases}
  \end{equation}
  This holds, in particular, if $\beta_N \equiv \beta \in (0,1)$ is kept fixed as $N\to\infty$. 
\end{theorem}

Let us discuss the significance of this result for the 2D SHE \eqref{eq:SHE}-\eqref{eq:SHE-mild} regularized by space-time discretization.
While the critical regime \eqref{eq:critical-regime} is the correct rescaling of $\beta$ to obtain the non-trivial SHF limit \cite{CSZ23}, it is natural to ask \emph{what happens beyond this critical window} (including the case when $\beta > 0$ is fixed).
\Cref{th:mainSHE} shows that non-trivial fluctuations of the solution can only be observed on the \emph{superdiffusive scale $\rho_{Nt}^{\beta,c}\sqrt{Nt}$ from \eqref{eq:SHE-exp-rescaled}}, for some $c \in (c',c'')$. 
We expect a similar phenomenon to occur for the SHE also in one spatial dimension, see Remark~\ref{rem:1D-SHE}.

Checking that the scale $\rho_{Nt}^{\beta,c}\sqrt{Nt}$ provides an \emph{upper bound} on the spatial fluctuation scale is not difficult: 
a variance computation yields the second line of \eqref{eq:mainSHE}, which shows that an averaged behavior takes place for suitable (large) $c>0$. 
However, in view of the intermittent nature of the solution, it is not at all clear whether the spatial scale identified by the variance is of the correct order. 
Our result shows that this is indeed the case, for suitable (small) $c>0$, thanks to a \emph{lower bound} on the scale of space fluctuations provided by the first line in \eqref{eq:mainSHE}.

\smallskip

Analogously to \Cref{conj:SHF-fixed-point}, we expect that the transition in \eqref{eq:mainSHE} is sharp.

\begin{conjecture} 
\label{conj:SHE-fixed-point}
  There exists a critical constant \(\hat c \in (0,\infty)\) such that the convergence in distribution \eqref{eq:mainSHE} still holds with both $c',c''$ replaced by $\hat c$. 
\end{conjecture}

We conclude with a few remarks.

\begin{remark}[Stretched exponential scale]
\label{rem:stretched-exponential}
  The rescaling factor $\rho_{N}^{\beta,c}$ in \eqref{eq:exponential-scale} is truly exponential in $N$ when $\beta > 0$ is kept fixed, while it is slower than exponential for vanishing~$\beta$. 
  A natural way to interpolate between the critical regime \eqref{eq:critical-regime} and the fixed $\beta > 0$ case is to consider
  \[
    \sigma^2(\beta) 
    = \frac{\hat\beta^2}{R_N} \qquad \text{for} \quad \hat\beta \in (1,\infty) \qquad \text{which implies} \qquad \beta 
    \sim \frac{\hat\beta}{\sqrt{R_N}}\,.
  \]
  This choice yields the \emph{stretched exponential scale}
  \begin{equation*}
    \rho_{N}^{\beta,c}
    = \e^{c \, (\e^\alpha N)^{1-\hat\beta^{-2}}}
  \end{equation*}
  which recovers the pure exponential scale as $\hat\beta \to \infty$.
\end{remark}

\begin{remark}[On the mollified 2D stochastic heat equation]
  We can also consider the 2D SHE regularized by mollification in space, whose solution is known to converge to the SHF in the critical window, see~\cite{Tsai24}.
  Adapting the techniques of the present paper, we can establish a version of \Cref{thm:quantitative,th:mainSHE} in this setting, which will be addressed in a forthcoming work.
\end{remark}

\begin{remark}[1D stochastic heat equation]
\label{rem:1D-SHE}
  In space dimension \(d=1\), the SHE \eqref{eq:SHE}-\eqref{eq:SHE-mild} has a well-defined solution $u^\beta(t,x)$ (with no need of regularization). 
  Defining, as in \eqref{eq:SHE-rescaled}, 
  \begin{equation*}
    u_N^\beta(t,x)
    \coloneqq u^{\beta}(Nt, \sqrt{N}x) \,,
  \end{equation*}
  one can check from \eqref{eq:SHE} that $u_N^\beta(t,x) \overset{d}{=} u^{\beta N^{1/4}}(t,x)$. 
  It follows that the \emph{critical regime} is $\beta \sim \hat\beta \, N^{-1/4}$, under which the law of $u_N^\beta$ remains invariant.
  
  In the \emph{supercritical regime} $\beta \gg \, N^{-1/4}$ we expect the same result as in \Cref{th:mainSHE} for an \emph{exponentially rescaled solution} $\hat{u}_{N}^{\beta,c}$, defined as in \eqref{eq:SHE-exp-rescaled}-\eqref{eq:exponential-scale} with $f_\beta$ replaced by~$\beta^4$:
  \[
    \hat{u}_{N}^{\beta,c}(t,x) 
    \coloneqq u^\beta\pigl( N t \,, \, \rho_{Nt}^{\beta,c} \sqrt{Nt\,}  \, x \pigr)\quad\text{with} \quad \rho_{Nt}^{\beta,c} 
    \coloneqq \e^{c \, Nt \, \beta^4}, \quad c \in (0,\infty) .
  \]
  We can give a heuristic explanation for the need of an exponential rescaling as follows. 
  Let us assume the expected KPZ one-point and process asymptotics of $\log u^\beta(t, x)$, which may be written as (we set $\beta = 1$ for simplicity): 
  for some constants $\text{\scshape F}, a, b > 0$ 
  \begin{equation} 
  \label{eq:Airy}
    u^{1}(t, x) \,
    \overset{d}{=}\, \exp\pigl(- \text{\scshape F}\, t \,+\, a \, t^{1/3} \mathcal{A}_1\bigl( \tfrac{x}{b \, t^{2/3}} \bigr) \,+\, o(t^{1/3}) \pigr) \qquad \text{as }t\to\infty \,,
  \end{equation}
  where $\mathcal{A}_1(\cdot)$ is the Airy process. 
  This process has an upper tail $\PP(\mathcal{A}_1(z) > u) \approx \exp(- c \, u^{3/2})$ and good spatial mixing properties, which yield slow spatially growing extrema: $\max_{\abs{z}\le R} \mathcal{A}_1(z) \approx (\log R)^{2/3}$. 
  This implies that we need \emph{exponentially large $\abs{x} \approx \exp(C t)$} for the second term $t^{1/3} \mathcal{A}_1\bigl( \tfrac{x}{b \, t^{2/3}} \bigr)$ in \eqref{eq:Airy} to overcome the leading term $- \text{\scshape F}\, t$, in order to prevent $u^{1}(t, x)$ from vanishing as $t\to\infty$.
\end{remark}

\subsection{Further results for directed polymers}
\label{sec:free-energy}

In the space dimension $d=2$, the point-to-plane partition function $Z_N^{\beta,\w}\coloneqq Z_N^{\beta,\w}(0)$ converges a.s. to \(0\) as \(N\to\infty\) \emph{for any fixed disorder strength $\beta > 0$}, and it does so exponentially fast, as shown in \cite{Lac10a}.
Its exponential decay rate to \(0\) is called (up to a sign) the free energy (or pressure) and is defined as 
\begin{equation}
\label{eq:def-free-energy}
  \tf(\beta) 
  \coloneqq \lim_{N\to\infty} \frac1N \log Z_{N}^{\beta,\w}  
  = \lim_{N\to\infty} \frac1N \, \EE \bigl[ \log Z_{N}^{\beta,\w} \bigr] \in (-\infty, 0]\,,
\end{equation}
where the limit is known to exist a.s. and in \(L^1(\PP)\), see e.g. \cite[Thm.~2.1]{Com17}. 
We point out that the free energy is related to localization properties of the polymer, see e.g.\ \cite{CH06,CSY03}. 

It was shown in~\cite{Lac10a} that \(\tf(\beta)<0\) for any \(\beta>0\) with some explicit bounds; 
a few years later, \cite{BL17} refined the bounds and showed that 
\[
  \tf(\beta) 
  =- \exp\Bigl( - (1+o(1)) \frac{\pi}{\beta^2} \Bigr) \quad  \text{ as } \beta\downarrow 0 \,.
\]
Our next result substantially improves these bounds: 
we identify the \emph{exact exponential decay rate} as \(\pi/\sigma^2(\beta)\), rather than simply\footnote{Note that $\lambda(\beta) = \frac{1}{2} \beta^2 + \frac{\kappa_3}{3!} \beta^3 + \frac{\kappa_4}{4!} \beta^4 + O(\beta^5)$ as $\beta \downarrow 0$, where $\kappa_3, \kappa_4$ are the third and fourth cumulants of the disorder distribution. It follows that we have \(\e^{-\pi/\beta^2} \sim (cst.)\, \e^{-\pi/\sigma^2(\beta)} \) only when $\kappa_3 = 0$.} \(\pi/\beta^2\), and we ``bring the $o(1)$ out of the exponential''.

\begin{theorem}[Improved free energy bounds]
\label{thm:freeenergy}
  There are constants $c,c' \in (0,\infty)$ such that
  \begin{equation}
  \label{eq:free-energy-bounds}
    \forall \beta \in (0,1) \colon \qquad - \frac{c'}{\sigma^2(\beta)^{4}} \, \exp\Bigl( - \frac{\pi}{\sigma^2(\beta)} \Bigr) 
    \leq  \tf(\beta) 
    \leq - c  \, \exp\Bigl( - \frac{\pi}{\sigma^2(\beta)} \Bigr) 
  \end{equation}  
  where we recall that \(\sigma^2(\beta) = \e^{\lambda(2\beta) -2\lambda(\beta)}-1\).
\end{theorem}

The upper bound in \eqref{eq:free-energy-bounds} is the main novelty: 
we deduce it from \Cref{thm:quantitative}, more precisely from the upper bound in \eqref{eq:upper-main}. 
We refer to \Cref{sec:others} for the proof.
The lower bound in \eqref{eq:free-energy-bounds} follows closely the strategy of \cite[\S4]{BL17} based on superadditivity and concentration arguments for \(\log Z_N^{\beta,\w}\); 
a few new estimates are needed, see \Cref{app:lower-free-energy} for details.

\begin{remark}
  We do not expect the lower bound in \eqref{eq:free-energy-bounds} to be optimal (the prefactor $\sigma^2(\beta)^{-4}$ is due to a limitation of the current techniques), however we believe the upper bound to be sharp, in particular,
  \[
    \tf(\beta) 
    \sim  - c\, \e^{- \frac{\pi}{\sigma^2(\beta)}} \quad \text{as } \beta\downarrow 0 \,.
  \]
\end{remark}

Our last result concerns \emph{variance estimates} for the directed polymer partition function, which will be used in Section~\ref{sec:strategy} to prove the lower bound in \eqref{eq:upper-main}. 
Let us focus on the initial condition given by the uniform distribution on the discrete ball of radius $\rho \sqrt{N}$ for some $\rho > 0$, namely
\begin{equation} 
  \mathcal{U}^{\mathrm{disc}}_{\rho \sqrt{N}}(x) 
  \coloneqq \frac{1}{\abs[\big]{B(0,\rho \sqrt{N}) \cap \Z^2}} \, \indic_{B(0,\rho \sqrt{N}) \cap \Z^2}(x) \,.
\end{equation}
The proof is given in \Cref{sec:Variances}.

\begin{proposition}[Variance estimates for directed polymers]
\label{prop:second-moment-DP}
  Define \( \th(N,\beta) \) as in~\eqref{eq:Ntheta0}.
  There is a constant \(c_3>0\) such that, uniformly over \(N\in \N\), \(\beta \in (0,1)\) and $\rho \in (0,\infty)$, we have
  \begin{equation}
  \label{eq:ub2mom-unif}
    \Var\bigl[Z_N^{\beta, \w}(\mathcal{U}^{\mathrm{disc}}_{\rho\sqrt{N}}) \bigr]   
    \leq c_3\, \frac{\exp\bigl(c_3 \, \e^{\th(N,\beta)} \bigr)}{\rho^2} \,.
  \end{equation}
  This can be sharpened for \(N\to\infty\), \(\beta\downarrow 0\) such that \(\th(N,\beta) \to \infty\): in this regime, uniformly over $\rho \in (0,\infty)$, we have
  \begin{equation} 
  \label{eq:ub2mom}
    \Var\bigl[Z_N^{\beta, \w}(\mathcal{U}^{\mathrm{disc}}_{\rho \sqrt{N}}) \bigr] 
    \le \frac{\exp\bigl((\e^{- \gamma}+o(1)) \, \e^{\th(N,\beta)} \bigr)}{\rho^2}.
  \end{equation}
\end{proposition}

\begin{remark}
  We believe the upper bound \eqref{eq:ub2mom} to be sharp.
  In other words, for (say) $\rho = 1$, one should also be able to prove that as \(N\to\infty\), \(\beta\downarrow 0\) with \(\th(N,\beta) \to \infty\) we have
  \[
    \Var\bigl[Z_N^{\beta, \w}(\mathcal{U}^{\mathrm{disc}}_{\sqrt{N}}) \bigr] 
    \geq \exp\bigl((\e^{-\gamma}+o(1)) \, \e^{\th(N,\beta)} \bigr) \,.
  \]
  Since this lower bound is not needed for our purposes, we omit the proof.
  One could also try to improve these estimates by ``bringing the $o(1)$ out of the exponential''.
\end{remark}

\section{Strategy of the proof of \texorpdfstring{\Cref{thm:quantitative}}{the main theorem}}
\label{sec:strategy}

In this section, we present the strategy of the proof of \Cref{thm:quantitative}. 
We first discuss the upper bound in \eqref{eq:upper-main}, which is the core of the paper, see \Cref{sec:steps}.
We then prove the lower bound by exploiting the variance estimates in \Cref{prop:second-moment-DP}, see \Cref{sec:lower-bound}.

Although we stated our main results for $\beta \in (0,1)$, in the rest of the paper we only work with $\beta$ small enough, say $\beta \in (0,\beta_0)$.
The extension to $\beta \in [\beta_0,1)$ then follows, after adjusting the constants, from monotonicity in $\beta$ of the truncated mean and of the fractional moments.

\subsection{Strategy of the upper bound in \texorpdfstring{\Cref{thm:quantitative}}{the main theorem}}
\label{sec:steps}

Our proof for the upper bound in \Cref{thm:quantitative} follows three main steps.
\begin{enumerate}
  \item First, we formulate a key result, \Cref{prop:key}, which is weaker than \Cref{thm:quantitative} in several respects: 
  (i)~the starting point is on the basic diffusive scale~\(\sqrt{N}\) without the factor $\e^{c_0 \, \e^\th}$ in \eqref{eq:upper-main};
  (ii)~the parameter \(\th \in [3,\infty)\) is fixed, so \Cref{prop:key} is proved only in the critical window;
  (iii)~the bound we obtain is polynomial in~\(\th\) instead of doubly exponential.
  \item Second, we use \emph{coarse-graining techniques} to extend the previous bound to any $\beta > 0$ (small) and \(N\in \N\), and at the same time obtain an exponential decay in~\(N\), see \Cref{prop:key+}. 
  By now, this is a well-established method, which dates back to~\cite{Lac10a} (in the context of directed polymers); technical details are postponed to \Cref{sec:coarse}.
  \item Third, we use a change-of-scale argument to quantify the effect of enlarging the scale of the starting point: this leads to \Cref{prop:key++}, which completes the proof of the upper bound in \Cref{thm:quantitative}.
\end{enumerate}
Thus, the proof of the upper bound in \Cref{thm:quantitative} reduces to \Cref{prop:key}, whose proof is given in \Cref{sec:keyprop}.
Let us now give some details on the steps outlined above.

\subsection*{Step 1: Key suboptimal result}

We first state a weaker version of \Cref{thm:quantitative}, in the form of the next key proposition. 
This is actually the core of the paper; we will present the key ideas of its proof in \Cref{sec:keyprop}. 

\begin{proposition}[Key proposition]
\label{prop:key}
  There is a universal constant \(C>0\) such that, for any given \(\th \in [3,\infty)\), if we consider any $\beta_N = \beta_N(\th)$ in the critical regime \eqref{eq:thetasharp}, we have
  \begin{equation}
  \label{eq:key-eq}
    \limsup_{N\to\infty} \sup_{f \in \cM_1^{\disc}(\sqrt{N})} \EE\bigl[Z_{N}^{\beta_N,\w}(f) \wedge 1\bigr] 
    \leq \frac{C}{\th}\,.
  \end{equation}
\end{proposition}
\begin{remark}
  The proof of this proposition relies on a change of measure argument, which we detail in \Cref{sec:keyprop}. 
  An important step in this argument is the choice of an event \(A_N\) which is atypical under \(\PP\) but typical under the size-biased measure \(\tPP\) defined in \eqref{eq:size-biased} below (see for instance \Cref{prop:change} and \Cref{prop:key2}).
  This is the only step at which the method becomes inherently model-specific (see \cite{JL24a,JL25} for the case of transient dimensions).
  However, we propose a ``canonical'' way of choosing the event \(A_N\), we refer to \Cref{sec:proxy} for details.
\end{remark}

We view \Cref{prop:key} as an estimate in the (upper) critical window, that is, we will use it for~\(\th\) large but fixed, in order to apply a ``finite-volume criterion'' in the next step.
Let us now deduce from~\eqref{eq:key-eq} a corresponding bound on a fractional moment of \(Z_N^{\beta_N,\w}(f)\): 
\begin{equation}
  \label{eq:key-eq-moments}
  \limsup_{N\to\infty} 
  \sup_{f \in \cM_1^{\disc}(\sqrt{N})} \EE\Bigl[Z_{N}^{\beta_N,\w}(f)^{1/2}\Bigr] \leq \frac{\sqrt{2
  C}}{\sqrt{\th}}\,,
\end{equation}
thanks to the following general result.

\begin{lemma}
  \label{lem:frac-moment}
For any random variable $Z \ge 0$ with $\EE[Z] = 1$ and any $\gamma \in (0,1)$ we have
\[
  \EE[Z\wedge 1] \leq \EE[Z^{\gamma}]\le 2^{\gamma} \, \EE[Z\wedge 1]^{\gamma \wedge (1-\gamma)} \,.
\]
In particular
\[
  \EE[Z\wedge 1] \leq \EE[Z^{1/2}]\le\sqrt{2} \, \EE[Z\wedge 1]^{1/2} \,.
\]
\end{lemma}

\begin{proof}
The first inequality simply uses that \(x\leq x^{\gamma}\) for \(x\in [0,1]\) to get that \(\EE[Z\wedge 1] \leq \EE[ (Z\wedge 1)^{\gamma}] \leq \EE[Z^{\gamma}]\).
For the second one, consider first $\gamma \le \frac{1}{2}$.
Using the identity \(Z=(Z\wedge 1)(Z\vee 1)\), we have
\[
  \EE[Z^\gamma] = \EE\bigl[(Z \wedge 1)^\gamma (Z \vee 1)^\gamma\bigr]
  \le \EE\bigl[Z \wedge 1\bigr]^\gamma \, \EE\bigl[(Z \vee 1)^{\frac{\gamma}{1-\gamma}}\bigr]^{1-\gamma} \,,
\]
by H\"older's inequality.
Since $\frac{\gamma}{1-\gamma} \le 1$ for $\gamma \le \frac{1}{2}$, Jensen's inequality yields $\EE\bigl[(Z \vee 1)^{\frac{\gamma}{1-\gamma}}\bigr] \le \EE[Z \vee 1]^{\frac{\gamma}{1-\gamma}}$ and note that $\EE[Z \vee 1]\le \EE[Z]+ 1 = 2$. 

Next consider $\gamma > \frac{1}{2}$. 
Since $\frac{\gamma}{1-\gamma} > 1$, we have $x^{\frac{\gamma}{1-\gamma}} \le x$ for $x \in [0,1]$, and H\"older's inequality gives
\[
  \EE[Z^\gamma] = \EE\bigl[(Z \wedge 1)^\gamma (Z \vee 1)^\gamma\bigr]
  \le \EE\bigl[(Z \wedge 1)^{\frac{\gamma}{1-\gamma}}\bigr]^{1-\gamma}
  \, \EE\bigl[Z \vee 1\bigr]^\gamma
  \le \EE\bigl[Z \wedge 1\bigr]^{1-\gamma}
  \, \EE\bigl[Z \vee 1\bigr]^\gamma  \,,
\]
which completes the proof since $\EE[Z \vee 1] \le 2$.
\end{proof}

\subsection*{Step 2: Finite-volume criterion via a coarse-graining procedure.}

Let us now upgrade the result of the key \Cref{prop:key}, more precisely the fractional moment version \eqref{eq:key-eq-moments}, by extending it to arbitrary \(N\in\N\) and \(\beta > 0\) small (that is, not necessarily in the critical window) and improving the bound with an exponential decay in \(N\).

Recalling \eqref{eq:Ntheta0} and \eqref{eq:RN}, for any $\beta > 0$ and $\th\in\R$ we define $N_\beta(\th)\in\N$ which inverts asymptotically the relation \eqref{eq:thetasharp}:
\begin{equation}
\label{eq:N0}
  N_\beta(\th) \coloneqq
  \pigl\lfloor \e^{-\alpha} \, \e^{\th} \, \e^{\frac{\pi}{\sigma^2(\beta)}} \pigr\rfloor
  =(1+o(1))\, \e^{-\alpha} \, \e^{\th} \, \e^{\frac{\pi}{\sigma^2(\beta)}}
  \quad \text{ as } \beta\downarrow 0\,.
\end{equation} 

\begin{proposition}[Improved bound]\label{prop:key+}
 There exist constants $\beta_0, \hat\th \in (0,\infty)$ such that the following holds: defining $\hat{N}_\beta \coloneqq N_\beta(\hat\th)$ by \eqref{eq:N0}, we have
\begin{equation} \label{eq:aftercoarse-graining}
  \forall \beta \in (0,\beta_0) \,, \
  \forall N \in \N : \quad
  \sup_{f \in \cM_1^{\disc}\bigl(\sqrt{\hat N_\beta}\bigr)} \EE\bigl[Z_{N}^{\beta,\w}(f)^{1/2}\bigr] 
    \le 3 \,
    \e^{-N/\hat N_\beta} \,.
\end{equation}
\end{proposition}

\begin{proof}
We now reparametrize \eqref{eq:key-eq-moments} in terms of $\beta$, using \(N = N_\beta(\th)\): if we let $\beta \downarrow 0$, then the pair \((N_\beta(\th),\beta)\) lies asymptotically in the critical regime \eqref{eq:thetasharp} with $N_\beta(\th) \to \infty$ from \eqref{eq:N0}. 
Then \eqref{eq:key-eq-moments} can be rewritten as follows:
\[
  \limsup_{\beta \downarrow 0} \sup_{f \in \cM_1^{\disc}(\sqrt{N_\beta(\th)})} \EE\bigl[Z_{N_\beta(\th)}^{\beta,\w}(f)^{1/2}\bigr] 
  \leq
  \frac{\sqrt{2
  C}}{\sqrt{\th}} \,.
\]
In particular, given any $\th \ge 3$, we can fix a suitable $\tilde\beta(\th) > 0$ small enough such that (say)
\begin{equation} \label{eq:reco}
    \forall \beta \in \bigl(0,\tilde\beta(\th)\bigr) \colon \qquad
  \sup_{f \in \cM_1^{\disc}(\sqrt{N_\beta(\th)})} \EE\bigl[Z_{N_\beta(\th)}^{\beta,\w}(f)^{1/2}\bigr] 
  \leq
  \frac{2 \, \sqrt{C}}{\sqrt{\th}} \,.
\end{equation}

We next improve this estimate allowing $N\in\N$ to be arbitrary.
The core idea is the following \emph{coarse-graining} result, which gives a finite-volume criterion for the exponential decay of the partition function: 
it shows that if the fractional moment is small at some scale \(L\), then it starts decreasing exponentially in~\(N/L\).
This result is somewhat classical in the literature, but we provide a self-contained proof in \Cref{sec:coarse} below. 
We recall that \(\cM_1^{\disc}(r)\) is defined in \eqref{eq:disc-density}, where we replace for convenience \(\abs{\,\cdot\,}\) with \(\abs{\,\cdot\,}_\infty\).

\begin{proposition}[Coarse-graining]
  \label{prop:coarse}
  If there exist \(L \in \N\), \(\beta > 0\) such that 
  \begin{equation*}
    \sup_{f \in \cM_1^{\disc}(\sqrt{L})} \EE\bigl[Z_{L}^{\beta,\w}(f)^{1/2}\bigr] 
    \le
    \frac{1}{113} \,,
  \end{equation*}
  then for all \(N\in\N\), we have
\begin{equation}\label{eq:NL}
      \sup_{f \in \cM_1^{\disc}(\sqrt{L})} \EE\bigl[Z_{N}^{\beta,\w}(f)^{1/2}\bigr] 
    \le
    3 \, \e^{- N/L} \,.
\end{equation}
\end{proposition}

\begin{remark}
It is enough to prove \eqref{eq:NL} for $N\geq L$, since for $N < L$ we have $\EE\bigl[Z_{N}^{\beta,\w}(f)^{1/2}\bigr] \le
\EE\bigl[Z_{N}^{\beta,\w}(f)\bigr]^{1/2} = 1$.
Also, let us stress that the constant \(\frac{1}{113}\) depends on the distribution of the random walk (we have simply taken a number that works for the simple random walk).
\end{remark}

Recalling \eqref{eq:reco}, we now fix \(\hat\th = (2\cdot 113)^2C\vee 3\) so that \(\frac{2 \, \sqrt{C}}{\sqrt{\hat\th}} \le \frac{1}{113}\).
If we correspondingly define $\beta_0 \coloneqq \tilde{\beta}(\hat\th) > 0$ and $\hat{N}_\beta \coloneqq N_\beta(\hat\th)$, then \eqref{eq:reco} yields
\[
  \forall \beta \in (0,\beta_0) \colon \qquad
  \sup_{f \in \cM_1^{\disc}(\sqrt{\hat{N}_\beta})} \EE\bigl[Z_{\hat{N}_\beta}^{\beta,\w}(f)^{1/2}\bigr] 
  \leq \frac{1}{113} \,.
\]
It only remains to apply \Cref{prop:coarse} with \(L = \hat{N}_\beta\) to complete the proof of \Cref{prop:key+}.
\end{proof}

\subsection*{Step 3. The change of scale argument.}

We finally show that the scale of the starting point in the bound \eqref{eq:aftercoarse-graining} can be enlarged, to get the following result.

\begin{proposition}[Large scale bound]\label{prop:key++}
 In the same setting as \Cref{prop:key+}, we have
\begin{equation} \label{eq:large-scale}
  \forall \beta \in (0,\beta_0) \,, \
  \forall N \in \N : \quad
    \sup_{f \in \cM_1^{\disc}\pigl(\e^{\frac{1}{2} N/\hat{N}_\beta} \sqrt{ N\,}\, \pigr)} \EE\bigl[Z_{N}^{\beta,\w}(f)^{1/2}\bigr]
  \le 7 \, \e^{-\frac{1}{3}  N/\hat{N}_\beta} \,.
\end{equation}
\end{proposition}

Assuming \Cref{prop:key++} for the moment, this completes the proof of the upper bound in \Cref{thm:quantitative}.
Indeed, if we define $\th = \th(N,\beta)$ by \eqref{eq:Ntheta0}, recalling \eqref{eq:etheta} and $\hat{N}_\beta = N_\beta(\hat{\th})$ from \Cref{prop:key+}, with $N_\beta(\cdot)$ defined in \eqref{eq:N0}, we can bound
\begin{equation*}
  \frac{N}{\hat{N}_\beta} 
  \ge N \, \e^{\alpha} \, \e^{-\hat\th} \, \e^{-\frac{\pi}{\sigma^2(\beta)}}
  = \e^{\alpha - \alpha_N} \, \e^{\th(N,\beta) - \hat\th} \ge \mathsf{c} \, \e^{\th(N,\beta) - \hat\th} \qquad \text{with} \quad
  \mathsf{c} \coloneqq \inf_{N\in\N} \e^{\alpha - \alpha_N}  > 0 \,.
\end{equation*}
Plugged into \eqref{eq:large-scale} and using \Cref{lem:frac-moment}, this yields \eqref{eq:upper-main} with \(c_0 = \frac{\mathsf{c}}{2}\e^{-\hat\th}\) and \(c_2= \min\{\frac{\mathsf{c}}{3}\e^{-\hat{\th}} , \frac{1}{7}\}\).

\begin{proof}[Proof of \Cref{prop:key++}]
The key tool is the following general result, which allows to control fractional moments with starting points at two different scales.

\begin{lemma}[Changing scales]
  \label{lem:change-scale}
  For any \(1\leq A \leq B\) and any \(\gamma \in (0,1)\), we have
  \[
  \sup_{f \in \cM_1^{\disc}(\sqrt{B})} \EE\Bigl[Z_N^{\beta,\w} (f)^\gamma\Bigr] 
  \leq
  \biggl(4 \,\frac{B}{A}\biggr)^{1-\gamma} \, \sup_{g \in \cM_1^{\disc}(\sqrt{A})} \EE\Bigl[Z_N^{\beta,\w} (g)^\gamma\Bigr] \,.
  \]
\end{lemma}

\begin{proof}
  We can include the \(L^{\infty}\) ball of radius \(\sqrt{B}\) in the union of \(K\) disjoint \(L^\infty\) balls \((B_i)_{1\leq i\leq K}\) of radius~\(\sqrt{A}\) (we can estimate \(K \le 4 \frac{B}{A}\)).
  For any discrete mass function \(f \in \cM_1^{\disc}(\sqrt{B})\), we can decompose it as \(f = \sum_{i:\alpha_i > 0} \alpha_i g_i\) where \(\alpha_i \coloneqq \sum_{x\in B_i} f(x)\) and if \(\alpha_i>0\), we set \(g_i \coloneqq \frac{1}{\alpha_i} f \, \indic_{B_i}\), which is simply \(f\) conditioned on \(B_i\).
  This way, we may write
  \begin{equation}
  \label{decomp:Zmu}
    Z_N^{\beta,\w} (f) 
    = \sum_{i=1}^{K} \alpha_i \, Z_{N}^{\beta,\w}(g_i) \,.
  \end{equation}
  For \(\gamma\in(0,1)\), using the subadditive inequality \((\sum_{i} z_i)^\gamma \leq \sum_i z_i^{\gamma}\) for \(z_i \ge 0\), we obtain that
  \[
  \EE\bigl[Z_N^{\beta,\w} (f)^\gamma\bigr]
    \le \sum_{i=1}^{K} \alpha_i^{\gamma}\, \EE\bigl[Z_N^{\beta,\w} (g_i)^{\gamma}\bigr] 
    \le \sup_{g \in \cM_1^{\disc}(\sqrt{A})} \EE\bigl[Z_N^{\beta,\w} (g)^{\gamma}\bigr] \,\sum_{i=1}^{K} \alpha_i^{\gamma}   \,,
  \]
  using also translation invariance.
  Now, using H\"older's inequality, we can bound \(\sum_{i=1}^K \alpha_i^{\gamma} \leq K^{1-\gamma}\), so recalling that \(K\le 4 B/A\) this concludes the proof.
\end{proof}
Thanks to \Cref{lem:change-scale}, we deduce from \Cref{prop:key+} that for all \(\beta \in (0,\beta_0)\) and \(N \in\N\)
\[
 \sup_{f \in \cM_1^{\disc}\pigl(\e^{\frac{1}{2} N/\hat{N}_\beta} \sqrt{ N\,}\, \pigr)} \EE\Bigl[Z_{ N}^{\beta,\w}(f)^{1/2}\Bigr]
  \le 
  2  \,\biggl(\frac{ N}{\hat{N}_\beta} \, \e^{ N/\hat{N}_\beta}
    \biggr)^{1/2} 3 \, \e^{- N/\hat{N}_\beta}
  = 6 \, \biggl(\frac{ N}{\hat{N}_\beta} \biggr)^{1/2} \,
  \e^{-\frac{1}{2}  N/\hat{N}_\beta} \,.
\]
This completes the proof of \eqref{eq:large-scale}, since $6 \sqrt{x} \, \e^{-\frac{1}{2}x} \le 7 \, \e^{-\frac{1}{3}x}$ for $x\ge 0$. 
\end{proof}

\subsection{Proof of the lower bound in \texorpdfstring{\Cref{thm:quantitative}}{the main theorem}}
\label{sec:lower-bound}

We consider the following inequality, in the spirit of Paley--Zygmund: for any random variable $Z\ge 0$
\begin{equation}\label{eq:PTZ}
  \EE[Z \wedge 1] \geq \frac{\EE[Z]^2}{1+ \EE[Z^2]} = \frac{\EE[Z]^2}{1+\EE[Z]^2+ \Var[Z]} \,.
\end{equation} 
The proof is simple: starting from the identity $Z = (Z \wedge 1)(Z \vee 1)$, we get by Cauchy--Schwarz
\begin{equation*}
  \EE[Z]^2
  \leq
  \EE[(Z\wedge 1)^2] \, \EE[(Z \vee 1)^2]
  =
  \EE[Z^2\wedge 1] \, \EE[Z^2 \vee 1]
  \leq
  \EE[Z\wedge 1] \, \bigl(1+ \EE[Z^2]\bigr) \,.
\end{equation*}

To prove the lower bound in \eqref{eq:upper-main}, it suffices to apply \eqref{eq:PTZ} with $Z = Z_N^{\beta, \w}(f)$ and $f = \mathcal{U}^{\mathrm{disc}}_{\sqrt{N}}$, noting that $\EE[Z] = 1$ and plugging in the estimate \eqref{eq:ub2mom-unif} from \Cref{prop:second-moment-DP} (with \(\rho =1\)). 
Overall, we get that \(\EE[Z\wedge 1]\ge \bigl(2+c_3\e^{c_3 \th}\bigr)^{-1}\), which shows the lower bound in \eqref{eq:upper-main} with \(c_1 = c_3 + 2\).
\qed

\section{Proof of \texorpdfstring{\Cref{prop:key}}{the key proposition}}
\label{sec:keyprop}

We have shown in \Cref{sec:strategy} how to reduce \Cref{thm:quantitative} to the weaker key \Cref{prop:key}. 
This section is devoted to the proof of \Cref{prop:key}.
We outline the strategy and state general propositions, which are proved in the following subsections.

\subsection{Size bias}

We first introduce the key notion of \emph{size-biased measure}.

\begin{definition}[Size-biased measure]
  Denote by $(\w,\cF,\PP)$ the probability space on which the disorder $\w$ is defined, and recall the point-to-plane partition function $Z_N^{\beta,\w}(x)$ from \eqref{eq:polymeas}. 
  For $x\in\Z^2$, we define the \emph{size-biased measure} with starting point~\(x\) by \(\dd \tPP_x = Z_{N}^{\beta,\w}(x) \dd \PP\), that is,
  \begin{equation}
    \label{eq:size-biased}
    \tPP_x(A) = \tPP_{x,N}^{\beta} (A)
    \coloneqq \EE\bigl[ \indic_{A} \, Z_{N}^{\beta,\w}(x) \bigr] \qquad \text{for any } A\in \cF \,.
  \end{equation}
  More generally, the \emph{size-biased measure with initial condition \(f\)} (a mass function) is
  \begin{equation*}
    \tPP_{f} (\cdot) = \tPP_{f,N}^{\beta} (\cdot)\coloneqq \EE\bigl[  \indic_{(\cdot)} \,  Z_{N}^{\beta,\w}(f) \bigr] \,.
  \end{equation*}
\end{definition}

Let us note that there is a nice interpretation of the size-biased measure \(\tPP_{x}\).
Indeed, simply using Fubini's theorem and the definition \eqref{eq:polymeas}-\eqref{eq:Zf} of \(Z_{N}^{\beta,\w}(f)\) we get that 
\begin{equation}
  \label{def:tilted2}
  \tPP_{f}(A)  
  = \E_f\Bigl[ \EE\Bigl[  \e^{ \sum_{n=1}^N (\beta \w_{n,S_n} - \lambda(\beta))}\, \indic_A \Bigr] \Bigr]  
  = \E_f\bigl[ \tPP^{(S)}(A ) \bigr]
  = \E_{f} \bigl[  \tPP_{N}^{\beta,(S)}(A) \bigr] \,,
\end{equation}
where we have denoted \(\P_f = \sum_{x\in \Z^2} f(x) \P_x\) the law of the simple random walk with initial distribution \(f\), and \(\dd \tPP_{N}^{\beta,(s)} = \prod_{n =1}^N \e^{\beta \w_{n,s_n} - \lambda(\beta)} \dd \PP\) is the law of an environment \textit{tilted along the path~\(s\)}.
Hence the size-biased measure \(\tPP_f\) can be understood as a two-step procedure: first, draw a simple random walk path \(S\) with starting distribution \(f\); then tilt the environment (up to time~\(N\)) by \(\e^{\beta \w- \lambda(\beta)}\) along the path \(S\).
This interpretation may be useful for building intuition when comparing \(\PP\) and \(\tPP_f\), but in practice we will not use it in our proofs.

\begin{remark}[Size bias and total variation distance]
\label{rem:TVdistance}
  Fix any $Z \ge 0$ with $\EE[Z] = 1$ and consider the size-biased measure \( \dd \tilde\PP = Z \dd \PP\). 
  For any event \(A\in \cF\) we have $\EE[Z\wedge 1] \le \PP(A) + \tilde \PP(A^c)$ (just bound $Z\wedge 1 \le 1$ on $A$ and $Z\wedge 1 \le Z$ on $A^c$) and the inequality is sharp, since we can take \(A = \{Z\geq 1\}\) to get equality.
  It follows that
  \[
    \EE[Z\wedge 1] 
    = \inf_{ A\in \cF} \bigl\{ \PP(A) + \tilde \PP(A^c)\bigr\} 
    = 1 - d_{\mathrm{TV}} (\PP, \tilde \PP),
  \]
  where \(d_{\mathrm{TV}}(\mu,\nu) \coloneqq \sup_{A\in \cF}\abs{\mu(A) - \nu(A)}\) is the total variation distance between two probability measures \(\mu,\nu\).
  Therefore, \emph{showing that \(\EE[Z\wedge 1]\) is small corresponds to showing that \(d_{\mathrm{TV}} (\PP, \tilde \PP)\) is close to \(1\)}.
\end{remark}

\begin{remark}[Divergence under the size-biased probability]
  Since \(\EE[Z\wedge 1]=\tEE[\frac{1}{Z}\wedge 1]\), \Cref{thm:quantitative} can also be interpreted as a result on the divergence of \(Z_N^{\beta,\w}(f)\) under the size-biased measure~\(\tPP_f\), more precisely giving a rate at which \(1/Z_N^{\beta,\w}(f)\) tends to \(0\) under \(\tPP_f\).
\end{remark}

\begin{remark}[Anomalous path detection]
  In view of \Cref{rem:TVdistance}, showing that \(\EE[Z_N^{\beta,\w}(f) \wedge 1] \to 0\) amounts to a statistical problem of being able to find an appropriate event $A=A_N$ to discriminate between two environment distributions: \(\PP\) and \(\tPP_{f,N}^{\beta}\).
  Such ``anomalous path detection problems'' have been investigated (mostly in dimension \(1\)), see for instance ~\cite{ACCHZ08,CZ18}, or \cite{ABBDL10} for a discussion on similar hypothesis testing problems.
\end{remark}

\subsection{Strategy of the proof of \texorpdfstring{\Cref{prop:key}}{the key proposition}}

We need to bound \(\EE[Z_N^{\beta,\w}(f)\wedge 1]\) from above in the critical window, that is, for fixed \(\th \in [3,\infty)\), and for diffusive initial conditions \(f\), that is, supported in a ball of radius \(\sqrt{N}\). 

We combine a \emph{change-of-scale} argument, which reduces the initial diffusive scale $\sqrt{N}$ to a smaller scale $\sqrt{\tN}$, with a \emph{change-of-measure} argument.
The first step is the following result, proved in \Cref{sec:change}.

\begin{proposition}[Change of scale and measure]
\label{prop:change}
  For any $\beta > 0$, for any $N,\tN \in\N$ with $\tN \le N$, we can bound
  \begin{equation}
  \label{eq:almost-done-1}
    \sup_{f \in \cM_1^{\disc}(\sqrt{N})} \EE\bigl[Z_{N}^{\beta,\w}(f)\wedge 1\bigr] 
    \leq  8\, \sup_{f \in \cM_1^{\disc}(\sqrt{\tN})} \EE\biggl[Z_{N}^{\beta,\w}(f)\wedge \frac{N}{\tN}\biggr]  \,.
  \end{equation}
  Additionally, for any \(f \in \cM_1^{\disc}(\sqrt{\tN})\) and any event $A_N$ (that may depend on \(f\)), recalling \eqref{eq:size-biased} we have 
  \begin{equation}
  \label{eq:almost-done-2}
    \EE\biggl[Z_{N}^{\beta,\w}(f)\wedge \frac{N}{\tN}\biggr] 
    \leq \frac{N}{\tN} \PP(A_N) + \tPP_f(A_N^c) \,.
  \end{equation}
\end{proposition}
For \eqref{eq:almost-done-1}-\eqref{eq:almost-done-2} to be useful, we must find an event \(A_N=A_N(f)\) such that both $\PP(A_N)$ and $\tPP_f(A_N^c)$ are small, that is, \emph{$A_N(f)$ is atypical under $\PP$ but typical under the size-biased measure $\tPP_f$}, uniformly for $f \in \cM_1^{\disc}(\sqrt{\tN})$.
The following result states that such an event \(A_N\) may be found.

\begin{proposition}[Bounds for the event \(A_N\)]
\label{prop:key2}
  Fix \(1 \le \eta < \th < \infty\) with $\th -\eta \ge 1$.
  For $N\in\N$ set \(\tN=\lfloor\e^{-\eta}N\rfloor\), and consider $\beta = \beta_N(\th)$ in the critical regime \eqref{eq:critical-regime}, or equivalently \eqref{eq:thetasharp} (note that $\tPP_f$ depends on $\beta$, see \eqref{eq:size-biased}).
  Then, for any  $f \in \cM_1^{\disc}(\sqrt{\tN})$ we can find for each $N\in\N$ an event \(A_N =A_N(f) \in \mathcal{F}\) such that
  \begin{gather}
  \label{eq:PAN-bound1}
    \limsup_{N\to\infty} \; \sup_{f \in \cM_1^{\disc}(\sqrt{\tN})} \PP\bigl(A_N (f) \bigr)
    \le C_1 \, \frac{\th-\eta}{\eta}\e^{-(\th-\eta)} \,, \\
  \label{eq:PAN-bound2}
    \limsup_{N\to\infty}  \sup_{f \in \cM_1^{\disc}(\sqrt{\tN})}\tPP_f\bigl( A_N^c(f) \bigr)
    \leq \frac{C_2}{\eta} \,,
  \end{gather}
  where \(C_1,C_2>0\) are universal constants.
\end{proposition}

The reason why the bounds \eqref{eq:PAN-bound1} and \eqref{eq:PAN-bound2} have the specified dependence on \(\eta, \th\) will be clear below (they are determined by the mean and variance of a suitable random variable~$X$). 
For the moment, it suffices to note that plugging these bounds in \eqref{eq:almost-done-1}-\eqref{eq:almost-done-2} we obtain
\[
  \limsup_{N\to\infty} \sup_{f \in \cM_1^{\disc}(\sqrt{N})}  \EE\bigl[Z_{N}^{\beta,\w}(f)\wedge 1\bigr] 
  \leq 8 C_1  \frac{\th-\eta}{\eta} \e^{2\eta - \th} + \frac{8 C_2}{\eta}  \,.
  \]
Then, if we choose \(\eta = \th/3\), this concludes the proof of \Cref{prop:key}.

\smallskip

To prove \Cref{prop:key2}, we need to find \(A_N=A_N(f)\) which is atypical under $\PP$, but which becomes typical under the size-biased measure $\tPP_{f}$.
One could in principle take \(A_N(f) = \{Z_{N}^{\beta,\w}(f) > \eps\}\) for some \(\eps>0\) small enough: 
this would easily yield \(\tPP_f(A_N^c(f)) \leq \eps\) by definition of \(\tPP_f\), but the difficult part remains to show that \(\PP(A_N(f))= \PP(Z_{N}^{\beta,\w}(f) > \eps)\) is small, so this does not simplify the original problem.

A more manageable solution is to find a \emph{simpler random variable \(X=X(f)\) which acts as a proxy for \(Z_{N}^{\beta,\w}(f)\)}, for which we can estimate the mean and variance under \(\PP,\tPP_f\), uniformly in \(f \in \cM_1^{\disc}(\sqrt{\tN})\).
We clarify this strategy in the following lemma.
\begin{lemma}[Choice of the event \(A_N\)]
\label{lem:TChebyshev}
  Consider some random variable \(X(f) = X_N(f)\) such that
  \[
    \EE[X(f)] 
    = 0 \qquad \text{and} \qquad \tEE_f[X(f)] 
    > 0
  \]
  and define the event 
  \[
    A_N
    =A_N(f) 
    =\Bigl\{ X(f) \geq \frac12 \, \tEE_f[X(f)]\Bigr\} \,.
  \]
  Then, we get that 
  \[
    \PP\bigl(A_N(f)\bigr) 
    \leq 4 \; \frac{\Var[X(f)]}{ \tEE_f[X(f)]^2} \,,\qquad \tPP_f\bigl( A_N^c(f)\bigr) 
    \leq 4 \; \frac{\tVV_f[X(f)]}{ \tEE_f[X(f)]^2} \,.
  \]
\end{lemma}
The proof follows directly from Chebyshev's inequality, applied with threshold \(\tfrac12 \, \tEE_f[X(f)]\) to \(X(f)\) under \(\PP\) and to \(X(f)-\tEE_f[X(f)]\) under \(\tPP_f\).
It therefore remains to find a suitable \(X=X(f)\) such that \(\Var[X(f)], \tVV_f[X(f)]\ll \tEE_f[X(f)]^2\) uniformly for \(f \in \cM_1^{\disc}(\sqrt{\tN})\).

\smallskip

\emph{The choice of \(X\) is the most delicate point in the strategy and the main novelty of our proof}. 
We discuss this issue in \Cref{sec:proxy}, arriving at the explicit definition of \(X\) in \eqref{def0:X}, which may be described as the first (linear) term in a coarse-grained chaos expansion of the partition function over time intervals of length~$\tN$.
This provides a ``canonical'' recipe for building a proxy for the partition function in a general setting.

We finally state our main estimates on \(\Var[X(f)]\), \(\tEE_f[X(f)]\), \(\tVV_f[X(f)]\) that will be proved in the next sections.
The first two lemmas follow from second-moment calculations and are proven in \Cref{sec:technicalI}.
The last estimate is more difficult and will be proven in \Cref{sec:technicalII}.

\begin{lemma}[Variance bound]
\label{lem:nontilted}
  Assume the setting of \Cref{prop:key2}. 
  Defining \(X(f) = X_N(f)\) by \eqref{def0:X} below, we have \(\EE[X(f)]=0\) and
  \[
    \limsup_{N\to\infty}\ \sup_{f \in \cM_1^{\disc}(\sqrt{\tN})}\Var[X(f)] 
    \leq  C\; \frac{\eta}{\th-\eta}   \, \e^{\th -\eta}\,, 
  \]
where $C < \infty$ is a universal constant.
\end{lemma}

\begin{lemma}[Size-biased mean bound]
\label{lem:tiltedE}
  Assume the setting of \Cref{prop:key2}.
  Defining \(X(f) = X_N(f)\) by \eqref{def0:X} below, we have
  \[
    \liminf_{N\to\infty}\  \inf_{f \in \cM_1^{\disc}(\sqrt{\tN})} \tEE_{f}[X(f)] 
    \geq c\; \frac{\eta}{\th-\eta}\, \e^{\th -\eta} \,,
  \]
  where \(c>0\) is a universal constant.
\end{lemma}

\begin{proposition}[Size-biased variance bound]
\label{prop:tiltedVar}
  Assume the setting of \Cref{prop:key2}.
  Defining \(X(f) = X_N(f)\) by \eqref{def0:X} below, we have
  \[
    \limsup_{N\to\infty}\ \sup_{f \in \cM_1^{\disc}(\sqrt{\tN})} \tVV_f[X(f)] 
    \leq C'\,\eta \, \Bigl(\frac{1}{\th-\eta} \, \e^{\th -\eta} \Bigr)^2\,,
  \]
where \(C'>0\) is a universal constant.
\end{proposition}

Together with~\Cref{lem:TChebyshev}, these estimates readily show that
\[
\begin{split}
  \limsup_{N\to\infty} \sup_{f \in \cM_1^{\disc}(\sqrt{\tN})} \PP(A_N(f))
  & \leq C c^{-2} \; \frac{\th-\eta}{\eta} \; \e^{ - (\th -\eta)} \,, \\
  \text{ and }\quad \limsup_{N\to\infty}\sup_{f \in \cM_1^{\disc}(\sqrt{\tN})}\tPP_f(A_N^c(f)) 
  & \leq  C'c^{-2}\; \frac{1}{\eta} \,,
\end{split}
\]
which are the bounds announced in \Cref{prop:key2}.
(Of course, the reason why we stated the bounds in the precise form \eqref{eq:PAN-bound1} and \eqref{eq:PAN-bound2} was dictated by the mean and variance of~$X$.)

\begin{remark}[Size bias revisited]
  Another approach to bound \(\tPP_f(A_N^c(f))\) would be to use the size-biased representation~\eqref{def:tilted2}, introducing some well-chosen (random walk) event \(B \in \sigma\{S_n, n\leq N\}\) and then writing \(\tPP_f(A_N^c(f)) \leq \E_f[\tPP_N^{\beta,(S)}(A_N^c(f)) \indic_B] + \P_f(B^c)\).
  This is what is usually done in this setting, see for instance \cite[\S3]{BL17} or \cite[\S6.2]{JL25}.
  The advantage of this idea is that, once one has reduced to work on the event \(B\), it possibly makes it easier to control \(\tEE_N^{\beta,(S)}[X_N(f)]\) and \(\tVV_N^{\beta,(s)}\bigl[X_N(f)\bigr]\), and thus \(\tPP_N^{\beta,(S)}(A_N^c(f))\).
  We will not need such a strategy, since our choice for event \(A_N(f)\) will already make the computation of \(\tEE_f[X_N(f)]\), \(\tVV_f\bigl[X_N(f)\bigr]\) manageable.
\end{remark}

The remainder of this section is devoted to the choice of the proxy \(X\) (see \Cref{sec:proxy}) and to the proof of \Cref{prop:change} (see \Cref{sec:change}). 

\subsection{Choice of a good proxy for the partition function}
\label{sec:proxy}

We next discuss the \emph{choice of the proxy $X(f) = X_N(f)$} for the partition function \(Z_{N}^{\beta,\w}(f)\).
Let us introduce some useful notation.
For \(n,N\in\N\), \(x\in\Z^2\), we denote by $q_n(x)$ or $q(n,x)$ the simple random walk transition probability, that is
\[
  q_n(x)
  = q(n,x)
  \coloneqq\P(S_n=x) \,.
\]
We will also denote 
\begin{equation} 
\label{eq:qf}
  q_n^{(f)}(x) 
  = q^{(f)}(n,x)
  \coloneqq\P_{f}(S_n=x) 
  = \sum_{z\in \Z^2} f(z) q_n(x-z) \,.
\end{equation}

\subsection*{A first approach: chaos expansion and \texorpdfstring{\(L^2\)}{L2} projections.}
Let us define
\[
  \xi_{n,x} 
  = \xi_{n,x}^{(\beta)} 
  \coloneqq \e^{\beta \w(n,x)-\lambda(\beta)} -1  
\]
and notice that the \((\xi_{n,x})_{n\in \N, x\in\Z^2}\) are i.i.d.\ with \(\EE[\xi_{n,x}] =0 \) and \(\EE[\xi_{n,x}^2] = \e^{\lambda(2\beta) -2\lambda(\beta)} -1 \eqqcolon \sigma^2(\beta) \), see \eqref{def:sigmaRN}.
Rewriting the partition function  as \(Z_{N}^{\beta,\w}(f) = \E_f[\prod_{n=1}^N (1+\xi_{n,S_n}) ]\) and expanding the product, we obtain the following \emph{polynomial chaos expansion}
\begin{equation}
  \label{eq:chaosexpansion0}
  Z_{N}^{\beta,\w}(f) 
  = 1 + \sum_{k=1}^N \sum_{1\leq n_1 < \cdots < n_k \leq N} \sum_{x_0\in \Z^2}f(x_0) \sum_{x_1,\ldots, x_k \in \Z^2} \prod_{i=1}^{k} q_{n_i-n_{i-1}}(x_i-x_{i-1}) \prod_{i=1}^{k} \xi_{n_i,x_i} \,,
\end{equation}
where we set by convention \(n_0=0\).
Note that the terms \(\prod_{i=1}^k \xi_{n_i,x_i}\) in the above expansion are orthogonal in \(L^2\).
Therefore, one can reinterpret the above as the \(L^2\) decomposition of \(Z_N^{\beta,\w}(f)\) over the linear subspace of \(L^2\) generated by the orthogonal variables
\[
  \xi(A) 
  = \prod_{z\in A} \xi_{z} \qquad \text{ for every finite set } A \subset \N \times \Z^2  \,.
\]
Then, gathering the space-time points in a subset $A = \{(n_i, x_i) \colon 1 \le i \le k\} \subset \llb 1, N \rrb \times \Z^2$, the chaos expansion above can be rewritten more compactly as follows:
\begin{equation}
\label{eq:chaosexpansion}
  Z_{N}^{\beta,\w}(f) 
  = \sum_{A \subset\; \llb 1, N \rrb \times \Z^2} q^{(f)}(A) \, \xi(A) \,,
\end{equation}
where we have set \(q^{(f)}(A) \coloneqq \P_{f}\bigl(A\subseteq \{(i,S_i)\}_{i\geq 1}\bigr)\). 
If \(A=\varnothing\), then \(q^{(f)}(A) =1\) by convention. If \(A\) contains two points with the same time coordinate, then \(q^{(f)}(A) =0\). Otherwise, if \(A = \{(n_i,x_i)\colon 1\leq i \leq k\}\) with \(1\leq n_1 < \cdots < n_k \leq N\), we have
\begin{equation}
\label{eq:defqx}
  q^{(f)}(A) 
  = \sum_{x_0\in \Z^2} f(x_0) \prod_{i=1}^{k} q_{n_i-n_{i-1}}(x_i-x_{i-1})  \,,
\end{equation}
with \(n_0=0\). 
We also denote \(q(A) =q^{(\delta_0)}(A)\) for simplicity.

A simple choice for a proxy \(X_N\) for \(Z_{N}^{\beta,\w}=Z_{N}^{\beta,\w}(0)\) is to take the first term in the chaos expansion, namely 
\begin{equation}
\label{eq:linearized}
  \sum_{n=1}^N \sum_{x\in \Z^2} q_n(x) \xi_{n,x} \,.  
\end{equation}
This corresponds to the \(L^2\text{-projection}\) of \(Z_{N}^{\beta,\w}\) onto the linear subspace generated by the \((\xi_{n,x})\).
We refer for instance to~\cite[Section~6]{JL25} where the functional \eqref{eq:linearized} is used to show that the martingale critical point is equal to \(0\) as soon as \(R_N\to\infty\).
In fact, one needs a slightly finer strategy than simply use Chebyshev's inequality to bound \(\tilde \PP(A_N^c)\), but let us not dwell on details here.
The method can be pushed to show that \(Z_{N}^{\beta,\w} \to 0\) in probability as soon as \(\sigma^2(\beta) R_N \to\infty\).
(This is not optimal, since the point-to-plane partition function \(Z_{N}^{\beta,\w}\) is known to converge to~$0$ in probability as soon as \(\liminf \sigma^2(\beta) R_N \ge 1\), see \cite[Theorem~2.8]{CSZ17b}.)

In \cite{BL17} (and \cite{BL18} for the disordered pinning model), the authors consider a more involved functional, namely (a slightly modified version of) the \(k\)-th order term in the chaos expansion \eqref{eq:chaosexpansion} with \(f=\delta_0\), that is
\begin{equation}
\label{eq:orderk}
   \sum_{A \subset\; \llb 1, N \rrb \times \Z^2, \abs{A}=k} q(A) \, \xi(A) \,.
\end{equation}
They take \(k=k_N \to \infty\) slowly (in fact \(k_N =(\log\log N)^2\)) to show that \(Z_{N}^{\beta,\w} \to 0\) in probability as soon as \(\liminf \sigma^2(\beta) R_N >1\) (which is closer to the optimal result mentioned above). 
The result of~\cite{BL17} is in fact stronger:
the authors also prove a bound on the free energy.

\smallskip

One could take an even more faithful approximation of \(Z_{N}^{\beta,\w}(f)\) than \eqref{eq:orderk}.
A nearly optimal proxy would indeed be to keep in the chaos expansion~\eqref{eq:chaosexpansion0} all orders \(1\leq k\leq \log N\), namely 
\begin{equation}
\label{eq:almostallorders}
  \sum_{A \subset\; \llb 1, N \rrb \times \Z^2, 1\leq \abs{A}\leq \log N} q^{(f)}(A) \, \xi(A) \,,
\end{equation}
since it captures a positive proportion of the variance of \(Z_N^{\beta,\w}(f)\) at criticality, that is, when \(\sigma^2(\beta) R_N =1 + O(\frac{1}{\log N})\), see \eqref{eq:critical-regime}.
However, this would make the analysis extremely technical: the calculations in~\cite{BL17} are already difficult, so dealing with variance terms in~\eqref{eq:almostallorders} would quickly become prohibitively cumbersome.

\subsection*{A new approach: a coarse-grained version of the chaos expansion}

Our new idea is to introduce a \emph{coarse-graining on the intermediate scale $\tN = \e^{-\eta} N$} (recall \Cref{prop:key2}):
more precisely, to make expression \eqref{eq:almostallorders} more manageable, we only consider subsets \(A\) with \emph{time-width at most~$\tN$}, and we further restrict the starting time of~$A$ to be larger than~$\tN$ (to forget about the initial condition) and smaller than~$N-\tN$. 
This leads to our proxy for \(Z_N^{\beta,\w}(f)\):
\begin{equation}
  \label{def0:X}
\begin{gathered}
  X(f) \coloneqq \sum_{A \in \mathcal{I}} q^{(f)}(A) \, \xi(A)
  \qquad \text{with} \\
   \mathcal{I} \coloneqq \bigl\{ A \subset \llb 1, N \rrb \times \Z^2 \colon \
   1\leq \abs{A}\leq \log N , \
   \mathrm{width}(A) \leq \tN , \
   \mathrm{start}(A) \in \llb \tN + 1, N-\tN \rrb 
  \bigr\} 
\end{gathered}
\end{equation}
where for \(A = \{(n_1,x_1), \ldots, (n_k,x_k)\}\) with \(1\leq n_1 < \cdots <n_k\) we have defined the quantities \(\mathrm{start}(A) \coloneqq n_1\) and \(\mathrm{width}(A) \coloneqq n_k-n_1\).

To give some more insight, let us explain why $X(f)$ in \eqref{def0:X} roughly corresponds to \emph{the first (linear) term in a suitable coarse-grained version of the chaos expansion}. 
Assuming for simplicity that \(M \coloneqq N/\tN\) is an integer, we can write the Hamiltonian in \eqref{eq:polymeas} as a sum of terms  corresponding to time intervals of size $\tN$:
\begin{equation*}
    H_N^{\beta,\w}(S) = \sum_{j=1}^{M} \cH_{j}^{\beta,\w}(S) \qquad \text{ with } \quad
  \cH_{j}^{\beta,\w}(S) 
  = \sum_{n \in \llb (j-1)\tN +1, j \tN \rrb} \bigl(\beta \w(n,S_n) - \lambda(\beta)\bigr) \,,
\end{equation*}
so that we can write 
\[
Z_{N}^{\beta,\w}(f)  
=
\E_f\Biggl[\,\prod_{j=1}^{M} \e^{\cH_{j}^{\beta,\w}(S)}\Biggr] \,.
\]
Then, writing each term \(\e^{\cH_{j}^{\beta,\w}(S)} = 1 + (\e^{\cH_{j}^{\beta,\w}(S)}-1) \eqqcolon 1 + \Xi_j(S)\) and expanding the product, we obtain a \textit{coarse-grained version} of the chaos expansion:
\begin{equation*}
  Z_{N}^{\beta,\w}(f) = 1 +  \sum_{j=1}^{M} \E_f\bigl[\, \Xi_j(S)\bigr]  + \sum_{1\leq j_1 < j_2 \leq M} \E_f\bigl[\, \Xi_{j_1}(S) \, \Xi_{j_2}(S) \bigr] + \cdots  \,,
\end{equation*}
where we omit higher-order terms to lighten notation.
The inspiration for our proxy~$X(f)$ is \emph{the first (linear) term \(\sum_{j=1}^{M} \E_f\bigl[\Xi_j(S)\bigr]\)} in this coarse-grained expansion (which may be viewed as a generalization of the basic linear approximation~\eqref{eq:linearized} on the coarse-grained scale \(\tN\)). 
The actual choice~\eqref{def0:X} for our proxy corresponds to a slight modification of this idea, where we consider a ``sliding'' strip of width \(\tN\).
The restriction \(1\leq \abs{A}\leq \log N\) will be needed for technical reasons, namely to obtain a good control on \(\tVV_f[X(f)]\).

\subsection{Change of scale and measure: proof of \texorpdfstring{\Cref{prop:change}}{the proposition}}
\label{sec:change}

We start with a \emph{change of scale} for 
the starting point.
For this purpose, we show an analogue of \Cref{lem:change-scale}, except for the truncated mean \(\EE[Z_N^{\beta,\w}(f)\wedge 1]\) rather than the fractional moment \(\EE[Z_N^{\beta,\w}(f)^{\gamma}]\).

\begin{lemma}[Change of scale]
\label{lem:change-scale2}
  For any $A,B \in \N$ with \(A \leq B\), we have
  \begin{equation}
  \label{eq:change-scale-bis}
    \sup_{f \in \cM_1^{\disc}(\sqrt{B})} \EE\bigl[Z_{N}^{\beta,\w}(f)\wedge 1\bigr] 
    \le 8 \, \sup_{f \in \cM_1^{\disc}(\sqrt{A})} \EE\biggl[Z_{N}^{\beta,\w}(f)\wedge \frac{B}{A}\biggr] 
  \end{equation}
\end{lemma}

The proof is a direct consequence of the following general lemma, together with the same decomposition as in the proof of \Cref{lem:change-scale}, see~\eqref{decomp:Zmu};
here one may take \(K \le 4 \frac{B}{A}\).

\begin{lemma}
  Let \((\alpha_i)_{1\leq i\leq K}\) be non-negative numbers with \(\sum_{i=1}^K \alpha_i = 1\), and let \((Z_i)_{1\leq i \leq K}\) be non-negative random variables.
  Then, if we set \(Z =\sum_{i=1}^K \alpha_i Z_i\), we have
  \[
    \EE[Z\wedge 1]\le 2 \max_{1\le i\le K} \EE\bigl[ Z_i\wedge K \bigr].
  \]
\end{lemma}
\begin{proof}
  Define \(A_i \coloneqq \{Z_i>K\}\) and let \(B=\bigcup_{i=1}^K A_i\).
  Then, bounding \(Z\wedge 1 \leq 1\) on the event \(B\) and \(Z\wedge 1 \leq Z\) on \(B^c\), we have that
  \[
    \EE[Z\wedge 1]
    \le \PP(B)+\EE[Z \, \indic_{B^c}]
    \le \sum_{i=1}^K \PP(A_i)+\sum_{i=1}^K \alpha_i \, \EE[Z_i \,\indic_{A_i^c}] \,,
  \]
  where we have used sub-additivity for the first term and  \(B^c = \bigcap_{i=1}^{K} A_i^c \subseteq A_i^c\) for all \(i\).
  Now, since \(A_i=\{Z_i>K\}\), we have that \(\PP(A_i)\le\oneover{K}\EE[Z_i\wedge K]\) by Markov's inequality and \(\EE[Z_i\, \indic_{A_i^c}]\le \EE[Z_i\wedge K]\) by definition of \(A_i^c\).
  Plugging this in the above gives that 
  \[
  \EE[Z\wedge 1] \leq \sum_{i=1}^K \oneover{K} \, \EE[Z_i\wedge K] + \sum_{i=1}^K \alpha_i \, \EE[Z_i\wedge K]\,,
  \]
  which concludes the proof.
\end{proof}

We next use a \emph{change of measure} argument to estimate the right-hand side of~\eqref{eq:change-scale-bis}.
We state it both for the truncated mean \(\EE[Z_N^{\beta,\w}(f)\wedge \frac{N}{\tN}]\) and for the fractional moment \(\EE[Z_N^{\beta,\w}(f)^{\gamma}]\), since the proof we have is simplified with respect to what we found in the literature.

\begin{lemma}[Change of measure]
  \label{lem:change-measure}
  Let \(Z\geq 0\) be a non-negative random variable.
  For any \(L>0\) and any event \(A \in \mathcal{F}\), we have
  \[
  \EE[Z\wedge L]\le L\,\PP(A)+\EE[Z\indic_{A^c}] \,.
  \]
  If additionally \(\EE[Z] =1\), then for any \(\gamma \in (0,1)\), we have, for any event \(A \in \mathcal{F}\)
  \[
  \EE[Z^{\gamma}] \leq \PP(A)^{1-\gamma} + \EE[Z\indic_{A^c}]^{\gamma} \,.
  \]
\end{lemma}
The above both simplifies and strengthens~\cite[Lem.~2.2]{JL24a}, which controls the moment of order~$1/2$; in fact, we simplify its proof and get a general fractional moment (note that~\cite[Lem.~3.2]{JL25} also controls a fractional moment, but in a non-optimal way).

\begin{proof}
  For the first inequality, we simply bound \(Z\wedge L \leq L\) on~\(A\) and \(Z\wedge L \leq Z\) on~\(A^c\): this gives the desired bound.

  For the fractional moment, we write \(\EE[Z^\gamma]  = \EE[Z^\gamma \indic_{A}]  + \EE[Z^\gamma \indic_{A^c}]\).
  For the first term, we use H\"older's inequality to get \(\EE[Z^\gamma \indic_{A}] \leq \EE[Z]^{\gamma} \PP(A)^{1-\gamma} = \PP(A)^{1-\gamma}\).
  For the second term, we use Jensen's inequality to get \(\EE[Z^\gamma \indic_{A^c}] \leq \EE[Z \indic_{A^c}]^{\gamma}\).
  This concludes the proof.
\end{proof}

\Cref{prop:change} is a simple combination of \Cref{lem:change-scale2} (for~\eqref{eq:almost-done-1}) with \(A=\tN\), \(B = N\) and of \Cref{lem:change-measure}  with \(L=\frac{N}{\tN}\), \(Z=Z_N^{\beta,\w}(f)\) (for \eqref{eq:almost-done-2}), recalling that \(\tPP_f(A_N^c) \coloneqq \EE[Z_N^{\beta,\w}(f)\indic_{A_N^c}]\).
\qed

\section{Second-moment estimates}
\label{sec:technicalI}

In this section, we prove \Cref{lem:nontilted} and \Cref{lem:tiltedE}, which rely mostly on second-moment estimates.
(\Cref{prop:tiltedVar} requires third-moment-type estimates;
we will prove it in \Cref{sec:technicalII}.)
We also prove \Cref{prop:second-moment-DP} in \Cref{sec:Variances}.

\emph{Throughout this section, 
we fix \(1 \le \eta < \th < \infty\) with $\th -\eta \ge 1$, for $N\in\N$ we set \(\tN=\lfloor\e^{-\eta}N\rfloor\), and we consider $\beta = \beta_N(\th)$ in the critical regime \eqref{eq:critical-regime}; equivalently, one may use the refined parametrization \eqref{eq:thetasharp}, which is asymptotically equivalent for fixed \(\th\).}

\subsection{Preliminary notation and variance estimate}

Let us rewrite the proxy \(X\) from \eqref{def0:X} in a form more convenient for calculations.
Introduce
\begin{equation}
\label{def:Im}
  \mathcal{I}_m 
  \coloneqq \bigl\{ A \subset \N\times \Z^2 \colon \mathrm{start}(A) =m \,, \ \mathrm{width}(A) \leq \tN \,,\ 1\leq \abs{A} \leq \log N  \bigr\} \,.
\end{equation}
We decompose~\(X(f)\) into contributions from strips of width \(\tN\) that start at time~\(m\):
\begin{equation}
\label{def:X}
  X(f)
  = \sum_{m=\tN+1}^{N-\tN} X_m(f) \qquad \text{ with }\qquad 
  X_m(f) \coloneqq \sum_{A\in\mathcal{I}_m } q^{(f)}(A) \, \xi(A) \,.
\end{equation}

\begin{remark}[Orthogonal projection]
  \label{rem:Xm}
  Denoting by \(\Pi_{\mathcal{I}_m}\) the orthogonal projection onto the linear subspace of \(L^2\) generated by the \(\xi(A)\) with \(A \in \mathcal{I}_m\), we can write \(X_m(f) = \Pi_{\mathcal{I}_m} Z_N^{\beta,\w}(f) \,.\)
\end{remark}

If \((m,y)\) denotes the first point in \(A\), we can write \(q^{(f)}(A) = q_m^{(f)}(y) \, q(A') \) with \(A' = A - (m,y)\) the set \(A\) translated by its first point (with this point being removed); see \eqref{eq:defqx}. 
This leads to the following decomposition of $X_m(f)$:
\begin{equation}
  \label{def:Xm}
  X_m(f) = \sum_{y\in \Z^2} q_m^{(f)}(y) \, \xi_{m,y}\, \hat Z_{\tN}^{\beta,\w} (m,y) \,,
\end{equation}
where, in view of \eqref{def:Im}, \(\hat Z_{\tN}^{\beta,\w}(m,y) \) is a partition function starting from \((m,y)\) with time-width at most \(\tN\) and restricted to chaos orders up to \(\log N -1\).
More precisely, denoting by \(\theta^{m,y} \w = (\w_{n+m,x+y})_{n\in \N,x\in \Z^2}\) the translated environment, we can write
\begin{equation}
  \label{def:hatZ}
\hat Z_{\tN}^{\beta,\w} (m,y) 
= \hat Z_{\tN}^{\beta,\theta^{m,y} \w}
\qquad \text{with} \qquad
\hat Z_{\tN}^{\beta,\w} 
\coloneqq \sum_{A \subset \llb 1,\tN \rrb \times \Z^2, \abs{A} \leq \log N -1} q(A) \xi(A) \,.
\end{equation}
We stress that \(\hat Z_{\tN}^{\beta,\w}(m,y)\) depends only on the variables \((\xi_{n,z})_{n \ge m+1, z\in \Z^2}\), hence it is independent of \(\xi_{m,y}\).

Looking back at \eqref{def:Xm}, it follows that the $X_m(f)$ are centered and pairwise uncorrelated:
\begin{equation}
  \label{eq:orthogonal}
  \Cov(X_m(f) \,, X_{m'}(f)) = \indic_{\{m'\}}(m) \, \Var\bigl[X_m(f)\bigr] \,.
\end{equation}
Similarly, recalling that $\EE[(\xi_{n,x})^2] = \sigma^2(\beta)$, we can write
\begin{equation}
  \label{eq:varXm}
  \Var\bigl[X_m(f)\bigr] 
  = \sum_{y\in \Z^2} q_m^{(f)}(y)^2 \, \sigma^2(\beta) \, \EE\bigl[\hat Z_{\tN}^{\beta,\w} (m,y)^2\bigr]
  = q_{2m}(f,f)
  \, \sigma^2(\beta) \, 
  \EE\bigl[(\hat Z_{\tN}^{\beta,\w})^2\bigr] \,,
\end{equation}
where we have used translation invariance and introduced the collision kernel
\begin{equation}
  \label{eq:collision}
q_{2m}(f,g) \coloneqq \sum_{x,x'\in \Z^2} f(x) q_{2m}(x-x') g(x') = \sum_{y\in \Z^2} q_m^{(f)}(y)\, q_m^{(g)}(y) \,, 
\end{equation}
the last identity following from Chapman--Kolmogorov.

Notice that, by orthogonality of the \(\xi(A)\) in the definition~\eqref{def:hatZ} of \(\hat Z_{\tN}^{\beta,\w}\), we have
\begin{equation}
\label{def:VN}
  \cV_{\tN} 
  = \cV_{\tN}(\beta,N)
  \coloneqq \EE\bigl[ \bigl(\hat Z_{\tN}^{\beta,\w} \bigr)^2\bigr]  
  = \sum_{k=0}^{\log N -1 } \sigma^2(\beta)^{k} \sum_{A\subseteq \llb 1,\tN \rrb\times\Z^2\,,\, \abs{A}=k} q(A)^2 \,.
\end{equation}

Then, we have the following estimate, whose proof is postponed to \Cref{sec:Variances} below.

\begin{lemma}
\label{lem:V} 
  Fix \(1 \le \eta < \th < \infty\) with $\th -\eta \ge 1$. For $N\in\N$ set \(\tN=\lfloor\e^{-\eta}N\rfloor\) and take $\beta = \beta_N(\th)$ in the critical regime \eqref{eq:critical-regime}, or equivalently \eqref{eq:thetasharp}. 
  For $N$ sufficiently large we have
  \begin{equation}
  \label{eq:Vbounds}
    \frac{c}{\th-\eta} \, \e^{\th-\eta}
    \le \sigma^2(\beta_N) \, \cV_{\tN}
    \le \frac{c'}{\th-\eta} \, \e^{\th-\eta} \,,
  \end{equation}
  where \(c,c'\in(0,\infty)\) are universal constants (we can take any $c < \frac{\pi}{4} \, \e^{-2\gamma}$ with \(\gamma\) the Euler--Mascheroni constant and any $c' > \pi$).
\end{lemma}

\begin{remark}
Even though we only need to apply \Cref{lem:V} when $\th$ is fixed,
an inspection of the proof shows that we could allow for $\th = \th_N \to \infty$, as long as $\th_N \ll \sqrt{\log N}$.
\end{remark}

\subsection{Variance and size-biased expectation: proofs of \texorpdfstring{\Cref{lem:nontilted,lem:tiltedE}}{}}

Recall the collision kernel \(q_{2m}(f,g) \coloneqq \sum_{x,x'} f(x) q_{2m}(x-x') g(x')\), and introduce the weighted Green function between time \(s\leq t\):
\begin{equation}
  \label{eq:weightedGreen}
  G_{s,t}(f,g) \coloneqq \sum_{m=s+1}^t q_{2m}(f,g) \,.
\end{equation}
Since the \(X_m(f)\) are centered and uncorrelated (see~\eqref{eq:orthogonal}), recalling the computation~\eqref{eq:varXm}, we therefore end up with the following expression for \(\Var[X(f)]\):
\begin{equation}
  \label{eq:VarXf}
  \Var[X(f)]
  = \sum_{m=\tN+1}^{N-\tN} \Var\bigl[X_{m}(f)\bigr] 
  = G_{\tN,N-\tN} (f,f) \, \sigma^2(\beta)\, \cV_{\tN} \,.
\end{equation}

As far as the size-biased expectation of \(X(f)\) is concerned, notice that by linearity we have that \(\tEE_f[X(f)] = \sum_{m=\tN+1}^{N-\tN} \tEE_f[X_{m}(f)]\).
Recalling the definition \eqref{eq:size-biased} of \(\tEE_f\) and the orthogonal projection \(\Pi_{\mathcal{I}_m}\) from \Cref{rem:Xm} (recall that \(X_m(f)= \Pi_{\mathcal{I}_m} Z_{N}^{\beta,\w}(f)\)), we can write
\begin{equation}
  \label{eq:tEXm}
  \begin{split}
  \tEE_f[X_m(f)] 
  = \EE[X_m(f)\, Z_{N}^{\beta,\w}(f)] 
  & = \EE[X_m(f)\, \Pi_{\mathcal{I}_m} Z_{N}^{\beta,\w}(f)] \\
  & = \EE\bigl[X_m(f)^2\bigr]
  =  \Var\bigl[X_m(f)\bigr]\,,
  \end{split}
\end{equation}
since \(X_m(f)\) is centered.
In particular, we get that 
\begin{equation}
  \label{eq:tEX}
  \tEE_f[X(f)] 
  = \sum_{m=\tN+1}^{N-\tN} \Var\bigl[X_{m}(f)\bigr] 
  = G_{\tN,N-\tN} (f,f) \, \sigma^2(\beta)\, \cV_{\tN}  
  =\Var[X(f)] \,.
\end{equation}

We also have the following lemma, which controls \(G_{\tN,N-\tN}(f,f)\), uniformly for \(f\in \cM_1^{\rm disc}(\sqrt{\tN})\).
\begin{lemma}
  \label{lem:Gfg}
  There are universal constants \(C,C'>0\) such that, for any \(1\leq s < t\)
  \[
   C \log \Bigl(\frac{t}{s}\Bigr) \leq \inf_{f \in \cM_1^{\rm disc}(\sqrt{s})}G_{s,t}(f,f) \leq \sup_{f \in \cM_1^{\rm disc}(\sqrt{s})} G_{s,t}(f,f) \leq C' \log \Bigl(\frac{t}{s}\Bigr) \,.
  \]
\end{lemma}

\begin{proof}
  The proof is an immediate consequence of the local central limit theorem together with a parity decomposition.
  More precisely, there are universal constants \(c,c'\) such that, for any \(m\in\N\) and \(z\in\Z^2\) \emph{even} (that is, with \(z_1 + z_2\) even) such that  \(\abs{z}\leq 2\sqrt{m}\), we have \(\frac{c}{m} \leq q_{2m}(z) \leq \frac{c'}{m}\).
In particular, for \(m > s\) we get that 
  \begin{equation}
    \label{eq:LLT-qff}
    \frac{c}{m} \leq \inf_{\substack{\abs{x},\abs{x'} \leq \sqrt{s} \\ x-x' \text{ even}}} \, q_{2m}(x-x') \leq \sup_{\substack{\abs{x},\abs{x'} \leq \sqrt{s} \\ x-x' \text{ even}}} \, q_{2m}(x-x') \leq \frac{c'}{m} \,,
  \end{equation}
  from which one deduces that \(\frac{c}{2m}\leq q_{2m}(f,f) \leq \frac{c'}{m}\) uniformly for \(f \in \cM_1^{\rm disc}(\sqrt{s})\), where the constant \(1/2\) in the lower bound comes from a parity consideration.
  Summing over \(m \in \llb s+1, t \rrb\) gives the desired conclusion.
\end{proof}

Combining the bounds in \Cref{lem:V,lem:Gfg}, the formulas~\eqref{eq:VarXf}-\eqref{eq:tEX} then yield that 
\[
c\,C\log \Bigl(\frac{N}{\tN}-1\Bigr) \, \frac{\e^{\th-\eta}}{\th-\eta}   
\leq 
\inf_{f \in \cM_1^{\disc}(\sqrt{\tN})}\tEE_f[X(f)] 
\leq 
\sup_{f \in \cM_1^{\disc}(\sqrt{\tN})}\Var[X(f)] 
\leq  
C'c' \, \log \Bigl(\frac{N}{\tN}\Bigr) \, \frac{\e^{\th-\eta}}{\th-\eta} \,,
\]
which concludes the proofs of \Cref{lem:nontilted,lem:tiltedE}.
\qed

\begin{remark}
  \label{rem:boundsgood}
  Let us stress that since we have \(\tEE_f[X(f)] = \Var[X(f)]\), going back to \Cref{lem:TChebyshev} we get
  \[
    \PP(A_N(f)) \leq \frac{4}{\Var[X(f)]} =\frac{4}{G_{\tN,N-\tN}(f,f) \,\sigma^2(\beta_N) \, \cV_{\tN}} \,.
  \]
  The role of \Cref{lem:V,lem:Gfg} is simply to make this bound more explicit, in fact, only lower bounds on~\(\sigma^2(\beta_N)\cV_{\tN}\) and \(\inf_{f \in \cM_1^{\rm disc}(\sqrt{\tN})}G_{\tN,N-\tN}(f,f)\) are actually needed here.
\end{remark}

\subsection{\texorpdfstring{Second moment estimates: proofs of \Cref{lem:V} and \Cref{prop:second-moment-DP}}{Second moment estimates: proofs}}
\label{sec:Variances}

Before we start the proofs, let us introduce some further notation and useful estimates.
We define
\begin{equation}
\label{def:u}
  u(n) 
  \coloneqq  \sum_{x\in \Z^2} q_n(x)^2 
  = q_{2n}(0) = \P(S_{2n}=0) \,,
\end{equation}
so that \(R_N = \sum_{n=1}^N u(n)\).
We recall that by \eqref{eq:RN}
\begin{equation}
\label{eq:Rbounds}
  \pi R_N = \log N + \alpha +o(1),\quad 
  \alpha\coloneqq \gamma+4\log 2-\pi\approx 0.208 \,.
\end{equation}
For \(I=\{i_1,\ldots,i_k\}\) with  \(i_1 <\cdots <i_k\), we also set 
\[
  u(I)\coloneqq \prod_{j=1}^k u(i_j-i_{j-1}) \,,
\]
with by convention \(i_0=0\), and \(u(\varnothing)=1\).
Let us also introduce, for \(m \geq 2\), i.i.d.\ random variables \(\Tm,\Tm_1,\Tm_2,\ldots\) taking values in \(\{1,\ldots,m\}\) with
\[
\mathrm{P}(\Tm\le j)=\frac{R_{j\wedge m}}{R_m}\,.
\]
In particular, we have \(\mathrm{P}(\Tm =j) =\frac{u(j)}{R_m}\indic_{\{1,\ldots,m\}}(j)\).
Therefore, we can write 
\begin{equation*}
  \sum_{I\subseteq \llb 1,m \rrb\,, \, \abs{I}=k} u(I) =  (R_m)^k\, \mathrm{P}\bigl( \tauk \leq m \bigr) \qquad \text{with}
  \quad \tauk \coloneqq \Tm_1 + \cdots+ \Tm_k \,.
\end{equation*}

\subsubsection{\texorpdfstring{Proof of \Cref{lem:V}}{Proof of}}

To simplify notation, we write \(\log N\) in place of \(\lfloor\log N\rfloor\) when integrality is needed.
Summing over the spatial coordinates in the definition~\eqref{def:VN} of~\(\cV_{\tN}\), we have that
\[
  \cV_{\tN} 
  = \sum_{k=0}^{\log N-1} \sigma^2(\beta)^k\sum_{I\subseteq \llb 1,\tN \rrb\,, \, \abs{I}=k} u(I) 
  = \sum_{k=0}^{\log N-1} \bigl(\sigma^2(\beta) R_{\tN}\bigr)^k  \mathrm{P}\bigl( \tau^{(\tN)}_k\leq \tN \bigr)\,.
\]
Bounding \(\mathrm{P}(\tau^{(\tN)}_k \leq \tN)\leq 1\) and \(\mathrm{P}(\tau^{(\tN)}_k \leq \tN) \geq \mathrm{P}(\tau_{\log N}^{(\tN)} \leq \tN)\) and summing the geometric series, we therefore get that 
\[
  \frac{(\sigma^2(\beta) R_{\tN})^{\log N}-1}{ \sigma^2(\beta) R_{\tN} -1} \, \mathrm{P}\bigl( \tau_{\log N}^{(\tN)} \leq \tN \bigr) 
  \leq \cV_{\tN} 
  \leq \frac{(\sigma^2(\beta) R_{\tN})^{\log N}-1}{ \sigma^2(\beta) R_{\tN} -1} \,.
\]
Notice that \(\log \tN = \log N - \eta \ge \frac{1}{2} \log N\) for $N$ large enough (in fact for $N \ge \e^{2\eta}$), so we can bound \(\log N \leq 2 \log \tN\) and then use \cite[Proposition~1.3]{CSZ19-3rd} to get that 
\[
  \mathrm{P}\bigl( \tau_{\log N}^{(\tN)} \leq \tN \bigr) 
  \geq \mathrm{P}\bigl( \tau_{2 \log \tN}^{(\tN)} \leq \tN \bigr) 
  \xrightarrow[\ N\to\infty \ ]{} \mathsf{p} > 0 \,,
\]
where \(\mathsf{p} = \P(Y_2 \le 1)\) is a universal constant associated with the Dickman subordinator \((Y_t)_{t\ge 0}\). 
A direct computation using the density of \(Y_2\) \cite[Theorem~1.1]{CSZ19-3rd} shows that \(\mathsf{p} = \frac{1}{2} \e^{-2\gamma} \approx 0.158  \).
Altogether, to prove our goal \eqref{eq:Vbounds}, we only need to get upper and lower bounds on \(\sigma^2(\beta) R_{\tN}-1\) and \((\sigma^2(\beta) R_{\tN})^{\log N}-1\) as $N\to\infty$, when we fix $\beta = \beta_N$ in the critical regime \eqref{eq:thetasharp}.

First of all, we can use~\eqref{eq:Rbounds} and the fact that \(\tN=\e^{-\eta} N\) to get that \( R_{\tN} - R_{N} = -\frac{\eta}{\pi} + o(1)\) as \(N\to\infty\).
Hence, for $\beta = \beta_N$ satisfying~\eqref{eq:thetasharp}, we have that 
\[
  \sigma^2(\beta_N) R_{\tN} -1 
  = \frac{R_{\tN}}{R_N - \frac{\th+o(1)}{\pi}}  -1 
  = \frac{\frac{\th +o(1)}{\pi} - \frac{\eta}{\pi} +o(1)}{R_N - \frac{\th+o(1)}{\pi}} 
  = (1+o(1))\, \frac{\frac{\th-\eta}{\pi}}{R_N } \,,
\]
where for the last identity we note that $o(1) = o(\th - \eta)$ since $\th - \eta \ge 1$.
In particular, using also that \( \sigma^2(\beta_N) R_N = 1+o(1)\) again by \eqref{eq:thetasharp}, we get that as \(N\to\infty\)
\[
  \frac{\sigma^2(\beta_N)}{\sigma^2(\beta_N) \, R_{\tN} -1} 
  = (1+o(1))\, \frac{\pi}{\th-\eta}\,.
\]

On the other hand, using again \( R_{\tN} = R_{N} -\frac{\eta}{\pi} + o(1)\) and \eqref{eq:thetasharp},
we get as $N\to\infty$
\begin{equation*}
  \sigma^2(\beta_N)\, R_{\tN}
  = \frac{R_{\tN}}{R_N - \frac{\th+o(1)}{\pi}} = \frac{1 - \frac{\eta}{\pi R_N} + \frac{o(1)}{R_N}}{1 - \frac{\th}{\pi R_N} + \frac{o(1)}{R_N}} 
  = \e^{ \frac{\th-\eta +o(1)}{\pi R_N}} \,.
\end{equation*}
Taking the \(\log N\) power and recalling that \(\pi R_N \sim \log N\) as \(N\to\infty\), we get that
\[
  (\sigma^2(\beta_N) \, R_{\tN})^{\log N}
  = (1+o(1)) \, \e^{\th-\eta} \,.
\]
Gathering the previous estimates, we therefore get that as \(N\to\infty\)
\begin{equation*}
  (1+o(1))\, \frac{\frac{\pi}{2} \, \e^{-2\gamma} }{\th-\eta} \, (\e^{\th-\eta}-1)
  \le \sigma^2(\beta_N) \, \cV_{\tN}
  \le (1+o(1))\, \frac{\pi}{\th-\eta} \, (\e^{\th-\eta}-1) \, ,
\end{equation*}
which concludes the proof of \Cref{lem:V}, using also that \(\frac12 \e^{\th-\eta}\leq \e^{\th-\eta}-1 \leq \e^{\th-\eta}\) for \(\th-\eta\geq 1\).
\qed

\subsubsection{\texorpdfstring{Proof of \Cref{prop:second-moment-DP}}{Proof of second moment upper bound}}

First of all, starting from the chaos expansion~\eqref{eq:chaosexpansion}, we can decompose over the starting point of non-empty subsets \(A\): we can write, in analogy with \eqref{def:X},
\[
  Z_N^{\beta,\w}(f) - \EE\bigl[Z_N^{\beta,\w}(f) \bigr] =    \sum_{x\in \Z^2} f(x) \sum_{m=1}^N  \sumtwo{ A \subset \llb m,N\rrb\times\Z^2}{\mathrm{start}(A)=m} q^{(x)}(A) \xi(A) \,.
\]
Notice that here we have subtracted the contribution of \(A=\varnothing\), that is, \(\EE[Z_N^{\beta,\w}(f)]\).
Now, note that if \((m,y)\) is the first point of~\(A\) we can again write \(q^{(x)}(A) = q_m(y-x) q(A')\), with \(A'=A-(m,y)\) the set \(A\) translated by its first point (with this point being removed).
Hence, decomposing over \((m,y)\) similarly as in~\eqref{def:Xm}, we get by orthogonality and translation invariance that 
\[
  \Var[Z_N^{\beta,\w}(f)] = \sum_{m=1}^N \sum_{y\in \Z^2} q_m^{(f)}(y)^2 \sigma^2(\beta)\EE\bigl[Z_{N-m}^{\beta,\w}(0)^2\bigr] \,.
\]
By Chapman--Kolmogorov, notice that \(\sum_{y\in \Z^2} q_m^{(f)}(y)^2 = q_{2m}(f,f)\), as in~\eqref{eq:collision}. 
Using that \(\EE\bigl[Z_{N-m}^{\beta,\w}(0)^2\bigr]\leq \EE\bigl[Z_{N}^{\beta,\w}(0)^2\bigr]\) and recalling the definition \eqref{eq:weightedGreen} of the weighted Green function (writing \(G_N(f,f)=G_{0,N}(f,f)\) for simplicity), we therefore end up with the following upper bound on the variance:
\begin{equation*}
  \Var[Z_N^{\beta,\w}(f)] \leq  G_N(f,f) \, \sigma^2(\beta)\, \EE\bigl[Z_{N}^{\beta,\w}(0)^2\bigr] \,.
\end{equation*}

\smallskip
Now, for \(f= \mathcal{U}^{\disc}_{\rho \sqrt{N}}\), we get that 
\[
G_N(\mathcal{U}^{\disc}_{\rho \sqrt{N}},\mathcal{U}^{\disc}_{\rho \sqrt{N}}) = \frac{1}{\abs{B(\rho \sqrt{N}) \cap \Z^2}^2} \sum_{m=1}^N \sum_{x,x' \in B(\rho \sqrt{N})\cap\Z^2} q_{2m}(x-x') \leq \frac{C}{\rho^2} \,,
\]
where we first used that \(\sum_{x\in \Z^2} q_{2m}(x-x') =1\) and then the fact that the cardinality of \(B(\rho \sqrt{N}) \cap \Z^2\) satisfies \(\abs{B(\rho \sqrt{N}) \cap \Z^2} \ge c (1\vee  \rho^2 N)\).

\smallskip

Altogether, we only need to get an upper bound on \(\sigma^2(\beta)\EE[Z_{N}^{\beta,\w}(0)^2]\), which is equal to
\begin{equation}
\label{eq:ZN02}
  \sigma^2(\beta)\sum_{k=0}^{N} \sigma^2(\beta)^k \sum_{I\subseteq \llb 1,N \rrb\,, \, \abs{I}=k} u(I) 
  = \sigma^2(\beta)\sum_{k=0}^{N} \bigl(\sigma^2(\beta) R_N\bigr)^k \mathrm{P}\bigl( \tau_k^{(N)} \leq N \bigr) \,.
\end{equation}
We now bound the probability appearing in the sum using Chernoff's bound: 
for any \(\lambda>0\),
\begin{equation}
\label{eq:Chernoff}
  \mathrm{P}\bigl( \tau_k^{(N)} \leq N \bigr) 
  \leq \e^{N\lambda } \, \mathrm{E}\bigl[ \exp\bigl( -\lambda T^{(N)}\bigr) \bigr]^{k} \,.
\end{equation}
To estimate the Laplace transform of \(T^{(N)}\), we use the following Tauberian theorem, from \cite[Thm.~3.9.1]{BGT89}.

\begin{lemma}
\label{lem:softthreshold}
  For a sequence \((u(n))_{n\in\N}\) of positive numbers, define the quantities
  \[
    R(m)\coloneqq \sum_{n=1}^m u(n)\qquad\text{and}\qquad \hat R(\lambda)\coloneqq \sum_{n=1}^{+\infty} \e^{-\lambda n} u(n) \,.
  \]
  If there exist constants \(a,b>0\) such that \(aR(m)= \log m + b + o(1)\) as \(m\to\infty\), then 
\begin{equation}\label{eq:Tauberian}
      a\hat R(\lambda) = \log\Bigl(\oneover{\lambda}\Bigr)+b-\gamma+\tilde o(1) \,,
\end{equation}
as \(\lambda\to 0\), where \(\gamma\) is the Euler--Mascheroni constant.
\end{lemma}

Let \(\eps>0\) be small and let us set \(\lambda\) such that 
\[
  \hat R(\lambda) \coloneqq \frac{1}{\sigma^2(\beta)} - \frac{\eps}{\pi} .
\] 
With this choice of \(\lambda\), we have 
\[
  \mathrm{E}\Bigl[ \exp\Bigl( -\lambda T^{(N)}\Bigr) \Bigr] 
  = \frac{1}{R_N} \sum_{n=1}^{N}\e^{- \lambda n} u(n) 
  \leq \frac{\hat R(\lambda)}{R_N} 
  = \frac{1 - \frac{\eps}{\pi} \sigma^2(\beta)}{\sigma^2(\beta) R_N}  \,.
\]
Using~\eqref{eq:Chernoff} and plugging this into~\eqref{eq:ZN02}, we get
\begin{equation}
\label{eq:sigmaboundvar}
  \sigma^2(\beta)\EE[Z_{N}^{\beta,\w}(0)^2] 
  \leq \e^{N\lambda} \sigma^2(\beta) \sum_{k=0}^{\infty} \Bigl(1 -\frac{\eps}{\pi} \sigma^2(\beta) \Bigr)^k 
  \leq \frac{\pi}{\eps}\, \e^{N\lambda}  \,.
\end{equation}

It remains to estimate $N\lambda$.
Note that the assumptions of \Cref{lem:softthreshold} are verified with $a=\pi$ and $b=\alpha$, see~\eqref{eq:Rbounds}.
By \eqref{eq:Tauberian}, we get \(\log \lambda = - \frac{\pi}{\sigma^2(\beta)} +\eps +\alpha - \gamma +o(1)\) as \(\beta\downarrow0\). 
Recalling also~\eqref{eq:etheta}, we get 
\begin{equation}
\label{eq:asymplambda}
  N \lambda 
  = \e^{\eps} \e^{\alpha -\gamma +o(1)} \,N \e^{- \frac{\pi}{\sigma^2(\beta)}}  
  \le \e^{2\eps} \, \e^{\th(N,\beta) -\gamma} 
  \qquad\text{as }\beta\downarrow 0\,, \ N\to\infty \,.
\end{equation}
Similarly, if we restrict $\beta \in (0,\beta_0)$, by \eqref{eq:Tauberian} we can bound \(\log \lambda \le - \frac{\pi}{\sigma^2(\beta)} +\eps +\alpha - \gamma +c\) for a suitable $c \in (0,\infty)$, therefore again by \eqref{eq:etheta} we get, for a suitable $C \in (0,\infty)$,
\begin{equation}
  \label{eq:boundlambda}
  N \lambda 
  \le \e^{\eps} \e^{\alpha -\gamma +c} \,N \e^{- \frac{\pi}{\sigma^2(\beta)}} 
  \le C \, \e^{\th(N,\beta)} 
  \qquad\text{for all } \beta \in (0,\beta_0), \ N\in \N \,.
\end{equation}

Plugging \eqref{eq:boundlambda} into \eqref{eq:sigmaboundvar}, we get the uniform bound \eqref{eq:ub2mom-unif}.
Similarly, plugging \eqref{eq:asymplambda} into \eqref{eq:sigmaboundvar}, we obtain
\[
  \sigma^2(\beta)\EE[Z_{N}^{\beta,\w}(0)^2]
  \le \frac{\pi}{\eps} \, \exp\bigl(\e^{2\eps} \, \e^{\th(N,\beta) -\gamma} \bigr)
  \le \exp\bigl(\e^{3\eps} \, \e^{\th(N,\beta) -\gamma} \bigr)  \,,
\]
where the second inequality holds eventually, if we assume that $\th(N,\beta) \to \infty$.
Since $\eps > 0$ is arbitrary, this completes the proof of \eqref{eq:ub2mom}, hence of \Cref{prop:second-moment-DP}.
\qed

\section{Control of the size-biased variance}
\label{sec:technicalII}

In this section, we control the size-biased variance, that is, we prove \Cref{prop:tiltedVar}.
\emph{Throughout this section,
we fix \(1 \le \eta < \th < \infty\) with $\th -\eta \ge 1$, for $N\in\N$ we set \(\tN=\lfloor\e^{-\eta}N\rfloor\) and we consider $\beta = \beta_N(\th)$ in the critical regime \eqref{eq:critical-regime}, or equivalently \eqref{eq:thetasharp}}.

First of all, recalling \eqref{def:X}, we write \(\tVV_f[X(f)]\) as the sum of size-biased covariances, which we split into two parts, called \textit{diagonal} and \textit{off-diagonal} terms:
\[
  \tVV_f[X(f)] 
  =  
  \sumtwo{m_1,m_2 =\tN+1}{\abs{m_1-m_2}\leq \tN}^{N-\tN} \tC_f\bigl[X_{m_1}(f),X_{m_2}(f)\bigr] + \sumtwo{m_1,m_2 =\tN+1}{\abs{m_1-m_2} > \tN}^{N-\tN} \tC_f\bigl[X_{m_1}(f),X_{m_2}(f)\bigr] \,.
\]
The proof reduces to proving the following estimates.

\begin{lemma}[Diagonal terms]
\label{lem:diag}
  There is a universal constant \(C>0\) such that
  \begin{equation}
  \label{eq:goal-diag}
\limsup_{N\to\infty}\sup_{f \in \cM_1^{\rm disc}(\sqrt{\tN})} \sumtwo{m_1,m_2 =\tN+1}{\abs{m_1-m_2}
\leq \tN}^{N-\tN} \tC_f\bigl[X_{m_1}(f),X_{m_2}(f)\bigr] 
\le C\, \Bigl( \frac{1}{\th-\eta} \e^{\th-\eta} \Bigr)^2 \eta\,.
\end{equation}
\end{lemma}

\begin{lemma}[Off-diagonal terms]
\label{lem:off-diag}
  There is a universal constant \(C>0\) such that, for any \(m_1,m_2 \in \llb \tN+1,N-\tN \rrb\) with \(m_2-m_1 > \tN\), we have
  \[
   \sup_{f \in \cM_1^{\rm disc}(\sqrt{\tN})} \tC_f\bigl[X_{m_1}(f),X_{m_2}(f)\bigr] \leq \frac{C}{(m_2)^2}\, \Bigl( \frac{1}{\th-\eta} \e^{\th-\eta} \Bigr)^2 \,.
  \]
\end{lemma}

\noindent
As a consequence, uniformly in \(f \in \cM_1^{\rm disc}\bigl(\sqrt{\tN}\bigr)\), the off-diagonal term satisfies
\[
  \sumtwo{m_1,m_2 =\tN+1}{\abs{m_1-m_2} > \tN}^{N-\tN}\tC_f\bigl[X_{m_1}(f),X_{m_2}(f)\bigr] 
  \leq 2C   \Bigl( \frac{1}{\th-\eta} \e^{\th-\eta} \Bigr)^2 \!\!\sum_{\tN < m_1 <m_2 < N}\!\! \frac{1}{(m_2)^2}  
  \leq  2C\Bigl( \frac{1}{\th-\eta} \e^{\th-\eta} \Bigr)^2  \eta \,,
\]
where for the last inequality we have used that there are at most \(m_2\) terms in the sum over~\(m_1\) so that the sum is bounded by \(\sum_{m_2 =\tN+1}^{N-1} \frac{1}{m_2} \leq \log (\frac{N}{\tN}) = \eta\).
Combining these two results concludes the proof of \Cref{prop:tiltedVar}.
\qed

\smallskip

It therefore remains to prove \Cref{lem:diag} and \Cref{lem:off-diag}.
We first deal with the off-diagonal term, that is, \Cref{lem:off-diag}, since it is slightly less technical than the diagonal term.

\subsection{\texorpdfstring{Off-diagonal terms: proof of \Cref{lem:off-diag}}{Off-diagonal terms: proof of lemma}}

Let \(m_2>m_1 > \tN\) with \(m_2-m_1 > \tN\).
Recalling \eqref{def:X} and expanding the covariance, we obtain
\[
  \tC_f \bigl[X_{m_1}(f),X_{m_2}(f)\bigr]
  = \sum_{A_1 \in\mathcal{I}_{m_1}, A_2 \in \mathcal{I}_{m_2}} q^{(f)}(A_1)q^{(f)}(A_2)\,\tC_f\bigl[\xi(A_1),\xi(A_2)\bigr] \,.
\]
Now, since the sets \(A_1,A_2\) are disjoint (because the strips \(\llb m_1, m_1+\tN \rrb\) and \(\llb m_2, m_2+\tN \rrb\) are disjoint), we get that 
\begin{align*}
  \tC_f\bigl[\xi(A_1),\xi(A_2)\bigr] 
  & = \tEE_f[\xi(A_1\cup A_2)] - \tEE_f[\xi(A_1)] \tEE_f[\xi(A_2)] \\
  & = \sigma^2(\beta)^{\abs{A_1}+\abs{A_2}} \bigl(q^{(f)}(A_1\cup A_2)-q^{(f)}(A_1) q^{(f)}(A_2)\bigr)\,,
\end{align*}
using also that
\begin{equation} 
\label{eq:tiltedE}
  \tEE_f[\xi(A)] 
  = \EE[Z_N^{\beta,\w}(f) \xi(A)] 
  = \sigma^2(\beta)^{\abs{A}} q^{(f)}(A) \,,
\end{equation}
recalling~\eqref{eq:chaosexpansion}.
Altogether, we have that \(\tC_f [X_{m_1}(f),X_{m_2}(f)]\) is equal to
\begin{equation*}
 \sum_{A_1 \in\mathcal{I}_{m_1}, A_2 \in \mathcal{I}_{m_2}}  \sigma^2(\beta)^{\abs{A_1}+\abs{A_2}}  q^{(f)}(A_1)q^{(f)}(A_2)\bigl(q^{(f)}(A_1\cup A_2)-q^{(f)}(A_1)q^{(f)}(A_2)\bigr) \, \,.
\end{equation*}

Denoting \((\ell_1,z_1)\) the last point of \(A_1\) and \((m_2,y_2)\) the first point of \(A_2\), we can write
\begin{equation*}
  q^{(f)}(A_1\cup A_2)= q^{(f)}(A_1)q_{m_2-\ell_1}(y_2-z_1)q(A_2')  \quad \text{ and }\quad 
  q^{(f)}(A_2) = q_{m_2}^{(f)}(y_2) q(A_2') \,, 
\end{equation*}
with \(A_2' = A_2- (m_2,y_2)\) the set \(A_2\) translated by its first point (with this point being removed). 
Hence, we have
\begin{multline*}
  q^{(f)}(A_1)q^{(f)}(A_2)\bigl(q^{(f)}(A_1\cup A_2)-q^{(f)}(A_1)q^{(f)}(A_2)\bigr) \\
  = q^{(f)}(A_1)^2 q_{m_2}^{(f)}(y_2)  \bigl(q_{m_2-\ell_1}(y_2-z_1)-q_{m_2}^{(f)}(y_2)\bigr) q(A_2')^2 \,. 
\end{multline*}
Summing over \(A_2'\), since \(\sum_{A_2'} \sigma^2(\beta)^{\abs{A_2}} q(A_2')^2 = \sigma^2(\beta) \, \cV_{\tN}\) by definition of~\eqref{def:VN}, we obtain that \(\tC_f[X_{m_1}(f),X_{m_2}(f)]\) is equal to
\begin{equation}
\label{eq:Cov-1}
  \begin{split}
  \sum_{A_1 \in \mathcal{I}_{m_1}} \sigma^2(\beta)^{\abs{A_1}}  q^{(f)}(A_1)^2  \Biggl( \sum_{ y_2 \in \Z^2 } q_{m_2}^{(f)}(y_2)
   \bigl(q_{m_2-\ell_1}(y_2-z_1)-q_{m_2}^{(f)}(y_2)\bigr)  \Biggr) \, \sigma^2(\beta) \cV_{\tN}&  \\
     = \sum_{A_1 \in \mathcal{I}_{m_1}} \sigma^2(\beta)^{\abs{A_1}} q^{(f)}(A_1)^2 
   \, \pigl( q_{2m_2-\ell_1}^{(f)}(z_1)- q_{2m_2}(f,f) \pigr) \,
   \sigma^2(\beta) \cV_{\tN} & \,,
\end{split}
\end{equation}
where we used Chapman--Kolmogorov (see also~\eqref{eq:collision}) and we also recall that \( (\ell_1,z_1)\) is the last point of \(A_1\). 

Note that \(\ell_1 \leq m_1 +\tN < m_2\), by definition of~\(\mathcal{I}_m\).
We now show that, uniformly over \((\ell_1,z_1) \in \N\times \Z^2\) such that \(\tN \leq m_1\leq \ell_1 \leq m_1 +\tN\leq m_2\), uniformly for \(f\in \cM_1^{\rm disc}(\sqrt{\tN})\), we have
\begin{equation}
\label{eq:sumtv}
  q_{2m_2-\ell_1}^{(f)}(z_1)- q_{2m_2}(f,f)
   \leq C \, \frac{m_1}{(m_2)^2}\, .
\end{equation}
First of all, since $q_{2n}(x)$ is maximized at $x=0$, we have that 
\[
  q_{2n}^{(f)}(z) 
  = \sum_{x\in \Z^2} f(x) q_{2n}(z-x) \leq \sum_{x\in \Z^2} f(x) q_{2n}(0) = q_{2n}(0) \,.
\]
Hence, we simply estimate
\[
  q_{2m_2-\ell_1}^{(f)}(z_1)- q_{2m_2}(f,f) 
  \le \pigl( q_{2m_2-\ell_1}(0)- q_{2m_2}(0) \pigr) + \pigl( q_{2m_2}(0) - q_{2m_2}(f,f)\pigr) \,.
\]
where we assume for simplicity that $\ell_1$ is even (the odd case is similar).
We now control both terms separately.
\begin{itemize}
\item For the first term, since $q_{2n}(0) = \frac{1}{\pi}  \frac{1}{n} + O(\frac{1}{n^2})$ as $n\to\infty$ by the local CLT (see e.g.\ \cite[Thm.~2.1.1]{LL10}, in particular~(2.5)), for $\ell_1 \le m_2$ (so that $2 m_2 - \ell_1 \ge m_2$) we have
\[
  q_{2m_2-\ell_1}(0)- q_{2m_2}(0) 
  = \frac{1}{\pi} \biggl(\frac{1}{m_2-\ell_1/2} - \frac{1}{m_2}\biggr) + O\biggl(\frac{1}{(m_2)^2}\biggr) 
 \leq c \, \frac{\ell_1}{m_2^2} 
 \leq 2c \, \frac{m_1}{m_2^2}  \,,
\]
where the last inequality comes from the fact that \(\ell_1 \leq m_1+\tN \leq 2m_1\) since \(m_1\geq \tN\).

\item For the second term, recalling the definition~\eqref{eq:collision} of \(q_{2m_2}(f,f)\), we start by writing, for \(f\in \cM_1^{\rm disc}(\sqrt{\tN})\)
\[
  \begin{split}
     \abs[\big]{q_{2m_2}(0)- q_{2m_2}(f,f)}
     & \leq \sum_{\abs{x},\abs{x'} \leq \sqrt{\tN}} f(x)f(x') \abs[\big]{q_{2m_2}(0)- q_{2m_2}(x-x')} \\
     & \leq \sup_{\abs{x}\leq 2 \sqrt{\tN}} \abs[\big]{q_{2m_2}(0)- q_{2m_2}(x)} \,.
  \end{split}
\]
Then, again by the local CLT, we get that
\[
  \abs[\big]{q_{2m_2}(0)- q_{2m_2}(x)}
  \leq \frac{1}{2\pi m_2} \abs[\big]{1- \e^{- \abs{x}^2/m_2}} + O\biggl(\frac{1}{(m_2)^2}\biggr)  
  \leq c \frac{1 + \abs{x}^2}{(m_2)^2} 
  \leq 2c\frac{m_1}{(m_2)^2 } \,,
\]
uniformly for \(\abs{x}\leq 2 \sqrt{\tN} \leq 2 \sqrt{m_1}\).
This concludes the proof of~\eqref{eq:sumtv}.
\end{itemize}

Plugging~\eqref{eq:sumtv} back into~\eqref{eq:Cov-1}, we get that 
\[
  \tC_f\bigl[X_{m_1}(f),X_{m_2}(f)\bigr]
  \le C \, \frac{m_1}{(m_2)^2} \, \sigma^2(\beta) \cV_{\tN} \, \sum_{A_1 \in \mathcal{I}_{m_1}} \sigma^2(\beta)^{\abs{A_1}} q^{(f)}(A_1)^2 \,.
\]
Let us notice that the last sum is exactly \(\tEE_f[X_{m_1}(f)]\), recalling~\eqref{def:X} together with
\eqref{eq:tiltedE}.
Thus, plugging the expression~\eqref{eq:tEXm} for \(\tEE_f[X_m(f)]\) (recall~\eqref{eq:varXm}), we get that
\begin{equation}
\label{eq:boundCov}
  \tC_f\bigl[X_{m_1}(f),X_{m_2}(f)\bigr]
  \le C \, \frac{m_1}{(m_2)^2} \, q_{2m_1}(f,f)\, \bigl(\sigma^2(\beta) \cV_{\tN}\bigr)^2 
  \le C' \, \frac{1}{(m_2)^2} \, \bigl(\sigma^2(\beta) \cV_{\tN}\bigr)^2\,,
\end{equation}
where we have also used~\eqref{eq:LLT-qff} to get that \(q_{2m_1}(f,f)\leq \frac{c'}{m_1}\) uniformly in \(f\in \cM_1^{\rm disc}(\sqrt{\tN})\).
This concludes the proof of \Cref{lem:off-diag} thanks to \Cref{lem:V}.
\qed

\begin{remark}
  Echoing \Cref{rem:boundsgood}, notice that the bound~\eqref{eq:boundCov} is again very general and does not rely on the specific value of \(\sigma^2(\beta)\cV_{\tN}\). 
  In fact, combining~\eqref{eq:boundCov} with~\eqref{eq:tEX} and~\Cref{lem:Gfg}, the contribution of the off-diagonal part of \(\tVV_f[X(f)]\) in the bound for \(\tPP_f(A_N^c(f))\) given in \Cref{lem:TChebyshev} is
  \[
    \frac{1}{ \tEE_{f}[X(f)]^2} \sumtwo{m_1,m_2=\tN+1}{\abs{m_1-m_2} > \tN}^{N-\tN} \tC_f\bigl[X_{m_1}(f),X_{m_2}(f)\bigr]  
    \leq \frac{C}{(\log \frac{N}{\tN})^2} \sum_{ \tN \leq m_1 < m_2 
    \leq N} \frac{1}{(m_2)^2} \leq \frac{C'}{\log \frac{N}{\tN}} \,.
  \]
  In particular, this is small \textit{irrespective of the value of \(\cV_{\tN}\)} (uniformly in \(f\in \cM_1^{\rm disc}(\sqrt{\tN})\)), provided that \(N/\tN\) is large.
\end{remark}

\subsection{\texorpdfstring{Diagonal term: proof of \Cref{lem:diag}}{Diagonal term: proof of lemma}}

Let us introduce intervals of length $2\tN$:
\begin{equation*}
  I_j \coloneqq \llb (j-1)\tN + 1, (j+1)\tN  \rrb
  \qquad \text{for } j=2, 3, \ldots \,,
\end{equation*}
in order to bound the diagonal term as follows:
\begin{equation} 
\label{eq:into}
  \sumtwo{m_1,m_2 =\tN+1}{\abs{m_1-m_2}
  \leq \tN}^{N-\tN} \tC_f\bigl[X_{m_1}(f),X_{m_2}(f)\bigr] \leq \sum_{j=2}^{\frac{N}{\tN} -1}\Biggl\{ \sum_{m_1,m_2 \in I_j} \tEE_f\bigl[X_{m_1}(f) X_{m_2}(f)\bigr] \Biggr\} \,.
\end{equation}
We will focus on the terms in brackets and show that there exists a universal constant \(C>0\) such that, for $N$ large enough, the following bound holds:
\begin{equation} 
\label{eq:goal-diag-j}
  \forall j \ge 2 \colon \qquad
  \sum_{m_1,m_2 \in I_j} \tEE_f\bigl[X_{m_1}(f) X_{m_2}(f)\bigr] \leq \eps_{j,N}(f) + \frac{C}{j}\, \Bigl( \frac{1}{\th-\eta} \e^{\th-\eta} \Bigr)^2 \,,
\end{equation}
where the terms \(\eps_{j,N}(f)\) satisfy
\begin{equation} 
\label{eq:ojN}
  \limsup_{N\to\infty} \sup_{f\in \cM_1^{\rm disc}(\sqrt{\tN})}\sum_{j=2}^{\frac{N}{\tN}-1}\eps_{j,N}(f) = 0 \,.
\end{equation}
These relations, when plugged into \eqref{eq:into}, yield \eqref{eq:goal-diag}, thus completing the proof of \Cref{lem:diag}.

\smallskip

It remains to prove \eqref{eq:goal-diag-j} and \eqref{eq:ojN}. We recall that 
\[
  \tEE_f[X_{m_1}(f) \, X_{m_2}(f)] = \EE[X_{m_1}(f) \, X_{m_2}(f) \, Z_N^{\beta,\w}(f)] \,.
\]
Using the representation~\eqref{def:X} of \(X_m(f)\) and the decomposition~\eqref{eq:chaosexpansion} of \(Z_N^{\beta,\w}(f)\), we get that
\begin{equation}
  \label{eq:diagonal-j}
  \begin{split}
  &\sum_{m_1,m_2 \in I_j} \tEE_f\bigl[X_{m_1}(f) X_{m_2}(f)\bigr]
  \\
  & = \!\!\!
  \sum_{A_1,A_2 \in \!\! \bigcup\limits_{m \in I_j} \!\!\!\mathcal{I}_m} \,
  \sum_{A_3\subseteq \llb 1,N\rrb \times \Z^2}
  \!\! \sigma(\beta)^{\abs{A_1}+\abs{A_2}+\abs{A_3}}
  \, q^{(f)}(A_1)q^{(f)}(A_2)q^{(f)}(A_3)
  \, \EE\bigl[\eta(A_1)\eta(A_2)\eta(A_3)\bigr]\,,
  \end{split}
\end{equation}
where we used the notation \(\sigma(\beta)^p=(\sigma(\beta)^2)^{\frac{p}{2}}\) and \(\eta_z = \sigma(\beta)^{-1} \xi_z\), in such a way that they have zero mean and unit variance, for \(A\subseteq \N\times \Z^2\), we set \(\eta(A) = \prod_{z\in A} \eta_z\).

Since the $\eta_{n,x}$'s are centered and independent, for \(\EE[\eta(A_1)\eta(A_2)\eta(A_3)]\) to be non-zero \emph{each point in \(A_1\cup A_2 \cup A_3\) must belong to at least two sets among \(A_1,A_2,A_3\)}. 
This means that, with $A_1 \triangle A_2 \coloneqq (A_1 \setminus A_2) \cup (A_2 \setminus A_1)$ the symmetric difference of $A_1$ and $A_2$, we must have
\begin{equation*}
  A_1 \triangle A_2 \subseteq A_3 \subseteq A_1 \cup A_2 \,, \qquad \text{that is,} \quad
  A_3 = (A_1 \triangle A_2) \cup D \quad \text{for some } D \subseteq A_1 \cap A_2 \,.
\end{equation*}
The terms with $A_3 \supsetneq A_1 \triangle A_2$, that is, with
$D \ne \varnothing$, correspond to \emph{triple intersections} (points which belong to all three sets $A_1, A_2, A_3$).

Let $\eps_{j,N}(f)$ denote the contribution to \eqref{eq:diagonal-j} coming from triple intersections, that is, from the restriction $A_3 \supsetneq A_1 \triangle A_2$. 
Since $I_j \subseteq \llb \tN+1,N\rrb$ for $2 \le j \le \frac{N}{\tN}-1$, we can bound
\begin{equation} 
\label{eq:triple-no}
  \begin{split}
  & \sum_{j=2}^{\frac{N}{\tN}-1} 
   \eps_{j,N}(f) \\
  & \ \le C \!\!\!\! \sumtwo{A_1,A_2,A_3\subseteq \llb \tN+1,N\rrb\times\Z^2}{A_1, A_2 \ne \varnothing, \ A_3  \supsetneq A_1\triangle A_2} \!\!\!\!
  \sigma(\beta)^{\abs{A_1}+\abs{A_2}+\abs{A_3}} \,
  q^{(f)}(A_1)q^{(f)}(A_2)q^{(f)}(A_3)\,
  \abs[\big]{\EE\bigl[\eta(A_1)\eta(A_2)\eta(A_3)\bigr]}  \,,
  \end{split}
\end{equation}
where the constant \(C\) accounts for the overlap of the intervals \(I_j\).
We now prove that \eqref{eq:ojN} holds, that is \emph{triple intersections give a negligible contribution}. 
To this purpose, we exploit \cite[Proposition~4.3]{CSZ20} to prove the following.

\begin{lemma}[No triple intersections]
\label{lem:no-triple}
  Fix \(0 \le \eta < \th < \infty\). 
  For $N\in\N$, set \(\tN=\lfloor\e^{-\eta}N\rfloor\) and consider $\beta = \beta_N(\th)$ in the critical regime \eqref{eq:critical-regime}. 
  Then \eqref{eq:ojN} holds.
\end{lemma}

\begin{proof}
  We will compare \eqref{eq:triple-no} with the \emph{centered third moment} \(\EE[(Z_{N-\tN}^{\beta,\w}(\hat f) - \EE[Z_{N-\tN}^{\beta,\w}(\hat f)])^3]\) of a partition function with time horizon $N-\tN$ and with initial condition $\hat f$ given by
  \begin{equation} 
  \label{eq:effe}
    \hat f(z) = \hat f_{\tN}(z) \coloneqq \sup_{f\in \cM_1^{\rm \disc}(\sqrt{\tN})} q_{\tN}^{(f)}(z) \,.
  \end{equation}
  Indeed, recalling \eqref{eq:chaosexpansion}, we can upper bound \(\EE[(Z_{N-\tN}^{\beta,\w}(\hat f) - \EE[Z_{N-\tN}^{\beta,\w}(\hat f)])^3]\) by
  \begin{equation} 
  \label{eq:terzo}
    \!\!\!\!\sumtwo{A_1',A_2',A_3'\subseteq \llb 1,N-\tN \rrb\times\Z^2}{A_1', A_2', A_3' \ne \varnothing} \!\!\!\! \sigma(\beta)^{\abs{A_1'}+\abs{A_2'}+\abs{A_3'}} q^{(\hat f)}(A_1') \, q^{(\hat f)}(A_2')\, q^{(\hat f)}(A_3')\,\abs[\big]{\EE\bigl[\eta(A_1')\eta(A_2')\eta(A_3')\bigr]}  \,.
  \end{equation}
  We can cast \eqref{eq:triple-no} in this form applying Chapman-Kolmogorov at time $\tN$, that is, writing
  \begin{equation*}
    q^{(f)}(A_i) 
    = \sum_{z\in\Z^2} q^{(f)}_{\tN}(z) \, q^{(z)}(A_i')
    \leq q^{(\hat f)} (A_i')
    \qquad \text{with } A_i' \coloneqq A_i - (\tN,0) \,, \quad i=1,2,3 \,.
  \end{equation*}
  This means that the right hand side of \eqref{eq:triple-no} can be bounded, uniformly over \(f\in \cM_1^{\disc}(\sqrt{\tN})\), by \emph{\eqref{eq:terzo} restricted to the terms with triple intersections} $A_3 \supsetneq A_1 \triangle A_2$.

  The latter contribution was studied in \cite{CSZ19-3rd} (see eq. (4.4) and Proposition~4.3, where it was denoted by $M_{0,1-\e^{-\eta}}^{N,\mathrm{T}}(\phi,1)$) and shown to vanish as $N\to\infty$, under the assumption that \(\hat f(x) \le \frac{C}{N} \, \phi\bigl(\frac{x}{\sqrt{N}}\bigr) \) for some continuous and compactly supported function \(\varphi\colon \R^2 \to [0,\infty)\). 
  Our function \(\hat f\) from~\eqref{eq:effe} satisfies \(\hat f_{\tN}(x) \le \frac{C}{N} \, \phi\bigl(\frac{x}{\sqrt{N}}\bigr) \) with $\phi(x) = \e^{-c\abs{x}^2}$ (for some $c = c_\eta>0$, by the local CLT): note that $\phi$ is bounded, but \emph{not} compactly supported. 
  However, the relevant property which is actually used in the proof of \cite[Proposition~4.3]{CSZ19-3rd} is that $\sum_{x\in\Z^2} \frac{1}{N} \phi\bigl(\frac{x}{\sqrt{N}}\bigr)$ remains \emph{bounded} as $N\to\infty$ (see the proof of Lemma~7.1 and the first display on page~427 in \cite{CSZ19-3rd}).
  This is clearly satisfied for our choice of \(\phi\), which completes the proof.
\end{proof}

 We may therefore restrict attention to the contribution in \eqref{eq:diagonal-j} with \(A_3=A_1\triangle A_2\), that is, with no triple intersections.
Our goal is to \emph{bound it by the last term in \eqref{eq:goal-diag-j}}.
We split this contribution into the two cases \(A_1=A_2\) and \(A_1\neq A_2\).

Since subsets \(A_i \in \mathcal{I}_m\) have width up to \(\tN\), see \eqref{def:Im}, and since \(I_j\) are intervals of width \(2\tN\), it follows that the sum over \(A_i \in \bigcup_{m \in I_j} \mathcal{I}_m\) can be enlarged to \(A_i \subseteq \tilde I_j \times \Z^2\) for the interval \(\tilde I_j\) of width~\(3\tN\) given by
\begin{equation} 
\label{eq:Itj}
  \tilde I_j \coloneqq I_j \cup I_{j+1} = \llb (j-1)\tN+1, (j+2)\tN \rrb \,.
\end{equation}
Set \(R \coloneqq R_{3\tN}\) for brevity.

\smallskip

\noindent
Case \(A_1=A_2\).
Here \(A_3=\varnothing\), hence the contribution is
\[
  \sum_{k=1}^{\log N}
  \sigma^2(\beta)^k
  \sum_{\substack{A\subseteq \tilde I_j\times\Z^2\\ \abs{A}=k}}
  q^{(f)}(A)^2 \,.
\]
We define
\begin{equation}
\label{eq:Emme=}
  \cM^{(f)}_{I,=}(k)
  \coloneqq
  \sum_{\substack{A\subseteq I\times\Z^2\\ \abs{A}=k}}
  \frac{q^{(f)}(A)^2}{(R_{\abs{I}})^k} \,.
\end{equation}
Therefore, the contribution of the terms with \(A_1=A_2\) is bounded by
\[
  \sum_{k=1}^{\log N}
  \bigl(\sigma^2(\beta) R\bigr)^k \,
  \cM^{(f)}_{\tilde I_j,=}(k) \,.
\]

\smallskip

\noindent
Case \(A_1\neq A_2\).
Here \(A_3=A_1\triangle A_2\neq \varnothing\), and
\[
  \sigma(\beta)^{\abs{A_1}+\abs{A_2}+\abs{A_1\triangle A_2}}
  =
  \sigma(\beta)^{2\abs{A_1\cup A_2}}
  =
  \sigma^2(\beta)^{\abs{A_1\cup A_2}} \,.
\]
Hence, decomposing according to \(k=\abs{A_1}\) and \(k'=\abs{A_2}\), we can bound the contribution of the terms with \(A_1\neq A_2\) as follows:
\begin{multline}
\label{eq:weby}
  \sum_{k,k'=1}^{\log N}
  \sum_{\substack{A_1, A_2 \subseteq \tilde I_j\times\Z^2\\\abs{A_1}=k,\,\abs{A_2}=k'\\ A_1\neq A_2}}
  \sigma^2(\beta)^{\abs{A_1\cup A_2}} q^{(f)}(A_1) \, q^{(f)}(A_2) \, q^{(f)}(A_1\triangle A_2) \\
  \le \sum_{k,k'=1}^{\log N} \bigl(\sigma^2(\beta) R_{3\tN} \bigr)^{k+k'} \, \widetilde{\cM}^{(f)}_{\tilde I_j}(k,k') \,,
\end{multline}
where for an interval \(I\) we define the normalized quantity
\begin{equation}
\label{eq:Emme}
  \widetilde{\cM}^{(f)}_{I}(k,k') 
  \coloneqq \sum_{\substack{A_1, A_2 \subseteq  I\times\Z^2\\\abs{A_1}=k,\,\abs{A_2}=k'\\ A_1\neq A_2}} \frac{q^{(f)}(A_1) \, q^{(f)}(A_2) \, q^{(f)}(A_1\triangle A_2)}{(R_{\abs{I}})^{\abs{A_1\cup A_2}}} \,.
\end{equation}
To derive \eqref{eq:weby}, we used that \(R_{\abs{I}}\ge 1\), \(\abs{A_1\cup A_2}\le k+k'\), and therefore
\[
  \sigma^2(\beta)^{\abs{A_1\cup A_2}}
  \le
  \bigl(\sigma^2(\beta)R_{\abs{I}}\bigr)^{k+k'} \, \frac{1}{(R_{\abs{I}})^{\abs{A_1\cup A_2}}} \,.
\]

Recall that we take \(\beta = \beta_N(\th)\) in the critical regime \eqref{eq:critical-regime}. 
Applying \eqref{eq:RN}, since \(\tN = \e^{-\eta} N\), we have \(R_{3\tN}/R_N = 1 + (\log 3 - \eta + o(1))/\log N\) which yields, as \(N\to\infty\),
\begin{equation} 
\label{eq:sigma3}
  \sigma^2(\beta) \, R_{3\tN}
  = \biggl( 1+\frac{\th+o(1)}{\log N}\biggr)  \biggl(1 + \frac{\log 3 - \eta + o(1)}{\log N}\biggr) = 1 + \frac{\th - \eta + \log 3 + o(1)}{\log N} \,,
\end{equation}
which is larger than \(1\) for large \(N\), since \(\th -\eta\ge 1\) by assumption.
We now use the following claim, that we prove below.

\begin{claim}
\label{claim:boundm}
  There is a constant \(C>0\) such that, for \(\tN\) large enough, we have
  \begin{align}
  \label{eq:our-claim=}
    \forall j \ge 2 \colon \qquad \sup_{k\geq 1} \sup_f \cM^{(f)}_{\tilde I_j,=}(k)
    & \leq  \frac{C}{j \, R_{3\tN}} \,, \\
  \label{eq:our-claim}
    \forall j \ge 2 \colon \qquad \sup_{k,k'\geq 1} \sup_f \widetilde{\cM}^{(f)}_{\tilde I_j}(k,k') 
    & \leq  \frac{C}{j \, R_{3\tN}^2} \,,
  \end{align}
  where the second supremum ranges over all mass functions on \(\Z^2\).
\end{claim}

With~\Cref{claim:boundm} at hand, since \(\sigma^2(\beta) R_{3\tN} \le 1 + \frac{\th - \eta + 2}{\log N} \le \exp( \frac{\th - \eta + 2}{\log N} )\) for large~\(N\), see~\eqref{eq:sigma3}, we obtain that, uniformly in \(f\in \mathcal{M}_1\pigl(\sqrt{\tN}\pigr)\), the contribution of the terms with \(A_1=A_2\) is bounded by
\[
  \frac{C}{j R_{3\tN}}
  \sum_{k=1}^{\log N}\e^{k \frac{\th-\eta +2}{\log N}}
  \le \frac{C'}{j R_{3\tN}}
  \frac{\e^{\th-\eta + 2}}{\frac{\th-\eta+2}{\log N}}
  \le \frac{C''}{j}\,\frac{\e^{\th-\eta}}{\th-\eta} \,,
\]
where in the first inequality we summed the geometric series and in the second inequality we exploited \eqref{eq:RN}. Since \(\th-\eta\ge 1\), this is bounded by
\[
  \frac{C''}{j}\Bigl(\frac{\e^{\th-\eta}}{\th-\eta}\Bigr)^2 \,.
\]

Similarly, uniformly in \(f\in \mathcal{M}_1\pigl(\sqrt{\tN}\pigr)\), the right hand side of \eqref{eq:weby} is bounded by
\[
  \frac{C}{j R_{3\tN}^2}\Biggl(\sum_{k=1}^{\log N}\e^{k \frac{\th-\eta +2}{\log N}}\Biggr)^2
  \le \frac{C'}{j R_{3\tN}^2}
  \Biggl(\frac{\e^{\th-\eta + 2}}{\frac{\th-\eta+2}{\log N}} \Biggr)^2 
  \le \frac{C''}{j}\Bigl(\frac{\e^{\th-\eta}}{\th-\eta} \Bigr)^2 \,,
\]
where in the first inequality we used \eqref{eq:our-claim}, in the second one we summed the geometric series, and in the third one we exploited \eqref{eq:RN} and bounded \(\th-\eta+2\geq \th-\eta\).
This yields our goal~\eqref{eq:goal-diag-j} and concludes the proof of \Cref{lem:diag}.
\qed

\subsubsection{\texorpdfstring{Proof of~\Cref{claim:boundm}}{Proof of claim}}

The quantity \(\widetilde{\cM}^{(f)}_{I}(k,k')\) is closely related to the contribution of pairwise intersections to the third moment of the partition function, similarly to the proof of \Cref{lem:no-triple}. 
We cannot apply results from \cite{CSZ20} out of the box, because of the \emph{local constraints} given by the fixed values of~\(k, k'\), but we can still adapt (and simplify) the arguments in \cite{CSZ20} to our context.

Set \(R \coloneqq R_{3\tN}\) for brevity.

We start with \eqref{eq:our-claim=}, that is, the contribution with \(A_1=A_2\).
Restricting \eqref{eq:Emme=} to \(I=\tilde I_j\), we have
\[
  \cM^{(f)}_{\tilde I_j,=}(k)
  = \sum_{\substack{A\subset \tilde I_j\times\Z^2\\ \abs{A}=k}} \frac{q^{(f)}(A)^2}{R^k} \,.
\]
Listing the elements of \(A = \{(m_i, z_i)\}_{1 \le i \le k}\) by increasing time, we have
\begin{equation} 
\label{eq:special-case}
  q^{(f)}(A)^2
  = q^{(f)}(m_1,z_1)^2\prod_{i=2}^{k} q(m_i-m_{i-1},z_i-z_{i-1})^2 \,.
\end{equation}
Renaming \((a, x) = (m_1, z_1)\) and \((b, y) = (m_k, z_k)\) the first and last point of~\(A\), and summing over the inner points \((m_i, z_i)\) for \(2 \le i \le k-1\), the product in \eqref{eq:special-case} yields the space-time convolution \(Q^{*(k-1)}(b-a,y-x)\), where we define the probability mass function on \(\N \times \Z^2\)
\begin{equation} 
\label{eq:Q}
   Q(m,z) 
   \coloneqq \frac{q(m,z)^2}{R}  \indic_{\{1\leq m \leq 3\tN\}} \,.
  \end{equation}
(By convention, \(Q^{*0}(m,z) = \indic_{m=0} \indic_{z=0}\).)
Therefore,
\begin{equation} 
\label{eq:cont=}
  \cM^{(f)}_{\tilde I_j,=}(k)
  = \frac{1}{R}\sum_{\substack{a\le b\in \tilde I_j\\x,y\in \Z^2}} q^{(f)}(a,x)^2 \, Q^{*(k-1)}(b-a,y-x) \,.
\end{equation}
Note that \(\sum_{y \in \Z^2} Q^{*(k-1)}(b-a,y-x) =
K^{*(k-1)}(b-a)\) where
\begin{equation}
\label{eq:Kappa}
  K(m) = \sum_{z\in \Z^2} Q(m,z) 
  = \frac{u(m)}{R} \indic_{\{1\leq m \leq 3\tN\}} \,,
\end{equation}
recalling also \eqref{def:u}. (Again, by convention, \(K^{*0}(m) = \indic_{m=0}\).)
We now use the following basic estimate: there exists a constant \(\hat c>1\) such that 
\begin{equation} 
\label{eq:est-q}
  q(m,z)
  \le\sup_{y\in\Z^2}q(m,y)\leq \hat c\,u(m) \,,\quad \text{hence by \eqref{eq:qf} also} \quad q^{(f)}(m,z) 
  \le \hat c \, u(m) \,,
\end{equation}
which we apply to \emph{one instance} of \(q^{(f)}(a,x)\) in \eqref{eq:cont=}. 
Since \(\sum_{x \in \Z^2} q^{(f)}(a,x) = 1\), we obtain
\[
  \sup_{k\ge 1} \, \sup_{f} \, \cM^{(f)}_{\tilde I_j,=}(k)
  \le \sup_{k\geq 1} \sum_{a \le b \in \tilde I_j}\hat{c} \, \frac{u(a)}{R} \, K^{*(k-1)}(b-a) 
  \le \sum_{a \in \tilde I_j} \hat{c} \, \frac{u(a)}{R},
\]
where we used the fact that \(\sum_{m\in\N} K^{*(k-1)}(m) = 1\). 
When we consider \(I = \tilde I_j\) from \eqref{eq:Itj}, we can bound \(u(a) \le \frac{c}{a} \le \frac{c}{(j-1)\tN}\) by \eqref{def:u} and the local CLT. Since \(\abs{\tilde I_j} = 3\tN\), we have shown that 
\[
  \sup_{k\ge 1} \, \sup_{f} \, \cM^{(f)}_{\tilde I_j,=}(k)
  \le \frac{3 c\,\hat{c}}{(j-1) \, R} \,,
\]
which yields our goal \eqref{eq:our-claim=} for \(j\ge 2\).

\smallskip

We next consider the contribution \(\widetilde{\cM}^{(f)}_{I}(k,k')\) from \eqref{eq:Emme}, which is the case \(A_1\neq A_2\).
We perform a change of variables: setting $A_3 \coloneqq A_1 \triangle A_2$,
we define the \emph{disjoint subsets}
\[
  C_{12} \coloneqq A_1 \cap A_2\,, \quad 
  C_{23} \coloneqq A_2 \cap A_3 = A_2 \setminus A_1\,, \quad
  C_{13} \coloneqq A_1\cap A_3 = A_1 \setminus A_2 \,,
\]
so that we can write
\begin{equation} \label{eq:AC}
    q^{(f)}(A_1) \, q^{(f)}(A_2) \, q^{(f)}(A_1\triangle A_2)
    = 
    q^{(f)}(C_{12}\sqcup C_{13}) \,
q^{(f)}(C_{12}\sqcup C_{23}) \,
q^{(f)}(C_{13}\sqcup C_{23}) \,.
\end{equation}
We will derive an explicit expression for this product
according to ``interaction diagrams'',
see \Cref{fig:diagram} for a graphical illustration.
To each space-time point \((m,z) \in A_1 \cup A_2 = C_{13} \cup C_{23} \cup C_{12}\) 
we associate a label $\mathsf{d} = ij \in \{12,23,13\}$ indicating the set to which it belongs, that is, $(m,z) \in C_{ij}$. This partitions the set $C_{13} \cup C_{23} \cup C_{12}$ into \emph{stretches} of points with common label $ij$, describing the ``interaction'' between the random walk configurations $A_i$ and $A_j$
(since $C_{ij} = A_i \cap A_j$).

\begin{figure}[htbp]
  \centering
  \captionsetup{width=.9\linewidth}
  \includegraphics[scale=1.1]{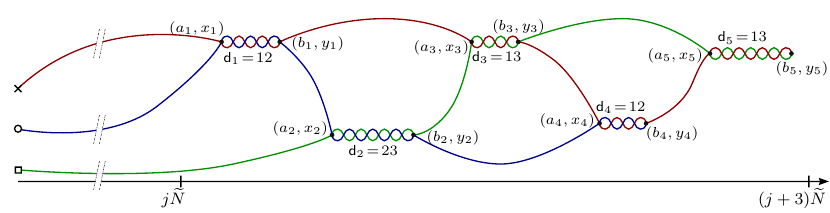}
  \caption{\footnotesize
  Illustration of an ``interaction diagram''. Pairwise interactions are grouped in stretches of space-time points belonging to the same set \(C_{ij}\), that is, with the same label \(\mathsf{d}=ij\).
  A labeled diagram corresponds to a collections of stretches, where each stretch has a label \(\mathsf{d}_p\), a size (cardinality)~\(k_p\) and ordered starting and ending points \((a_p,x_p)\), \((b_p,y_p)\).
  In the above diagram, there are \(\ell=5\) stretches.}
  \label{fig:diagram}
\end{figure}

To give a formal definition of the stretches, we order elements of \(C_{13}\cup C_{23}\cup C_{12}\) by increasing time, obtaining a list \((m_i,z_i)_{1\leq i \leq \abs{C_{13}\cup C_{23}\cup C_{12}}}\).
For the first stretch, we let \((a_1,x_1) = (m_1,z_1) = \min\{C_{13}\cup C_{23}\cup C_{12}\}\) be its first point, to which we associate label \(\mathsf{d}_1 \in \{12,23,13\}\) such that \((a_1,x_1) \in C_{\mathsf{d}_1}\), and we then add elements to the first stretch as long as they are in \(C_{\mathsf{d}_1}\).
The size of the first stretch is then \(k_1 = \sup\{k\colon (m_i,z_i) \in C_{\mathsf{d}_1}  \  \forall i\leq k\}\), and \((b_1,y_1) = (m_{k_1},z_{k_1})\) is its last point.
The first stretch is then \(\mathcal{S}_1 \coloneqq \{ (m_i,z_i), 1\leq i \leq k_1\}\).

We then proceed iteratively to define the subsequent stretches. 
If the stretches \(\mathcal{S}_1,\ldots, \mathcal{S}_{p-1}\) with respective sizes \(k_1,\ldots, k_{p-1}\) have been defined, the first element of the \(p\)-th stretch (if it exists) is then \((a_p,x_p) \coloneqq \min \{(C_{13}\cup C_{23}\cup C_{12}) \setminus (\mathcal{S}_1 \cup\cdots\cup \mathcal{S}_{p-1})\}\), which is in fact the element \((m_{k_1+\cdots+k_{p-1}+1},z_{k_1+\cdots+k_{p-1}+1})\).
The associated label is \(\mathsf{d}_p \in \{12,23,13\}\) such that \((a_p,x_p) \in C_{\mathsf{d}_p}\), and note that $\mathsf{d}_p \ne \mathsf{d}_{p-1}$. 
We define the size \(k_p\) and the last element \((b_p,y_p)\) of the stretch exactly as above.
The \(p\)-th stretch is then \(\mathcal{S}_p \coloneqq \{ (m_i,z_i)\colon k_1+\cdots+k_{p-1} +1 \leq i \leq k_1+\cdots+ k_p\}\).

\smallskip

We can now rewrite \eqref{eq:AC} as a product over stretches $\mathcal{S}_1, \mathcal{S}_2, \ldots, \mathcal{S}_\ell$, where \(\ell\geq 2\) the total number of stretches (note that there are at least two stretches, since we consider the contribution of $A_1 \ne A_2$). 
Recalling \eqref{eq:AC}, we can write (see again \Cref{fig:diagram} for an illustration) 
\[
\begin{split}
& q^{(f)}(A_1) \, q^{(f)}(A_2) \, 
q^{(f)}(A_1\triangle A_2) \\
 &\ = q^{(f)}(a_1,x_1)^2  
 \prod_{i=2}^{k_1} q(m_i-m_{i-1},z_i-z_{i-1})^2 \\
  &\ \quad \cdot q^{(f)}(a_2,x_2) \, q(a_2-b_1,x_2-y_1) \prod_{i=k_1+2}^{k_1+k_2} q(m_i-m_{i-1},z_i-z_{i-1})^2  \\
 &\ \quad  \cdot \prod_{p=3}^{\ell} q(a_p-b_{p-2},x_p-y_{p-2}) \, q(a_p-b_{p-1},x_{p}-y_{p-1})
 \!\! \prod_{i=k_1+ \cdots+k_{p-1}+2}^{k_1+ \cdots+k_p}
 \!\! q(m_i-m_{i-1},z_i-z_{i-1})^2 \,.
\end{split}
\]
We plug this expression into \eqref{eq:Emme} restricted to $A_1 \ne A_2$, which defines \(\cM^{(f)}_{I,\ne}(k,k')\). 

We now follow similar steps as for the case $A_1 = A_2$, see \eqref{eq:special-case} and the following lines. First we perform a partial sum inside each stretch: if we fix the starting and ending points $(a_p, x_p)$, $(b_p, y_p)$ of the $p$-th stretch, recalling the mass function $Q$ defined in \eqref{eq:Q}, the sum over internal space-time points $(m_i, z_i)$ for $k_1 + \ldots + k_{p-1} +1 \le i \le k_1 + \ldots + k_p - 1$ yields the space-time convolution
\(Q^{*(k_p-1)}(b_p-a_p,y_p-x_p)\) (recall the definition~\eqref{eq:Q}).
Altogether, we can rewrite \eqref{eq:Emme} as a sum over diagrams, encoded by the number \(\ell\ge 2\) of stretches and by the label $\mathsf{d}_p$, size $k_p$, starting and ending points $(a_p, x_p)$, $(b_p, y_p)$ of each stretch. The conditions $\abs{A_1} = \abs{C_{12}}+\abs{C_{13}}=k$, $\abs{A_2} = \abs{C_{12}}+\abs{C_{23}}=k'$ can be expressed as a constraint on the sizes $k_1, \ldots, k_\ell$ depending on the labels:
\begin{equation} \label{eq:constraint}
  \mathcal{C}_{k,k'}(\mathsf{d}_1,\ldots, \mathsf{d}_{\ell}) \coloneqq
  \Biggl\{ k_1 \ge 1 \,, \ldots \,, k_\ell \ge 1 \colon 
  \sumtwo{p = 1,\ldots, \ell:}{\mathsf{d}_p 
  \in \{12,13\}} \!\! k_p  = k \,, 
  \sumtwo{p = 1,\ldots, \ell:}{\mathsf{d}_p 
  \in \{12,23\}} \!\! k_p = k' \Biggr\} \,.
\end{equation}
Note that
\[
  \abs{A_1\cup A_2}
  = \abs{C_{13}}+\abs{C_{23}}+\abs{C_{12}}
  = \sum_{p=1}^{\ell} k_p \,.
\]
Hence the denominator in \eqref{eq:Emme} is exactly
\[
  \frac{1}{R^{\abs{A_1\cup A_2}}}
  = \frac{1}{R^{\sum_{p=1}^{\ell} k_p}} \,.
\]
Therefore, the contribution \(\widetilde{\cM}^{(f)}_{I}(k,k')\) is bounded above by
\begin{equation*}
\begin{split}
  \cM^{(f)}_{I,\ne}(k,k')
   \, \le & \,\sum_{\ell=2}^{\infty} \
  \sumtwo{\mathsf{d}_1,\ldots, \mathsf{d}_{\ell} \in \{12,23,13\}}{\mathsf{d}_p \ne \mathsf{d}_{p-1} \, \forall p=2,\ldots,\ell} \
   \sum_{k_1,\ldots,k_{\ell} \, \in \, \mathcal{C}_{k,k'}(\mathsf{d}_1,\ldots, \mathsf{d}_{\ell})} \
  \sum_{\substack{a_1\le b_1<\cdots<a_\ell\le b_\ell\in I\\x_1,y_1,\cdots,x_\ell,y_\ell\in \Z^2}}
  \\
  & \quad  \phantom{\cdot}
  \frac{q^{(f)}(a_1,x_1)^2}{R}\, Q^{*(k_1-1)}(b_1-a_1,y_1-x_1) 
  \\
  & \quad  
  \cdot \frac{q^{(f)}(a_2,x_2) \, q(a_2-b_1,x_2-y_1)}{R}\, Q^{*(k_2-1)}(b_2-a_2,y_2-x_2)
  \\
  & \quad \cdot\prod_{p=3}^\ell\frac{q(a_p-b_{p-2},x_p-y_{p-2}) \, q(a_p-b_{p-1},x_p-y_{p-1})}{R}\, Q^{*(k_p-1)}(b_p-a_p,y_p-x_p)\,.
\end{split}
\end{equation*}

\smallskip

We next apply the estimate \eqref{eq:est-q} to the kernels \(q(a_p-b_{p-2},x_p-y_{p-2})\), as well as to \(q^{(f)}(a_2,x_2)\) and to \emph{one instance} of \(q^{(f)}(a_1,x_1)\). 
This allows us to sum over all space variables iteratively, starting from $y_\ell,x_{\ell}, y_{\ell-1}, x_{\ell-1}$, $\ldots$ until $y_1, x_1$.
To this purpose, recalling \eqref{eq:Kappa},
we have $\sum_{y_p \in \Z^2 } Q^{*(k_p-1)}(b_p-a_p,y_p-x_p) = K^{*(k_p-1)}(b_p - a_p)$.
Using also $\sum_{x_p\in \Z^2} q(a_p-b_{p-1},x_p-y_{p-1}) =1$ and $\sum_{x_1 \in \Z^2} q^{(f)}(a_1, x_1) = 1$, we then obtain
\begin{multline}
\label{eq:Msimple}
  \cM^{(f)}_{I,\ne}(k,k')
   \, \le \
  \sum_{\ell=2}^{\infty} \
  (\hat{c})^\ell
  \sumtwo{\mathsf{d}_1,\ldots, \mathsf{d}_{\ell} \in \{12,23,13\}}{\mathsf{d}_p \ne \mathsf{d}_{p-1} \, \forall p=2,\ldots,\ell} \
   \sum_{k_1,\ldots,k_{\ell} \, \in \, \mathcal{C}_{k,k'}(\mathsf{d}_1,\ldots, \mathsf{d}_{\ell})} \
  \sum_{a_1\le b_1<\cdots<a_\ell\le b_\ell\in I} \\
  \frac{u(a_1)}{R}  \, K^{*(k_1-1)}(b_1-a_1) \,\frac{u(a_2)}{R} \, K^{*(k_2-1)} (b_2-a_2) \cdot\prod_{p=3}^\ell K(a_p-b_{p-2}) \, K^{*(k_p-1)}(b_p-a_p)\,.
\end{multline}
Note that there is no longer any dependence on~$f$.
  
We now sum over the time variables \(b_\ell,b_{\ell-1}\) which only appear in the last two stretches (these are free ends of the diagram, see \Cref{fig:diagram}). 
As a consequence, the kernels $K^{*k_\ell}$ and $ K^{*k_{\ell-1}}$ are erased from \eqref{eq:Msimple}, since \(\sum_{b_p} K^{*(k_p-1)}(b_p-a_p) \le \sum_{m \in \N_0} K^{*(k_p-1)}(m) =1\).
This is in fact a crucial step, since we can now sum over the stretch sizes \(k_{\ell},k_{\ell-1}\) \emph{removing the constraint \(\mathcal{C}_{k,k'}(\mathsf{d}_1,\ldots, \mathsf{d}_{\ell})\)}. 
Indeed, \emph{there is at most one pair $(k_\ell, k_{\ell-1})$ for which the constraint is fulfilled}, see \eqref{eq:constraint}, hence $\sum_{k_{\ell}} \sum_{k_{\ell-1}} \ind{(k_1, \ldots, k_\ell) \in \mathcal{C}_{k,k'}(\mathsf{d}_1,\ldots, \mathsf{d}_{\ell})} \le 1$ for any $\mathsf{d}_1,\ldots, \mathsf{d}_{\ell}$ and
$k_1,\ldots, k_{\ell-2}$.

If $\ell \ge 3$, we then sum freely 
over \(k_p\) for $1 \le p \le \ell-3$, replacing $K^{*(k_p-1)}(b_p-a_p)$ by $U(b_p - a_p)$ with
\begin{equation*}
   U(m) \coloneqq \sum_{k\ge 0} K^{*k}(m) \,.
\end{equation*}
Estimating \(u(a_1), u(a_2) \le \frac{c}{(j-1) \tN}\) uniformly for $a_1, a_2 \ge (j-1)\tN$, see \eqref{eq:Itj}, and bounding the number of labels $(\mathsf{d}_1,\ldots, \mathsf{d}_{\ell})$ by $3 \cdot 2^{\ell-1}$, we finally get
\[
  \cM^{(f)}_{\tilde I_j,\neq}(k,k')
  \le \frac{3^3 \, c^2}{2 \, (j-1)^2 \, R^2 }\,\sum_{\ell=2}^{\infty} (2 \hat c)^\ell\, J_{\tN,\ell}\,,
\]
where we define $J_{\tN,\ell}$ as follows:
for $\ell \in \{2,3\}$ we set
\begin{equation}
  \label{def:Jell23}
  J_{\tN,2} \coloneqq \sum_{a_1 < a_2 \,\in\, \tilde I_j}  \oneover{(3\tN)^2} \,, \qquad
  J_{\tN,3} \coloneqq \sum_{a_1\le b_1
  < a_2 < a_3 \,\in\, \tilde I_j}   \oneover{3\tN}
  \, U(b_1-a_1) \, \oneover{3\tN} \, K(a_3 - b_1) \,,
\end{equation}
while for $\ell \ge 4$ we set
\begin{equation}
  \label{def:Jell}
  \begin{split}
      J_{\tN,\ell} \coloneqq \sumtwo{a_1\le b_1<\cdots
  <a_{\ell-2} \le b_{\ell-2}}{<a_{\ell-1}< a_\ell \,\in\, \tilde I_j}   \oneover{3\tN}
  \, U(b_1-a_1) \, \oneover{3\tN} 
  \, U(b_2-a_2) 
  \prod_{p=3}^{\ell-2} K(a_p-b_{p-2}) \, &U(b_p-a_p) \\
  \cdot K(a_{\ell-1}-b_{\ell-3})
  \cdot & K(a_{\ell}-b_{\ell-2}) \,.
  \end{split}
\end{equation}
Since by assumption we have \(j \ge 2\), our goal \eqref{eq:our-claim} follows by the next claim, which shows that \(J_{\tN,\ell}\) decays super-exponentially in~\(\ell\), uniformly in~\(\tN\).
By \Cref{claim:stretches}, the series is uniformly bounded in \(\tN\), hence
\[
  \sup_{k,k'\ge 1} \, \sup_{f} \, \widetilde{\cM}^{(f)}_{\tilde I_j}(k,k')
  \le \frac{C}{j R^2} \,.
\]
Combining this with the estimate already proved for \(\cM^{(f)}_{\tilde I_j,=}(k)\), we obtain \eqref{eq:our-claim=} and \eqref{eq:our-claim}.
\qed

\begin{claim}
  \label{claim:stretches}
  For every \(\varepsilon>0\), there exists a constant \(C_{\varepsilon} <+\infty\) such that 
  \[
     \sup_{\tN\in\N} J_{\tN,\ell}  \leq C_{\varepsilon} \,\varepsilon^{\ell} \qquad \forall \ell\geq 2 \,.
  \]
\end{claim}

\begin{remark}[Renewal interpretation]
  The quantity $J_{\tN,\ell}$ from \eqref{def:Jell23}-\eqref{def:Jell} enjoys a probabilistic interpretation. Note that it only depends on the length of the interval \(\tilde I_j\), so we can replace \(\tilde I_j\) by \(\llb 1,3\tN \rrb\). Let \(\tau,\tau'\) be independent renewal processes, started from $\tau_0, \tau_0'$ uniformly sampled in \(\llb 1,3\tN \rrb\) and with step probability mass function \(K(m)\) from \eqref{eq:Kappa}. 
  Denote by \(\mathcal{L}_{3\tN}(\tau,\tau')\) the number of alternating stretches of \(\tau,\tau'\) in the interval \(\llb 1,3\tN \rrb\), then we can write 
  \[
    J_{\tN,\ell} 
    = \mathrm{P}\bigl( \mathcal{L}_{3\tN}(\tau,\tau')\ge\ell \,, \, A_{\tN,\ell} \bigr) \,.
  \]
  where the event $A_{\tN,\ell}$ is defined as follows:
  denoting by $\sigma_\ell = \sigma_\ell(\tau,\tau')$ the starting point of the $\ell$-th alternating stretch (if it exists), we set
  \[
    A_{\tN,\ell} 
    = \{\tau_0 < \tau'_0, \, 
    \tau\cap\tau'\cap \llb 1, \sigma_{\ell}-1 \rrb  =\varnothing \} \,.
  \]
  Indeed, the right hand side of \eqref{def:Jell23}-\eqref{def:Jell} gives precisely the probability that there are at least~$\ell$ alternating stretches of $\tau, \tau'$, with $\tau_0 < \tau_0'$ and no common point before $\sigma_{\ell}$ (see \Cref{fig:taustretch}).

  \begin{figure}[htbp]
    \centering
    \captionsetup{width=.9\linewidth}
    \includegraphics[scale=1.1]{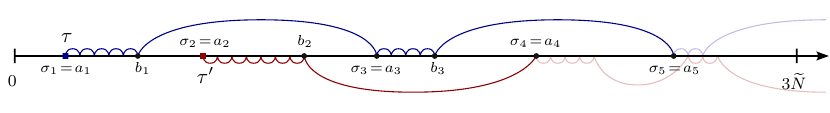}
    \caption{\footnotesize
    Illustration of the renewal interpretation of the formula~\eqref{def:Jell} for \(J_{\tN,\ell}\). 
    The first points of~\(\tau,\tau'\) are chosen uniformly in \(\llb 1,3\tN\rrb\) and the two renewals have inter-arrival distribution \(K(m)\) defined above.
    The stretches alternate between \(\tau\) and \(\tau'\) and have starting and ending point denoted by~\(a_p\) (or \(\sigma_p\)) and \(b_p\), in reference to the interaction diagrams (see \Cref{fig:diagram}).
    The number of alternating stretches is denoted \(\mathcal{L}_{3\tN}(\tau,\tau')\): in the above picture we have \(\mathcal{L}_{3\tN}(\tau,\tau') \geq \ell= 5\) and the two renewals \(\tau\cap \tau'\) do not intersect before the beginning \(\sigma_5\) of the \(5^{\rm th}\) stretch (but they might intersect afterwards).}
  \label{fig:taustretch}
  \end{figure}

  In particular, it follows by \Cref{claim:stretches} that, on the event $A_{\tN,\ell}$, the number of alternating stretches $\mathcal{L}_{3\tN}(\tau,\tau')$ has finite exponential moments:
  \(\mathrm{E}\bigl[\e^{\lambda \, \mathcal{L}_{3\tN}(\tau,\tau')}\indic_{A_{\tN,\ell}}\bigr] <+\infty\) for all $\lambda>0$ uniformly in $\tN$, which is equivalent to \(\sum_{\ell=2}^{\infty} \e^{\lambda \, \ell} \, J_{\tN,\ell} <+\infty\).
\end{remark}

\subsubsection{\texorpdfstring{Proof of \Cref{claim:stretches}}{Proof of stretches}}

The proof is similar to the argument in \cite[Section 5.3]{CSZ20}, but we include it because the present version is substantially simpler.

Let us fix \(\ell\geq 2\).
We can combine the definitions~\eqref{def:Jell23}-\eqref{def:Jell} in the single formula
\[
  J_{\tN,\ell} 
  = \sum_{ \substack{0<a_1\le b_1<\cdots<a_{\ell-2}\\  \le b_{\ell-2}<a_{\ell-1}< a_\ell\leq 3\tN} } \oneover{(3\tN)^{2}}  \prod_{i=1}^{\ell-2}U (b_i-a_i) \prod_{i=3}^{\ell} K(a_i-b_{i-2}) \,.
\]
Now, by \cite[Theorem~1.4]{CSZ19-Dickman} (and recalling~\eqref{def:u}-\eqref{eq:Rbounds} for controlling \(u(m)\) and \(R_{\tN}\)), we have that there is a constant \(C>1\) such that
\[
  U(m) 
  \leq C \, \frac{\log (3\tN)}{3\tN} \cdot G_0\Bigl( \frac{m+1}{3\tN} \Bigr)\,, \qquad K(m) 
  \leq \frac{C}{\log(3\tN)} \, \frac{1}{m}\,,
\]
where \(G_0(t) \coloneqq \int_0^{\infty} \frac{1}{\Gamma(s+1)}\, s \, t^{s-1} \,\e^{-\gamma s} \dd s\) for \(t\in (0,1]\) is the renewal function of the so-called Dickman subordinator.
(By taking \(\tN\) large we could make the constant \(C\) arbitrarily close to \(1\).)

Plugging this in the above formula (and noticing that all the terms ``\(\log (3\tN)\)'' cancel out)
\[
\begin{split}
  J_{\tN,\ell} 
  & \leq C^{2(\ell-2)}  \sum_{ \substack{0<a_1\le b_1<\cdots<a_{\ell-2}\\  \le b_{\ell-2}<a_{\ell-1}< a_\ell\leq 3\tN} }  \frac{1}{(3\tN)^2} \prod_{i=1}^{\ell-2} \frac{1}{3\tN} G_0\Bigl( \frac{b_i-a_i+1}{3\tN} \Bigr) \prod_{i=3}^{\ell}  \frac{1}{3\tN} \frac{3\tN}{a_i-b_{i-2}} \\
  & \leq (C')^{\ell}  \idotsint\limits_{\substack{0<s_1 < t_1<\cdots<s_{\ell-2}\\  \le t_{\ell-2}<s_{\ell-1}< s_\ell <1} }  \prod_{i=1}^{\ell-2} G_0(t_i-s_i) \prod_{i=3}^{\ell}  \frac{1}{s_i-t_{i-2}}  \de \mathbf{s} \de \mathbf{t}\,,
\end{split}
\]
the last inequality following from a Riemann sum bound.
With a change of variable \(u_i = t_i-s_i\) for \(1\leq i \leq \ell-2\) and \(v_i = s_i-t_{i-1}\) for \(1 \le i\leq \ell -1\) (with $t_0=0$), as well as \(v_{\ell} = s_{\ell}-s_{\ell-1}\), we get
\[
  J_{\tN,\ell} 
  \leq (C')^{\ell} \idotsint\limits_{\substack{ u_i\in (0,1),v_i\in(0,1) \\ u_1+\cdots +u_{\ell-2} + v_1+\cdots +v_{\ell} <1} }  \prod_{i=1}^{\ell-2} G_0(u_i) \cdot \prod_{i=3}^{\ell-1}  \frac{1}{v_i + u_{i-1} + v_{i-1}} \cdot
  \frac{1}{v_{\ell} + v_{\ell-1}}  \de \mathbf{u} \de \mathbf{v} \,.
\]
Then, bounding \(v_i + u_{i-1} + v_{i-1} \geq v_i + v_{i-1}\), introducing a multiplier \(\lambda>0\) and using that \( \prod_{i=1}^{\ell-2} \e^{\lambda u_i}  \leq \e^{\lambda}\), we get
\[
  J_{\tN,\ell}
  \leq (C')^{\ell}\, \e^{\lambda}\, \biggl( \int_0^1 G_0(u) \e^{-\lambda u} \de u \biggr)^{\ell-2}\, \int_{ (0,1)^{\ell-1}} \prod_{i=3}^{\ell}\frac{1}{v_i + v_{i-1}} \de \mathbf{v}  \,.
\]

By \cite[Lemma~5.2]{CSZ20}, there exists $c < \infty$ such that for all \(\lambda \ge 1\)
\[
  \int_0^1 G_0(u) \e^{-\lambda u} \de u 
  \le \frac{c}{2 + \log \lambda} \,,
\]
so we end up with 
\[
  J_{\tN,\ell} 
  \leq  (C')^{2}\, \e^{\lambda} \, \biggl( \frac{c C'}{2+\log \lambda} \biggr)^{\ell-2} \int_{ (0,1)^{\ell-1}} \prod_{i=3}^{\ell}\frac{1}{v_i + v_{i-1}} \de \mathbf{v} \,,
\]
and it remains to control the last integral.

For this, define \(\phi^{(0)}(v) \equiv 1\) and by iteration \(\phi^{(k)}(v) = \int_{(0,1)} \frac{1}{v+u} \phi^{(k-1)}(u) \de u\), so that the integral is equal to \(\int_0^1\phi^{(\ell-2)}(v) \dd v\).
We now show by induction that for any \(k\geq 1\) 
\[
  \forall v\in (0,1)\qquad \phi^{(k)}(v) 
  \le \frac{\pi^k}{\sqrt{v}} \,.
\]
The base case \(k=0\) is trivial, since \(\phi^{(0)}(v) = 1\le \frac{1}{\sqrt{v}}\) for \(v\in(0,1)\). 
For the inductive step, we assume that \(\phi^{(k-1)}(v) \leq \frac{\pi^{k-1}}{\sqrt{v}}\) for all \(v\in (0,1)\). Then, we have
\[
\begin{split}
  \phi^{(k)}(v) 
  = \int_0^1 \frac{1}{v+u} \phi^{(k-1)}(u) \de u 
  & \le \pi^{k-1} \int_0^1 \frac{1}{\sqrt{u}(v+u)} \de u  \\
  \text{(setting $u=t^2$)}\ \ & = 2\pi^{k-1} \int_0^1 \frac{1}{v+t^2}\de t
  =\frac{2\pi^{k-1}}{\sqrt{v}} \arctan\Bigl(\frac{1}{\sqrt{v}}\Bigr)
  \le\frac{\pi^k}{\sqrt{v}}.
\end{split}
\]
Therefore, we obtain that 
\[
  \int_{ (0,1)^{\ell-1}} \prod_{i=3}^{\ell}\frac{1}{v_i + v_{i-1}} \de \mathbf{v} 
  = \int_{0}^1  \phi^{(\ell-2)}(v) \de v 
  \le \int_{0}^1  \frac{\pi^{\ell-2}}{\sqrt{v}} \de v 
  \le 2\pi^{\ell-2}\,.
\]
Altogether, we  conclude that 
\[
  J_{\tN,\ell} 
  \le 2 \, (C')^{2}\, \e^{\lambda} \, \biggl( \frac{c C' \pi }{2+\log \lambda} \biggr)^{\ell-2}
\]
Taking \(\lambda = \lambda_{\varepsilon}\) large enough concludes the proof of \Cref{claim:stretches}.
\qed

\section{Proof of the other main results}
\label{sec:others}

In this section, we give the proofs of \Cref{th:mainSHF,cor:mainSHF,cor2:SHF,th:mainSHE,thm:freeenergy}. 
We start with \Cref{thm:freeenergy}, which is a direct consequence of \Cref{thm:quantitative}.

\begin{proof}[Proof of \Cref{thm:freeenergy}]
  
  The lower bound in \eqref{eq:free-energy-bounds} is proved in \Cref{app:lower-free-energy}: we follow~\cite[\S4]{BL17}, exploiting superadditivity and concentration of measure arguments for \(\log Z_N^{\beta,\w}\).
  We focus here on the upper bound in \eqref{eq:free-energy-bounds}, which we deduce from \Cref{thm:quantitative}.

  We first show that one may truncate $Z_{N}^{\beta,\w}$ at~$1$ in the definition \eqref{eq:def-free-energy} of the free energy and write
  \begin{equation}
  \label{eq:free-energy-equiv}
    \tf(\beta) 
    = \lim_{N\to\infty} \frac1N \,\EE\bigl[ \log  (Z_{N}^{\beta,\w}\wedge 1) \bigr]  \,.
  \end{equation}
  Indeed, since \(Z_N^{\beta,\w} = (Z_N^{\beta,\w} \wedge 1) (Z_N^{\beta,\w}\vee 1)\), it suffices to show that
  \[
    \lim_{N\to\infty} \frac1N \EE\bigl[\log  (Z_{N}^{\beta,\w}\vee 1) \bigr] 
    = 0 \,.
  \]
  But this follows by the inequalities \(1\leq Z_N^{\beta,\w}\vee 1 \leq 1+ Z_N^{\beta,\w}\), which yields
  \[
    0
    \leq \lim_{N\to\infty} \frac1N \, \EE\bigl[ \log  (Z_{N}^{\beta,\w}\vee 1) \bigr] 
    \leq \lim_{N\to\infty} \frac1N  \log \EE \bigl[ 1+ Z_{N}^{\beta,\w} \bigr]
    =0 \,,
  \]
  recalling also that \(\EE[Z_{N}^{\beta,\w}]=1\).

  Recalling \eqref{eq:etheta} and applying \eqref{eq:upper-main} to $Z_N^{\beta,\w} = Z_N^{\beta,\w}(0) = Z_N^{\beta,\w}(\ind{0})$, we get
  \[
    \EE\bigl[ Z_N^{\beta,\w} \wedge 1 \bigr] 
    \leq \frac{1}{c_2} \, \exp\Bigl(-c_2 \, \e^{\th(N,\beta)} \Bigr)
    = \frac{1}{c_2} \, \exp\Bigl(-c_2 \,\e^{\alpha + o(1)} N \, \e^{-\frac{\pi}{\sigma^2(\beta)}} \Bigr) \,.
  \]
  Applying relation \eqref{eq:free-energy-equiv} together with $\EE\bigl[ \log  (Z_{N}^{\beta,\w}\wedge 1) \bigr]  \leq \log  \EE\bigl[ Z_{N}^{\beta,\w}\wedge 1 \bigr]$ (by Jensen's inequality), we obtain the upper bound on the free energy in \eqref{eq:free-energy-bounds} for some \(c \in (0,c_2 \, \e^\alpha)\).
\end{proof}

We then prove \Cref{th:mainSHF} about the SHF, which follows from the corresponding result for directed polymers, \Cref{thm:quantitative}.

\begin{proof}[Proof of \Cref{th:mainSHF}]

  We first prove \eqref{eq:mainSHF}.
  Fix $\th\in\R$ and let $\beta_N = \beta_N(\th)$ for $N\in\N$ satisfy \eqref{eq:critical-regime}, or equivalently \eqref{eq:thetasharp}. 
  We are going to exploit \eqref{eq:upper-main} for $\beta = \beta_N$, so that \(\th(N,\beta) \to \th\) as $N\to\infty$, see \eqref{eq:Ntheta0}.
  Recall the convergence \eqref{eq:convergence-polymer} of the directed polymer partition function to the SHF, and note that the support of $f=\varphi^{(N)}$ is $\sqrt{N}$ times the size of the support of $\varphi$. 
  If we let $N\to\infty$ in \eqref{eq:upper-main} for $f = \varphi^{(N)}$, since \(\th(N, \beta_N) \to \th\) we obtain the upper bound in \eqref{eq:mainSHF} for $t=1$:
  \begin{equation} 
  \label{eq:mainSHFt=1}
    \sup_{\varphi \in \cM_1\bigl(\e^{c_0\, \e^{\th}}\,\bigr)}\EE\bigl[ \mathscr{Z}_1^\th(\varphi) \wedge 1 \bigr]
    \leq \frac{1}{c_2}\, \e^{-c_2 \, \e^{\th} } \,.
  \end{equation}
  For the lower bound, we note that
  \[
    \sup_{\varphi \in \cM_1\bigl(\e^{c_0\, \e^{\th}}\,\bigr)}\EE\bigl[ \mathscr{Z}_1^\th(\varphi) \wedge 1 \bigr]
    \ge  \sup_{\varphi \in \cM_1\bigl(\e^{c_0\, \e^{\th}}\,\bigr)}\frac{1}{2} \, \PP\Bigl( \mathscr{Z}_1^\th(\varphi) \ge \frac{1}{2}\Bigr)
    \ge C_1\,\e^{-c_1 \, \e^{\th} } \,,
  \]
  where the last inequality follows by the lower bound in \eqref{eq:probab-shf}, which we prove below.
  For general $t>0$, we use the scaling covariance property $\mathscr{Z}_1^{\th + \log t}(\varphi) \overset{d}{=} \mathscr{Z}_t^\th(\varphi_{\sqrt{t}})$, see the second relation in \eqref{eq:scaling-th-t}:
  applying \eqref{eq:mainSHFt=1} with $\th$ replaced by $\th + \log t$ yields
  \[
    \frac{1}{c_1} \, \e^{ -c_1 \, t \, \e^{\th} } 
    \leq \sup_{\varphi \in \cM_1\bigl(\e^{c_0\, t\, \e^{\th}}\,\bigr)} \EE\bigl[ \mathscr{Z}_t^\th(\varphi_{\sqrt{t}}) \wedge 1 \bigr]
    \leq \frac{1}{c_2}\, \e^{-c_2 \, t \, \e^{\th} } \,.
  \]
  We finally note that $\psi = \varphi_{\sqrt{t}} \in \cM_1\bigl(\e^{c_0\, t\, \e^{\th}} \sqrt{t}\,\bigr)$ for $\varphi \in \cM_1\bigl(\e^{c_0\, t\, \e^{\th}}\,\bigr)$, which proves \eqref{eq:mainSHF}.

\smallskip

  Next, we prove \eqref{eq:probab-shf}.
  We already remarked that the upper bound follows by the upper bound in \eqref{eq:mainSHF} and Markov's inequality $\PP(Z \ge \eps) \le (\eps \wedge 1)^{-1} \, \EE[Z \wedge 1]$, which yields $C_{2,\eps} = (c_2 \, \eps)^{-1}$. 
  For the lower bound, it suffices to consider the uniform density $\varphi = \mathcal{U}_{\sqrt{t}}$ on the ball of radius $\sqrt{t}$, see \eqref{eq:unif-dist}.
  The Paley--Zygmund inequality gives, for $Z = \mathscr{Z}_t^\th(\mathcal{U}_{\sqrt{t}})$ with $\EE[Z]=1$,
  \[
    \PP\bigl(Z \ge \eps \bigr) 
    \ge (1-\eps)^2 \, \frac{\EE[Z]^2}{\EE[Z^2]} 
    = \frac{(1-\eps)^2}{1 + \Var[Z]}  \,.
  \]
  Since $\Var[\mathscr{Z}_t^\th(\mathcal{U}_{\sqrt{t}})] \le c_3 \, \e^{c_3 \, t \, \e^\th}$ by \eqref{eq:ub2mom-unif} from \Cref{prop:second-moment-DP} and \eqref{eq:convergence-polymer}, we see that the lower bound in \eqref{eq:probab-shf} holds with $c_1 = c_3$ and $C_{1,\eps} = \frac{(1-\eps)^2}{1+c_3}$.
\end{proof}

We next deduce \Cref{cor:mainSHF} from \Cref{th:mainSHF} and \Cref{prop:second-moment-DP}.

\begin{proof}[Proof of \Cref{cor:mainSHF}]
  Recall the uniform density \eqref{eq:unif-dist}.
  Fix $c, \delta > 0$ (to be determined later) and set $\rho = \e^{c \, t \, \e^{\th}}$, $\eps \coloneqq t \, \e^{-\delta \, t \, \e^{\th}}$.
  By Markov's inequality, setting \(\eps' \coloneqq \frac{\eps}{\pi \, \rho^2 \, t} = \frac{1}{\pi} \e^{-(\delta+2c) \, t \, \e^\th} \le 1\),
  \[
    \PP\Bigl(\mathscr{Z}_t^\th\pigl(B(0,\rho \sqrt{t  \, })\pigr) > \eps \Bigr) 
    = \PP\Bigl(\mathscr{Z}_t^\th\pigl(\mathcal{U}_{\rho\sqrt{t  \, }}\pigr) > \eps' \Bigr) 
    \le \frac{\EE\bigl[\mathscr{Z}_t^\th(\mathcal{U}_{\rho \sqrt{t  \, }}) \wedge 1\bigr]}{\eps' \wedge 1}
    \le \frac{\frac{1}{c_2} \, \e^{-c_2 \, t \, \e^\th}}{\frac{1}{\pi} \, \e^{-(\delta+2c) \, t \, \e^\th}} \,,
  \]
  where we applied the upper bound from \eqref{eq:mainSHF} in \Cref{th:mainSHF} assuming $c \le c_0$. 
  The right hand side is $\frac{\pi}{c_2} \, \e^{-(c_2 - 2c - \delta) \, t \, \e^\th} \le \frac{\pi}{c_2} \, \e^{-\delta \, t \, \e^\th}$
  if we fix $c < \min\{c_0, \frac{c_2}{2} \}$ and $\delta \le \frac{1}{2}(c_2 - 2c)$.
  This proves the first line in \eqref{eq:SHF-largeball} provided we further take $\delta \le \frac{c_2}{\pi}$.

  For the second line, we exploit the following upper bound on the variance of the SHF:
  \begin{equation} 
  \label{eq:var-shf-unif}
    \Var\Bigl[ \mathscr{Z}_t^\th\pigl(\mathcal{U}_{\rho \sqrt{t  \, }}\pigr) \Bigr]
    \le c_3 \, \frac{\exp\bigl(c_3 \, t \, \e^{\th} \bigr)}{\rho^2} \,.
  \end{equation}
  This follows by \eqref{eq:ub2mom-unif} for $t=1$, recall \eqref{eq:convergence-polymer}, while the general case $t > 0$ can be deduced by the scaling properties of the SHF, see the second relation in \eqref{eq:scaling-th-t}.
  Let us set $\chi\coloneqq t \, \e^{\delta \, t \, \e^\th}$. 
  Note that $\chi'\coloneqq\frac{\chi}{\pi \, \rho^2 t} = \frac{1}{\pi} \e^{(\delta -2 c) t \e^{\th} } \leq \frac{1}{\pi}$, provided that \(\delta \leq 2c\).
  Since $\EE\bigl[\mathscr{Z}_t^\th(\mathcal{U}_{\rho \sqrt{t  \, }})\bigr] = 1$, Chebyshev's inequality yields
  \[
    \PP\pigl(\mathscr{Z}_t^\th\pigl(B(0,\rho \sqrt{t  \, })\pigr) \le \chi \pigr) 
    = \PP\pigl(\mathscr{Z}_t^\th\pigl(\mathcal{U}_{\rho \sqrt{t  \, }}\pigr) \le \chi' \pigr) 
    \le \frac{\Var\pigl[ \mathscr{Z}_t^\th(\mathcal{U}_{\rho \sqrt{t  \, }}) \pigr]}{(1-\frac{1}{\pi})^2}
    \le 3 c_3 \, \frac{\e^{c_3 \, t \, \e^\th}}{\rho^2} \,.
  \]
  Plugging in $\rho = \e^{c \, t \, \e^{\th}}$ with $c = c'' > \frac{c_3}{2}$, the right hand side is $3c_3 \, \e^{-(2c'' - c_3) \, t \, \e^\th} \le \frac{1}{\delta} \, \e^{-\delta \, t \, \e^\th}$ provided we fix $\delta \le \min\{2c'' - c_3, \frac{1}{3c_3}\}$.
\end{proof}

We finally prove \Cref{cor2:SHF,th:mainSHE}.

\begin{proof}[Proof of \Cref{cor2:SHF}]
  Recalling \eqref{eq:SHF-rescaled} we may write, by a change of variables,
  \begin{equation} 
  \label{eq:hatscaling}
    \hat{\mathscr{Z}}_t^{\th,c}(\varphi) 
    = \mathscr{Z}_t^{\th}(\hat \varphi)\qquad \text{with} \quad\hat\varphi(x) 
    = R^{-2} \, \varphi(R^{-1} x) \,, \quad R 
    = \e^{c \, t \, \e^\th} \sqrt{t} \,.
  \end{equation}
  By an approximation argument, we may assume that \(\varphi\) is bounded and belongs to \(\cM_1(1)\), see \eqref{eq:probab-density},
  hence \(\hat\varphi \in \cM_1(R)\). For $c < c_0$ we can then apply \Cref{th:mainSHF} to get $\EE[\mathscr{Z}_t^{\th}(\hat \varphi) \wedge 1] \to 0$ as $t \to \infty$ or $\th \to \infty$, hence $\mathscr{Z}_t^{\th}(\hat \varphi) \to 0$ in distribution, which proves the first line of \eqref{eq:SHF-fixed} with $c' = c_0$.

  Next, we observe that, since $\EE[\hat{\mathscr{Z}}_t^{\th,c}(\varphi)] = 1$, we can write by \eqref{eq:hatscaling}
  \[
    \EE\bigl[ (\hat{\mathscr{Z}}_t^{\th,c}(\varphi) - 1)^2 \bigr]
    = \Var\bigl[\hat{\mathscr{Z}}_t^{\th,c}(\varphi)\bigr] 
    = \Var\bigl[\mathscr{Z}_t^{\th}(\hat \varphi)\bigr]
    \le \pi^2 \, \norm{\varphi}_\infty^2 \, \Var\bigl[\mathscr{Z}_t^{\th}(\mathcal{U}_{R})\bigr] \,,
  \]
  where we simply bounded \(\hat \varphi(x) \leq \pi \, \norm{\varphi}_\infty \, \mathcal{U}_{R}(x)\), see \eqref{eq:unif-dist}. Applying \eqref{eq:var-shf-unif} we then get
  \[
    \EE\bigl[ (\hat{\mathscr{Z}}_t^{\th,c}(\varphi) - 1)^2 \bigr]
    \le c_3 \, \pi^2 \, \norm{\varphi}_\infty^2 \, \frac{\exp\bigl(c_3 \, t \, \e^{\th} \bigr)}{(\e^{c \, t \, \e^\th})^2} \,,
  \]
  hence the second line of \eqref{eq:SHF-fixed} holds with $c'' =  c_3$.
\end{proof}

\begin{proof} [Proof of \Cref{th:mainSHE}]
  The proof is similar to that of \Cref{cor2:SHF}. 
  We first treat $t=1$.
  The general case follows in the same way, using \eqref{eq:thNt} to replace \(\th(N,\beta)\) by \(\th(\lfloor Nt\rfloor,\beta)\). 
  Recalling \eqref{eq:SHE-exp-rescaled} and \eqref{eq:SHE-rescaled}, as well as \eqref{eq:Zf} and \eqref{eq:varphiresc}, we have the identity (in distribution, since the time-reversed environment has the same law as the original i.i.d. environment)
  \[
    \int_{\R^2} \varphi(x) \, \hat{u}_{N}^{\beta,c}(1,x) \, \dd x 
    \overset{\mathrm{d}}{=} Z_{N}^{\beta,\w}(\hat{f}_N) \qquad \text{with} \quad \hat{f}_N 
    = \varphi^{(\lfloor(\rho_{N}^{\beta,c})^2 \, N\rfloor)} \,, \quad \rho_{N}^{\beta,c} 
    = \e^{c \, \e^{\th(N,\beta)}} \,.
  \]
  For $\varphi \in \cM_1(1)$, see \eqref{eq:probab-density}, we have \(\hat{f}_N \in \cM_1^{\text{disc}}(\rho_{N}^{\beta,c} \sqrt{N}) = \cM_1^{\text{disc}}(\e^{c \, \e^{\th(N,\beta)}} \sqrt{N})\), see \eqref{eq:disc-density}. 
  Applying \Cref{thm:quantitative}, for $c < c_0$ we get $Z_{N}^{\beta,\w}(\hat{f}_N) \to 0$ in distribution as $N\to\infty$, which proves the first line of \eqref{eq:mainSHE} with $c' = c_0$. 
  We next bound, by \eqref{eq:ub2mom-unif},
  \[
    \EE \biggl[ \bigg( \int_{\R^2} \varphi(x) \, \hat{u}_{N}^{\beta,c}(1,x) - 1 \bigg)^2 \biggr] 
    = \Var\bigl[ Z_{N}^{\beta,\w}(\hat{f}_N) \bigr] 
    \le C \,\Var\bigl[Z_N^{\beta, \w}(\mathcal{U}^{\mathrm{disc}}_{\rho_{N}^{\beta,c}\sqrt{N}}) \bigr]  
    \leq C'\, \frac{\exp\bigl(c_3 \, \e^{\th(N,\beta)} \bigr)}{(\rho_{N}^{\beta,c})^2} \,,
  \]
which yields the second line of \eqref{eq:mainSHE} with $c'' = c_3$.
\end{proof}

\begin{appendix}

\section{The coarse-graining procedure}
\label{sec:coarse}

In this section, we prove \Cref{prop:coarse}, which we recall is a \textit{finite-volume criterion} showing that a small fractional moment at one time scale yields exponential decay of the partition function at larger time scales. 
We recall the definition \eqref{eq:disc-density} of the family \(\cM_1^{\disc}(r)\), where we replace for convenience \(\abs{\,\cdot\,}\) with \(\abs{\,\cdot\,}_\infty\).

\begin{proof}[Proof of \Cref{prop:coarse}]
  Recall that we assume that \(L\in \N\) and \(\beta\in (0,1)\) are such that
  \begin{equation}
  \label{eq:for-coarse-graining}
    \sup_{f \in \cM_1^{\disc}(\sqrt{L})} \EE\bigl[Z_{L}^{\beta,\w}(f)^{1/2}\bigr] 
    \le\frac{1}{113} \,.
  \end{equation}
  We will prove the result~\eqref{eq:NL} only when \(N\) is an integer multiple of \(L\), that is, \(N=mL\) for some \(m\in \N\); 
  the general case \(N\geq L\) follows easily by monotonicity in \(N\).
  We also assume for simplicity that \(\sqrt{L}\) is an integer.

  For any integers \(s<t\), for any probability measure \(\mu\) on \(\Z^2\) and any \(B\subset \Z^2\), let us introduce the notation
  \[
    Z_{s,t}^{\beta,\w} (\mu ; B)  
    \coloneqq \E_{\mu}\Bigl[ \exp\Bigl(\sum_{n=s+1}^{t} (\beta \w(n,S_n) -\lambda(\beta)) \Bigr) \indic_{\{S_{t} \in B\}} \Bigr] \,,
  \]
  which is the partition function of a polymer with initial distribution \(\mu\) at time \(s\) and constrained to end in \(B\) at time \(t\).
  We also denote \(Z_{s,t}^{\beta,\w} (x,y)\) when \(\mu\) is a Dirac mass at \(x\) and \(B\) is reduced to the set \(\{y\}\).

  Then, for a ``skeleton'' \(\mathcal{Y} = (y_i)_{i\geq 1} \in (\Z^2)^{\N}\), we define a \(\sqrt{L}\)-scale coarse-grained partition function starting from \(f\in \cM_1^{\disc}(\sqrt{L})\) and with skeleton~\(\mathcal{Y}\),  by setting for \(m\in\N\)
  \[
    Z_{mL}^{\beta,\w}(f ; \mathcal{Y}) 
    = \sum_{x_0 \in B(0)} f(x_0) \sum_{x_1 \in B(y_1)} \cdots \sum_{x_m \in B(y_m)}  \prod_{j=1}^m Z_{(j-1)L,jL}^{\beta,\w}(x_{j-1},x_j) \,,
  \]
  where for simplicity we denoted \(B(y) = B_{\sqrt{L}}(y)\coloneqq 2y\sqrt{L} +\llb -\sqrt{L},\sqrt{L}\llb^2\) the (half open) \(L^{\infty}\) ball centered at \(2 y\sqrt{L}\) of radius \(\sqrt{L}\), in such a way that \((B(y))_{y\in \Z^2}\) is a partition of \(\Z^2\).
  Note also that we have used the Markov property to write the partition function constrained to visit the \(x_i\)'s as a product of point-to-point partition functions.

  Using the standard inequality \((\sum_i z_i)^{1/2} \leq \sum_i z_i^{1/2}\) for non-negative \((z_i)\), we then get that for any \(m\in \N\),
  \begin{equation}
  \label{eq:uppercoarse}
    Z_{m L}^{\beta, \w}(f)^{1/2} 
    =\Bigl( \sum_{ (y_1,\ldots, y_m) \in (\Z^2)^m} Z_{m L}^{\beta,\w}(f;\mathcal{Y})  \Bigr)^{1/2} 
    \leq \sum_{(y_1,\ldots, y_m) \in (\Z^2)^m} Z_{m L}^{\beta,\w}(f;\mathcal{Y})^{1/2} \,,
  \end{equation}
  so that we are reduced to estimating a fractional moment along a skeleton \(\cY\).

  Now, let us stress that we have some coarse-grained product structure for \(Z_{k L}^{\beta,\w}(f;\mathcal{Y})\).
  Indeed, we can write 
  \[
  Z_{(k+1) L}^{\beta,\w}(f;\mathcal{Y}) 
  = Z_{k L}^{\beta,\w}(f;\mathcal{Y}) \; Z_{k L,(k+1) L}^{\beta,\w}\bigl(\mu_{k,f,\cY}^{\beta,\w} ; B(y_{k+1}) \bigr)\,,
  \]
  where \(\mu_{k,f,\cY}^{\beta,\w}\) is the ``\(\cY\)-skeleton polymer'' probability distribution, supported on \(B(y_{k})\), given by
  \[
    \mu_{k,f,\cY}^{\beta,\w}(x) 
    \coloneqq \frac{1}{Z_{kL}^{\beta,\w}(f;\mathcal{Y})} \sum_{x_0 \in B(0)} f(x_0) \sum_{x_1 \in B(y_1)} \cdots \sum_{x_{k-1} \in B(y_{k-1})}  \prod_{j=1}^k Z_{(j-1)L,jL}^{\beta,\w}(x_{j-1},x_j) \indic_{\{x_k=x\}} \,.
  \]
  Therefore, taking the conditional expectation with respect to \(\cF_{kL} = \sigma(\w(n,z) \colon n \le k L, z\in\Z^2)\) and using that \(\mu_{k,f,\cY}^{\beta,\w}\) is \(\cF_{kL}\)-measurable, we get that 
  \[
    \EE\Bigl[ Z_{(k+1)L}^{\beta,\w}(f;\mathcal{Y})^{1/2} \; \Big| \; \cF_{kL} \Bigr]
    \leq  Z_{kL}^{\beta,\w}(f;\mathcal{Y})^{1/2}\cdot  \sup_{\mu \colon  \mathrm{supp}(\mu) \subset B(y_k)} \EE\Bigl[ Z_{k L,(k+1)L}^{\beta,\w}\bigl(\mu ; B(y_{k+1}) \bigr)^{1/2} \Bigr] \,,
  \]
  where in the supremum \(\mu\) is a probability distribution.
  Therefore, if we define
  \[
    \mathcal{Q}(y) 
    \coloneqq \sup_{\mu \colon  \mathrm{supp}(\mu) \subset B(0)} \EE\Bigl[  Z_{0,L}^{\beta,\w}\bigl(\mu ;  B(y) \bigr)^{1/2} \Bigr] \,,
  \]
  then by translation invariance we get by iteration that 
  \[
  \sup_{f \in \cM_1^{\disc}(\sqrt{L})} \EE\Bigl[Z_{m L}^{\beta, w}(f;\mathcal{Y})^{1/2} \Bigr] 
  \leq  \prod_{i=1}^m \mathcal{Q}(y_i-y_{i-1}) \,.
  \]
  Therefore, plugged into~\eqref{eq:uppercoarse} we get that 
  \[
    \EE\bigl[Z_{m L}^{\beta, w}(f)^{1/2}\bigr] 
    \leq \sum_{(y_1,\ldots, y_m) \in (\Z^2)^m} \prod_{i=1}^m \mathcal{Q}(y_i-y_{i-1}) 
    = \Bigl( \sum_{y \in \Z^2} \mathcal{Q}(y) \Bigr)^m \,.
  \]
  It thus only remains to show that under~\eqref{eq:for-coarse-graining} we have that \(\sum_{y \in \Z^2} \mathcal{Q}(y) \leq \e^{-1}\).

  First of all, we always have that \(\mathcal{Q}(y) \leq \frac{1}{113}\), thanks to \eqref{eq:for-coarse-graining}.
  On the other hand, simply applying Jensen's inequality, we have that
  \[
  \EE\bigl[  Z_{0,L}^{\beta,\w}\bigl(\mu ;  B(y) \bigr)^{1/2} \bigr] 
  \leq \sqrt{\sum_{ x \in B(0)} \mu(x) \P_x\bigl(S_{L} \in B(y)\bigr)}
  \leq \sqrt{\P\Bigl(S_{L} \in 2y \sqrt{L} + \llb -2\sqrt{L},2\sqrt{L}\llb^2 \Bigr)} \, ,
  \]
  where we have widened the ball around \(2y\sqrt{L}\) by \(\sqrt{L}\) to account for the worst case scenario for the starting point \(x\in B(0) = \llb -\sqrt{L},\sqrt{L}\llb^2\).
  Now, notice that \((\pm S_n^{(1)}\pm S_n^{(2)})_{n\geq 0}\) are standard simple random walks in dimension 1, so that
  \[
  \P\Bigl(S_{L} \in 2y \sqrt{L} + \llb -2\sqrt{L}, 2\sqrt{L}\llb^2 \Bigr) 
  \leq \P \Bigl( \mathrm{SRW}_{L} \geq (2\abs{y}_1-4)\sqrt{L} \Bigr) 
  \leq \e^{- 2 (\abs{y}_1-2)^2} \,,
  \]
  where the last inequality is standard.

  Therefore, for any integer threshold \(K\ge 1\), we obtain that 
  \[
    \sum_{y \in \Z^2} \mathcal{Q}(y) 
    \leq \sum_{\abs{y}_1 \leq K} \frac{1}{113} + \sum_{\abs{y}_1 > K} \sqrt{\e^{- 2(\abs{y}_1-2)^2}}
    = (2 K^2 + 2K+1) \cdot \frac{1}{113} + \sum_{r >K} 4r \e^{- (r-2)^2}\,.
  \]
  Now, it turns out that for \(K=4\) the first term is~\(\frac{41}{113} \approx 0.3628\) and the second is \(\approx 0.0025\), with the sum of the two being smaller than \(0.366<\e^{-1}\).
  This concludes the proof.
\end{proof}

\section{Lower bound on the free energy}
\label{app:lower-free-energy}

We now prove the lower bound in \eqref{eq:free-energy-bounds} from \Cref{thm:freeenergy}, using the same strategy as in~\cite{BL17}.
The idea is to start from the superadditivity of \(\EE[\log Z_{N}^{\beta,\w}]\), which gives that 
\[
  \tf(\beta) 
  = \sup_{N\geq 1} \frac{1}{N} \EE\bigl[ \log Z_{N}^{\beta,\w} \bigr] \,,
\]
see e.g.\ \cite[Theorem 2.1]{Com17}.

Let \(N_c = N_c(\beta)\coloneqq\min\{N\ge 2:\sigma^2(\beta)R_{N}\ge 1\}\), that is, 
\begin{equation}
  \label{eq:N-critical}
  \sigma^2(\beta) R_{N_c-1} < 1 \le \sigma^2(\beta) R_{N_c} \,.
\end{equation}
Since \(\pi R_N=\log N+\alpha+o(1)\) as \(N \to \infty\), this implies \(\log N_c=\frac{\pi}{\sigma^2(\beta)} - \alpha + o(1)\),
as \(\beta \downarrow 0\), which yields 
\[
  \tf(\beta) 
  \geq \frac{1}{N_c} \EE\bigl[\log Z_{N_c}^{\beta,\w}\bigr] 
  \geq  C\, \e^{-\frac{\pi}{\sigma^2(\beta)}} \,\EE\bigl[\log Z_{N_c}^{\beta,\w}\bigr] \,.
\]
It remains to prove the following lemma, which provides a lower bound on \(\EE[\log Z_{N}^{\beta,\w}]\) near criticality.

\begin{lemma}
\label{lem:ElogZ}
  Fix \(\th\ge0\). There exists a constant \(C_\th>0\) such that, for all \(N\ge2\) and all \(\beta\in(0,1)\) satisfying \(\sigma^2(\beta)R_N\le \exp\bigl(\frac{\th}{\log N}\bigr)\), we have
  \[
    \EE[\log Z_{N}^{\beta,\w}] 
    \geq - C_\th (\log N)^{4} \,.
  \]
\end{lemma}

We then want to apply this lemma to \(N= N_c =N_c(\beta)\) defined above.
By the definition~\eqref{eq:N-critical}, we have \(\sigma^2(\beta) R_{N_c} -1\le \sigma^2(\beta) R_{N_c} - \sigma^2(\beta) R_{N_c-1}= \sigma^2(\beta) u(N_c)\), so that 
\[
\sigma^2(\beta) R_{N_c} \leq 1+\frac{\sigma^2(\beta)}{\pi N_c} \leq 1+\frac{1}{\log N_c} \le\exp\Bigl(\frac{1}{\log N_c}\Bigr) 
\]
for all sufficiently small \(\beta\).
We can thus apply \Cref{lem:ElogZ} with \(\th=1\), which concludes the proof of the lower bound in \Cref{thm:freeenergy}.

\begin{remark}
  The bound in \Cref{lem:ElogZ} is not expected to be optimal.
  In particular, we expect that \(\EE[\log Z_{N_c}^{\beta,\w}] \sim - \frac12 \log\log N_c\) in view of \cite{GT26}.
  Combined with superadditivity, this would give a lower bound of order \(-\log(\frac{1}{\sigma^2(\beta)})\exp(-\frac{\pi}{\sigma^2(\beta)})\) for the free energy.
  However, we still expect the upper bound in \Cref{thm:freeenergy} to provide the correct asymptotics.
\end{remark}

\begin{proof}[Proof of \Cref{lem:ElogZ}]
  The proof relies on concentration inequalities for the left tail of \(\log Z_{N}^{\beta}\).
  We use the following concentration inequality from \cite[Prop.~3.4]{CTT17}.
  \begin{proposition}
  \label{prop:concentration}
    Assume that the environment is bounded, that is, \(\abs{\w}\leq K\), and let \(f\) be a convex function. 
    Then, there exists some constant \(c>0\) such that for any \(a, M\) and \(t>0\), we have 
    \[
      \PP\bigl( f(\w) \geq a ; \abs{\nabla f} \leq M \bigr) \, \PP\bigl( f(\w) \leq a-t \bigr) 
      \leq 2 \e^{ - c \frac{t^2}{K^2 M^2}} \,.
    \]
  \end{proposition}

  We will apply this result to \(\log Z_{N}^{\beta,\w}\), which is a convex function in \(\w\), whose norm of the gradient is given by
  \[
    \abs[\big]{\nabla\log Z_{N}^{\beta,\w}}^2 
    = \sum_{n=1}^{N} \sum_{\abs{x} \leq n} \Bigl(\frac{\partial}{\partial \w_{n,x}} \log Z_{N}^{\beta,\w} \Bigr)^2 \,. 
  \]

  Our first lemma controls the first factor in \Cref{prop:concentration}.
  \begin{lemma}
  \label{lem:firstpart}
    Assume that \(\sigma^2(\beta) R_N \leq \e^{\frac{\th}{\log N}}\) for some \(\th \in \R\).
  Then, there is a constant \(C=C(\th)>0\) such that  
  \[
    \PP\Bigl(\log Z_{N}^{\beta,\w} \geq -1 ;\abs[\big]{\nabla\log Z_{N}^{\beta,\w}}^2 \le C (\log N)^3\Bigr) 
    \geq \frac{1}{C \log N} \,.
  \]
  \end{lemma}
  Then, applying \Cref{prop:concentration} with \(a=-1\) and \(M= \sqrt{C} (\log N)^{3/2}\), we get that for a bounded environment \(\abs{\w}\leq K\), 
  \[
    \PP\bigl(  \log Z_{N}^{\beta,\w} \leq -1 -t \bigr) 
    \leq 2 C \log N \, \e^{ - \frac{c}{C} \frac{t^2}{K^2 (\log N)^3}} \,.
  \]
We can in fact reduce to a bounded environment with a large constant \(K = (\log N)^{3/2}\): define \(\tilde \w_{n,x} = \w_{n,x} \indic_{\{\abs{\w_{n,x}} \leq (\log N)^{3/2}\}}\), and note that 
\[
\PP\bigl( \exists  n \in \llbracket 1, N \rrbracket , \abs{x} \leq n  \text{ such that } \tilde \w_{n,x} \neq \w_{n,x} \bigr) 
\leq 
9 N^3\, \PP\bigl(\abs{\w} \geq (\log N)^{3/2}\bigr) \leq 9 N^3\, \e^{- c_0 (\log N)^{3/2}} \,.
\]
Therefore,
\begin{equation*}
  \PP\bigl(  \log Z_{N}^{\beta,\w} \leq -1 -t \bigr)
   \leq \PP\bigl(  \log Z_{N}^{\beta,\tilde\w} \leq -1 -t \bigr) + 9 N^3\, \e^{- c_0 (\log N)^{3/2}} \,.
\end{equation*}
Note that setting \(\tilde \lambda(\beta) = \log \EE[e^{\beta \tilde \w}]\) and \(\tilde \sigma^2(\beta) = \e^{\tilde\lambda(2\beta)-2\tilde\lambda(\beta)}\), we can check that for \(\beta\in (0,1)\) we have \(\tilde \lambda(\beta) = \lambda(\beta) + O(\e^{-c (\log N)^{3/2}})\) and \(\tilde \sigma^2(\beta) = \sigma^2(\beta) +O(\e^{-c (\log N)^{3/2}})\).
In particular we can harmlessly replace \(\lambda(\beta)\) by \(\tilde \lambda(\beta)\) in \(Z_{N}^{\beta,\tilde\w}\), to which we can then apply \Cref{lem:firstpart}, say with \(1+\th\) instead of \(\th\).
Applying \Cref{prop:concentration} with \(K= (\log N)^{3/2}\), \(a=-1\), \(M = \sqrt{C} (\log N)^{3/2}\), we end up with
\begin{equation*}
  \PP\bigl(  \log Z_{N}^{\beta,\w} \leq -1 -t \bigr) \leq 2 C \log N \e^{ - \frac{c}{C} \frac{t^2}{(\log N)^6}} +  9 N^3\, \e^{- c_0 (\log N)^{3/2}}\,.
\end{equation*}

Then, using that \(-\EE[ \log Z_N^{\beta,\w}] \leq 1+ \int_1^{\infty} \PP( - \log Z_{N}^{\beta} \geq  u ) \dd u \), we can split the integral into two parts.
The first part is
\[
\int_1^{N^2} \PP( \log Z_{N}^{\beta} \leq -1 - u ) \dd u \leq C' (\log N)^{4} + 9 N^5 \e^{- c_0 (\log N)^{3/2}}
\]
where we have used the upper bound on the left tail of \(\log Z_N^{\beta,\w}\) found above.
For the remaining part, we use a very rough bound: writing \(\sum_{n=1}^N \beta \w_{n,S_n} \geq \beta N  \min_{n \in \llb 1, N\rrb, \abs{x} \leq N}\{\w_{n,x}\}\), we get that for \(u\geq 2\lambda(\beta) N \)
\[
\begin{split}
\PP( \log Z_{N}^{\beta} \leq - u )&  \leq \PP\Bigl( \beta N  \min_{n \in \llb 1, N\rrb, \abs{x} \leq N}\{\w_{n,x}\}  - \lambda(\beta) N  \leq -u \Bigr) \\
& \leq \PP \Bigl( \min_{n \in \llb 1, N\rrb, \abs{x} \leq N}\{\w_{n,x}\}   \leq - \frac{1}{2} \frac{u}{\beta N} \Bigr) \leq 9 N^3\e^{- c_0 \frac{u}{2\beta N}}\,.
\end{split}
\]
Thus, the second part of the integral \(\int_{N^2}^{\infty} \PP( \log Z_{N}^{\beta} \leq -1 - u ) \dd u \) is bounded by \(c \beta N^4 \e^{- c_0 N/2\beta} \), which is negligible compared to the first term.
This concludes the proof of \Cref{lem:ElogZ}.
\end{proof}

\begin{proof}[Proof of \Cref{lem:firstpart}]
  First of all, let us write
  \[
  \begin{split}
    \PP\bigl(  \log Z_{N}^{\beta,\w} \geq -1 &; \abs[\big]{\nabla \log Z_{N}^{\beta,\w}}^2 \le C (\log N)^3\bigr)\\
    &= \PP\bigl( Z_{N}^{\beta,\w} \geq \e^{-1} \bigr) - \PP\bigl(Z_{N}^{\beta,\w} \geq \e^{-1} ; \abs[\big]{\nabla\log Z_{N}^{\beta,\w}}^2 > C (\log N)^3\bigr).
  \end{split}
  \]
  For the first term, we use Paley--Zygmund inequality to get that 
  \[
    \PP(  Z_{N}^{\beta,\w} \geq \e^{-1}) 
    \geq (1-\e^{-1})^{2} \frac{1}{ \EE\bigl[ (Z_N^{\beta,\w})^2\bigr]} 
    \geq \frac{c}{ \log N} \,,
  \]
  where we have used that, in the critical window, \(\EE[ (Z_N^{\beta,\w})^2] \leq c \log N\) for some constant \(c=c(\th)\).
  For the second term, a direct calculation gives that 
  \[
    \abs[\big]{\nabla \log Z_{N}^{\beta,\w}}^2 
    = \frac{\beta^2}{(Z_{N}^{\beta,\w})^2} \E^{\otimes 2} \biggl[ \sum_{n=1}^N \indic_{\{S_n =\tilde S_n\}} \e^{\sum_{n=1}^N \beta (\w_{n,S_n} + \w_{n, \tilde S_n}) - 2\lambda(\beta) } \biggr] \,.
  \]
  Bounding \(\frac{1}{(Z_{N}^{\beta,\w})^2} \leq \e^2\) on the event \(Z_{N}^{\beta,\w} \geq \e^{-1}\), we get that, applying also Markov's inequality 
  \[
  \begin{split}
    \PP\Bigl(  Z_{N}^{\beta,\w} \geq \e^{-1} ; \abs[\big]{\nabla\log Z_{N}^{\beta,\w}}^2 &> C (\log N)^3\Bigr)\\
    & \leq \frac{\e^2}{C( \log N)^3} \E^{\otimes 2} \biggl[ \beta^2 \sum_{n=1}^N \indic_{\{S_n =\tilde S_n\}} \e^{ \lambda_2(\beta) \sum_{n=1}^N \indic_{\{S_n =\tilde S_n\}}} \biggr],
  \end{split}
  \]
  with \(\lambda_2(\beta) = \lambda(2\beta) -2\lambda(\beta)\).
  Then, we can use that, at criticality, we have the following bound, that we prove below
  \begin{claim}
  \label{claim:LexpL}
    Assume that \(\sigma^2(\beta) R_N  \leq \e^{ \frac{\th}{\log N}}\) for some \(\th\geq 0\).
    Then there is a constant \(C' = C' (\th)\) such that 
  \[
  \E^{\otimes 2} \Bigl[ \beta^2 \sum_{n=1}^N \indic_{\{S_n =\tilde S_n\}} \e^{ \lambda_2(\beta) \sum_{n=1}^N \indic_{\{S_n =\tilde S_n\}}} \Bigr]  \leq C'\, (\log N)^2 \,.
  \]
  \end{claim}
Altogether, this gives that 
\[
 \PP\Bigl(  \log Z_{N}^{\beta,\w} \geq -1 ; \abs{\nabla \log Z_{N}^{\beta,\w}}^2 \leq C (\log N)^3\Bigr) \geq  \frac{c}{\log N} - \frac{\e^2 C'}{C \log N} \geq \frac{c}{ 2 \log N} \,,
\]
provided that we had fixed \(C\) large enough.
\end{proof}

\begin{proof}[Proof of \Cref{claim:LexpL}]
  Recalling that \(\sigma^2(\beta) = \e^{\lambda_2(\beta)}-1\), we can perform the following chaos expansion:
  \begin{equation}
  \label{eq:LexpL}
  \begin{split}
    \E^{\otimes 2}& \Bigl[ \beta^2 \sum_{n=1}^N \indic_{\{S_n =\tilde S_n\}} (1+\sigma^2(\beta))^{  \sum_{n=1}^N \indic_{\{S_n =\tilde S_n\}}} \Bigr]  \\
    & = \beta^2 \sum_{k=0}^{\infty} \sigma^2(\beta)^k \sum_{1\leq n_1 <\cdots <n_k \leq N} \sum_{n=1}^N \P^{\otimes 2}\bigl(S_{n_i} =\tilde S_{n_i} \;\forall i \in \{1,\ldots, k\}, S_n =\tilde S_n \bigr) \,.
  \end{split}
  \end{equation}
  Now, we consider two contributions.
  First, if \(n \in \{n_1,\ldots, n_k\}\), this gives a term
  \[
    \beta^2\sum_{k=0}^{\infty} k \sigma^2(\beta)^k \sum_{1\leq n_1 <\cdots <n_k \leq N}  \prod_{i=1}^k u(n_i-n_{i-1}) \,,
  \]
  where~\(k\) is simply a combinatorial factor due to the choice of index \(i \in \{1,\ldots, k\}\) such that~\(n=n_i\).
  Second, if \(n \notin \{n_1,\ldots, n_k\}\), this gives a term
  \[
    \beta^2\sum_{k=0}^{\infty} (k+1) \sigma^2(\beta)^k \sum_{1\leq n_1 <\cdots <n_{k+1} \leq N}  \prod_{i=1}^{k+1} u(n_i-n_{i-1}) \,,
  \]
  where the combinatorial factor is due to the choice of interval \((n_{i-1},n_i)\) in which~\(n\) falls.
  Altogether, after a change of index for the second term, the left-hand side in \eqref{eq:LexpL} is equal to
  \[
    \beta^2\bigl(1 + \sigma^2(\beta)^{-1} \bigr) \sum_{k=0}^{\infty} k \sigma^2(\beta)^k \sum_{1\leq n_1 <\cdots <n_k \leq N}  \prod_{i=1}^k u(n_i-n_{i-1}) \,.
  \]
  Noticing that \(\beta^2(1 + \sigma^2(\beta)^{-1})\) is bounded by a constant, we therefore focus on sum.
  We use the following upper bound, see \cite[Lemma~5.4]{CSZ19-Dickman}: 
  there is a constant \(c>0\) such that, for every \(k \geq 1\)
  \[
    \frac{1}{(R_N)^k}\sum_{1\leq n_1 <\cdots <n_k \leq N}  \prod_{i=1}^k u(n_i-n_{i-1}) 
    \leq  \e^{- c \frac{k}{ \log N} \log^+(\frac{k}{\log N})} \,,
  \]
  where \(\log^+(x) = \log x\vee 0\).
  With this bound at hand, we get that
  \[
  \begin{split}
    \sum_{k=0}^{\infty} k \sigma^2(\beta)^k \sum_{1\leq n_1 <\cdots <n_k \leq N}  \prod_{i=1}^k u(n_i-n_{i-1}) & \leq \sum_{k=0}^{\infty} k \bigl(\sigma^2(\beta)R_N\bigr)^k  \e^{- c \frac{k}{ \log N} \log^+(\frac{k}{\log N})} \\
    & \leq (\log N)^2 \times \frac{1}{\log N}  \sum_{k=0}^{\infty} \frac{k}{\log N}\,  \e^{\th\frac{k}{ \log N}  -c\frac{k}{ \log N} \log^+(\frac{k}{\log N})} \,,
  \end{split}
  \]
  where we have also used that \(\sigma^2(\beta) R_N \leq \e^{\th/\log N}\).
  The last term converges to \(\int_0^{\infty} t\, \e^{ \th t - c t \log_+(t) } \dd t\) by a Riemann approximation, so in particular it is bounded by some constant (that depends on~\(\th\)).
  This concludes the proof.
\end{proof}

\end{appendix}

\bibliographystyle{alpha}
\bibliography{biblio.bib}

\end{document}